\documentclass[notitlepage]{report}
\usepackage[utf8]{inputenc}
\usepackage{hyperref}
\usepackage{amsmath, amsthm, amssymb, latexsym,epsfig,amsthm,enumerate,multicol,wasysym}
\usepackage{xcolor, comment}
\usepackage{bbm}
\usepackage{mathtools}
\usepackage{enumitem}
\usepackage[pagewise]{lineno}
\usepackage[symbol]{footmisc}

\usepackage{titling}

\pretitle{\begin{center}\huge}
\posttitle{\par\end{center}\vskip 0.5em}
\preauthor{\begin{center}\Large}
\postauthor{\end{center}}
\predate{\par\large\centering}
\postdate{\par}

\numberwithin{equation}{section}

\title{Optimal Agnostic Control of Unknown Linear Dynamics in a Bounded Parameter Range}
\author{J. Carruth, M. F. Eggl, C. Fefferman, C. W. Rowley}
\date{\today}

\newcommand{\Z}{\mathbbm{Z}}
\newcommand{\R}{\mathbb{R}}
\newcommand{\C}{\mathbb{C}}
\newcommand{\N}{\mathbb{N}}
\newcommand{\prob}{\text{Prob}}
\newcommand{\eprob}{\emph{Prob}}
\newcommand{\tru}{\text{TRUE}}
\newcommand{\etru}{\emph{TRUE}}
\newcommand{\mx}{\text{MAX}}
\newcommand{\emx}{\emph{MAX}}
\newcommand{\vxi}{\vec{\xi}}
\newcommand{\prior}{\text{Prior}}
\newcommand{\eprior}{\emph{Prior}}
\newcommand{\post}{\text{Post}}
\newcommand{\epost}{\emph{Post}}
\newcommand{\cE}{\mathcal{E}}
\newcommand{\E}{\text{E}}
\newcommand{\eE}{\emph{E}}
\newcommand{\veta}{\vec{\eta}}
\newcommand{\tnu}{t_\nu}
\newcommand{\qnu}{q_\nu}
\newcommand{\unu}{u_\nu}
\newcommand{\tame}{\text{TAME}}
\newcommand{\etame}{\emph{TAME}}
\newcommand{\nottame}{\text{NOT TAME}}
\newcommand{\enottame}{\emph{NOT TAME}}
\newcommand{\signu}{\sigma_{\tnu}}
\newcommand{\qsig}{q^{\sigma}}
\newcommand{\usig}{u^{\sigma}}
\newcommand{\zosig}{\zeta_1^\sigma}
\newcommand{\ztsig}{\zeta_2^\sigma}
\newcommand{\dqnu}{\Delta \qnu}
\newcommand{\cF}{\mathcal{F}}
\newcommand{\zo}{\zeta_1}
\newcommand{\zt}{\zeta_2}
\newcommand{\cZ}{\mathcal{Z}}
\newcommand{\qnusig}{\qsig_\nu}
\newcommand{\qnutsig}{q^{\tsig}_\nu}
\newcommand{\unusig}{\usig_\nu}
\newcommand{\unutsig}{u^{\tsig}_\nu}
\newcommand{\zonusig}{\zeta_{1,\nu}^\sigma}
\newcommand{\zonutsig}{\zeta_{1,\nu}^{\tsig}}
\newcommand{\ztnusig}{\zeta_{2,\nu}^\sigma}
\newcommand{\ztnutsig}{\zeta_{2,\nu}^{\tsig}}
\newcommand{\zonu}{\zeta_{1,\nu}}
\newcommand{\ztnu}{\zeta_{2,\nu}}
\newcommand{\barzosig}{\bar{\zeta}_1^\sigma}
\newcommand{\barztsig}{\bar{\zeta}_2^\sigma}
\newcommand{\barq}{\bar{q}}
\newcommand{\barN}{{\bar{N}}}
\newcommand{\baru}{\bar{u}}
\newcommand{\terma}{\text{Term}\;\alpha}
\newcommand{\termb}{\text{Term}\;\beta}
\newcommand{\termc}{\text{Term}\;\gamma}
\newcommand{\err}{\text{ERR}}
\newcommand{\eerr}{\emph{ERR}}
\newcommand{\ecost}{\text{ECOST}}
\newcommand{\erroro}{\text{ERROR 1}}
\newcommand{\errort}{\text{ERROR 2}}
\newcommand{\errorth}{\text{ERROR 3}}
\newcommand{\cR}{\mathcal{R}}
\newcommand{\compE}{\prescript{c}{}{\cE}}
\newcommand{\cost}{\textsc{Cost}}
\newcommand{\osc}{\textsc{Osc}}
\newcommand{\qcsig}{q^\sigma_C}
\newcommand{\qdsig}{q^\sigma_D}
\newcommand{\zocsig}{\zeta_{1,C}^\sigma}
\newcommand{\ztcsig}{\zeta_{2,C}^\sigma}
\newcommand{\zodsig}{\zeta_{1,D}^\sigma}
\newcommand{\ztdsig}{\zeta_{2,D}^\sigma}
\newcommand{\bad}{\textsc{bad}}
\newcommand{\term}{\textsc{Term}}
\newcommand{\hatt}{\hat{t}}
\newcommand{\hatN}{{\hat{N}}}
\newcommand{\hatsigma}{\hat{\sigma}}
\newcommand{\qhatsig}{q^{\hat{\sigma}}}
\newcommand{\uhatsig}{u^{\hat{\sigma}}}
\newcommand{\hatq}{\hat{q}}
\newcommand{\hatu}{\hat{u}}
\newcommand{\disaster}{\textsc{disaster}}
\newcommand{\ndisaster}{\textsc{non-disaster}}
\newcommand{\ok}{\text{OK}}
\newcommand{\barzonu}{\bar{\zeta}_{1,\nu}}
\newcommand{\barztnu}{\bar{\zeta}_{2,\nu}}
\newcommand{\op}{\text{opt}}
\newcommand{\eop}{\emph{opt}}
\newcommand{\tsig}{\tilde{\sigma}}
\newcommand{\zomu}{\zeta_{1,\mu}}
\newcommand{\ztmu}{\zeta_{2,\mu}}
\newcommand{\cU}{\mathcal{U}}
\newcommand{\ctg}{\text{CTG}}
\newcommand{\ectg}{\emph{CTG}}
\newcommand{\discrep}{\textsc{Discrep}}
\newcommand{\casei}{\textsc{Case}\;\text{I}}
\newcommand{\caseii}{\textsc{Case}\;\text{II}}
\newcommand{\error}{\text{ERROR}}
\newcommand{\cX}{\mathcal{X}}
\newcommand{\sigveta}{\sigma_{\veta}}
\newcommand{\goodflips}{\textsc{goodflips}}
\newcommand{\badflips}{\textsc{badflips}}
\newcommand{\vsigma}{\vec{\sigma}}
\newcommand{\excost}
{\textsc{ECost}}
\newcommand{\regret}{\textsc{Regret}}
\newcommand{\mr}{\text{MR}}
\newcommand{\good}{\textsc{good}}
\newcommand{\vexcost}{\overrightarrow{\excost}}
\newcommand{\cK}{\mathcal{K}}
\newcommand{\cV}{\mathcal{V}}
\newcommand{\mix}{\text{MIX}}
\newcommand\extrafootertext[1]{%
    \bgroup
    \renewcommand\thefootnote{\fnsymbol{footnote}}%
    \renewcommand\thempfootnote{\fnsymbol{mpfootnote}}%
    \footnotetext[0]{#1}%
    \egroup
}
\newcommand{\bayes}{\text{Bayes}}
\newcommand{\hr}{\text{HReg}}

\newcommand{\mreg}{\text{MReg}}
\newcommand{\ar}{\text{AReg}}
\newcommand{\vvsigma}{\vec{\vec{\sigma}}}

\newtheorem{lem}{Lemma}
\newtheorem{thm}{Theorem}
\newtheorem{rmk}{Remark}

\newtheorem{cor}{Corollary}

\begin{document}

\maketitle

\begin{abstract}
Here and in the follow-on paper \cite{almostoptimal2023}, we consider a simple control problem in which the underlying dynamics depend on a parameter $a$ that is unknown and must be learned. In this paper, we assume that
$a$ is bounded, i.e., that $|a| \le a_\mx$, and we study two variants of the control problem. In the first variant, Bayesian control, we are given a prior probability distribution for $a$ and we seek a strategy that minimizes the expected value of a given cost function. Assuming that we can solve a certain PDE (the Hamilton-Jacobi-Bellman equation), we produce optimal strategies for Bayesian control. In the second variant, agnostic control, we assume nothing about $a$ and we seek a strategy that minimizes a quantity called the regret. We produce a prior probability distribution $d\prior(a)$ supported on a finite subset of $[-a_\mx,a_\mx]$ so that the agnostic control problem reduces to the Bayesian control problem for the prior $d\prior(a)$.
\end{abstract}

\tableofcontents

\section{Introduction}
Here and in \cite{carruth2022controlling,fefferman2021optimal}, we explore a new flavor of adaptive control theory, which we call ``agnostic control.''\extrafootertext{This work was supported by AFOSR grant FA9550-19-1-0005 and by the Joachim Herz Foundation.} Our introduction borrows heavily from that of \cite{almostoptimal2023}.

Many works in adaptive control theory attempt to control a system whose underlying dynamics are initially unknown and must be learned from observation. The goal is then to bound $\regret$, a quantity defined by comparing our expected cost with that incurred by an opponent who knows the underlying dynamics. Typically one tries to achieve a regret whose order of magnitude is as small as possible after a long time. Adaptive control theory has extensive practical applications; see, e.g., \cite{bertsekas2012dynamic, cesa2006prediction, hazancontrol, powell2007approximate}.

In some applications, we don't have the luxury of waiting for a long time. This is the case, e.g., for a pilot attempting to land an airplane following the sudden loss of a wing, as in \cite{Brazy:2009}. Our goal, here and in \cite{almostoptimal2023}, is to achieve the absolute minimum possible regret over a fixed, finite time horizon. This poses formidable mathematical challenges, even for simple model systems.

We will study a one-dimensional, linear model system whose dynamics depend on a single unknown parameter $a$. When $a$ is large positive, the system is highly unstable. (There is no ``stabilizing gain'' for all $a$.) Here, we suppose that the unknown $a$ is confined to a known interval $[-a_\mx, a_\mx]$ and we don't assume that we are given a Bayesian prior probability distribution for it. In \cite{almostoptimal2023}, we extend our results to deal with the case in which $a$ may be any real number.

Modulo an arbitrarily small increase in regret, we reduce the problem, here and in \cite{almostoptimal2023}, to a Bayesian variant in which the unknown $a$ is confined to a finite set and governed by a prior probability distribution.

For the Bayesian problem, our task is to find a strategy that minimizes the expected cost. This leads naturally to a PDE, the Bellman equation. We prove here that the optimal strategy for Bayesian control is indeed given in terms of the solution of the Bellman equation, and that any strategy significantly different from that optimum incurs a significantly higher cost. We proceed modulo assumptions about existence and regularity of the relevant PDE solutions, for which we lack rigorous proofs. (However, we have obtained numerical solutions, which seem to behave as expected.)

Let us now explain the above in more detail. 

\subsection*{The Model System}

Our system consists of a particle moving in one dimension, influenced by our control and buffeted by noise. The position of our particle at time $t$ is denoted by $q(t) \in \R$. At each time $t$, we may specify a ``control'' $u(t) \in \R$, determined by history up to time $t$, i.e., by $(q(s))_{s \in [0,t]}$. A ``strategy'' (aka ``policy'') is a rule for specifying $u(t)$ in terms of $(q(s))_{s \in [0,t]}$ for each $t$. We write $\sigma, \sigma', \sigma^*, \text{etc.}$ to denote strategies. The noise is provided by a standard Brownian motion $(W(t))_{t\ge 0}$.

The particle moves according to the stochastic ODE
\begin{equation}\label{eq: nintro 1}
dq(t) = \big(aq(t) + u(t)\big)dt + dW(t), \qquad q(0) = q_0,
\end{equation}
where $a$ and $q_0$ are real parameters. Due to the noise in \eqref{eq: nintro 1}, $q(t)$ and $u(t)$ are random variables; these random variables depend on our strategy $\sigma$, and we often write $q^\sigma(t)$, $u^\sigma(t)$ to make that dependence explicit.

Over a time horizon $T>0$, we incur a $\cost$, given\footnote[2]{By rescaling, we can consider seemingly different cost functions of the form $\int_0^T(q^2+\lambda u^2)$ for $\lambda >0$.} by
\begin{equation}\label{eq: nintro 2}
    \cost = \int_0^T \big\{ (q(t))^2 + (u(t))^2\big\} dt.
\end{equation}
This quantity is a random variable determined by $a, q_0,T$ and our strategy $\sigma$. Here, the starting position $q_0$ and time horizon $T$ are fixed and known, but we don't know the parameter $a$. 

We would like to keep our cost as low as possible. We write $\excost(\sigma,a)$ to denote the expected value of the $\cost$ \eqref{eq: nintro 2} for the given $a$ in \eqref{eq: nintro 1}.

In this paper we study two variants of the above control problem, which we call \emph{Bayesian control} and \emph{agnostic control}. 

\underline{For Bayesian control}, we are given a prior probability distribution $d\prior(a)$ for the unknown $a$ in \eqref{eq: nintro 1}. We assume that $d\prior$ is supported in an interval $[-a_\mx, a_\mx]$. Our task is to pick the strategy $\sigma$ to minimize
\begin{equation}\label{eq: nintro 4}
    \excost(\sigma, d\prior) = \int_{-a_\mx}^{a_\mx} \excost(\sigma, a) d\prior(a).
\end{equation}

\begin{sloppypar}
\underline{For agnostic control}, we are given that $a$ belongs to a known interval $[-a_\mx,a_\mx]$, but we aren't given a prior probability distribution $d\prior(a)$, so we can't define an expected cost by \eqref{eq: nintro 4}. Instead, our goal will be to minimize \emph{worst-case regret}, defined by comparing the performance of our strategy with that of an opponent who knows the value of $a$ and plays optimally. Let $\sigma_\op(a)$ be the optimal strategy for known $a$. Thus $\excost(\sigma,a)$ is minimized over all $\sigma$ by taking $\sigma = \sigma_\op(a)$.\footnote{See standard textbooks (e.g., \cite{astrom}) for the computation of $\sigma_\op(a)$ and its expected cost.} We will introduce several variants of the notion of regret. 
\end{sloppypar}

To a given strategy $\sigma$ we associate the following functions on $[-a_\mx, a_\mx]$:
\begin{itemize}
    \item \emph{Additive Regret}, defined as
    \[
    \ar(\sigma,a) = \excost(\sigma,a) - \excost(\sigma_\op(a),a) \ge 0.
    \]
    \item \emph{Multiplicative Regret} (aka ``competitive ratio''), defined as
    \[
    \mreg(\sigma,a) = \frac{\excost(\sigma,a)}{\excost(\sigma_\op(a),a)} \ge 1.
    \]
    \item \emph{Hybrid Regret}, defined in terms of a parameter $\gamma>0$ by setting 
    \[
    \hr_\gamma(\sigma,a) = \frac{\excost(\sigma,a)}{\excost(\sigma_\op(a),a)+\gamma}
    \]
\end{itemize}
See \cite{almostoptimal2023} for a discussion of the regimes in which these three notions provide useful information.

Writing $\regret(\sigma,a)$ to denote any one of the above three functions on $[-a_\mx, +a_\mx]$, we define the \emph{worst-case regret}:
\begin{equation}\label{eq: nintro 5}
    \regret^*(\sigma) = \sup\{ \regret(\sigma, a) : a \in [-a_\mx, a_\mx]\}.
\end{equation}
We seek a strategy $\sigma$ having the least possible worst-case regret.

Thus, we have posed two problems: For Bayesian control, find the strategy that minimizes our expected cost; and for agnostic control, find a strategy that minimizes worst-case regret.

To prepare to present our results, we next discuss a relevant PDE, the \emph{Bellman equation}.

We will see that the problem of Bayesian control is intimately connected to the following PDE for an unknown function $S(q,t,\zo, \zt)$ of four variables:
\begin{equation}\label{eq: nintro +1}
        \begin{split}
        0 = &\partial_t S + (\bar{a}(\zo,\zt) q + u_\op)\partial_q S + \bar{a}(\zo,\zt) q^2 \partial_{\zo}S + q^2 \partial_{\zt}S+ \frac{1}{2}\partial_q^2 S \\& + q \partial_{q\zo} S  + \frac{1}{2}q^2 \partial_{\zo}^2 S + (q^2 + u_\op^2),
    \end{split}
\end{equation}
where
\begin{equation}\label{eq: nintro +2}
u_\op = - \frac{1}{2} \partial_q S,
\end{equation}
with terminal condition
\begin{equation}\label{eq: nintro +3}
    S|_{t = T} = 0.
\end{equation}
Here, $\bar{a}(\zo,\zt)$ is a known, smooth function of two variables.

We have succeeded in finding numerical solutions of \eqref{eq: nintro +1}--\eqref{eq: nintro +3}, but we lack rigorous proofs of existence and smoothness of solutions. Accordingly, we impose a \emph{PDE Assumption} to the effect that \eqref{eq: nintro +1}--\eqref{eq: nintro +3} admit a solution $S$ satisfying plausible estimates (see Section \ref{sec: pde}). Our numerics suggest that the PDE Assumption is correct. \underline{Our results below are conditional on the PDE Assumption}.

We are ready to state our main results. We begin with Bayesian control. For a function $\bar{a}(\zo, \zt)$ given in terms of $d\prior$ by an elementary formula, we define a function $u_\op(q,t,\zo,\zt)$ as in \eqref{eq: nintro +1}--\eqref{eq: nintro +3}, and then specify a strategy $\sigma = \sigma_\bayes(d\prior)$ by setting
\begin{align}
    &\usig(t) = u_\op(\qsig(t),t,\zo(t),\zt(t)), \;\text{with}\label{eq: nintro *1}\\
    &\zo(t) = \int_0^t \qsig(s)[ d\qsig(s) - \usig(s) ds]\qquad \text{and}\qquad \zt(t) = \int_0^t (\qsig(s))^2 \ ds.\label{eq: nintro *2}
\end{align}
Note that $\zo(t)$, $\zt(t)$ are determined by past history up through time $t$, hence so is $\usig(t)$ in \eqref{eq: nintro *1}. As explained in \cite{almostoptimal2023}, heuristic reasoning suggests that $\sigma_\bayes(d\prior)$ is the optimal strategy for Bayesian control with prior belief $d\prior$. Our rigorous result confirms this intuition. Recall that $q_0$ is our starting position.
\begin{thm}\label{thm: nintro 1}
    Fix a probability distribution $d\eprior$ on $[-a_\emx,a_\emx]$, and let $S$, $u_\eop$, $\sigma = \sigma_{\emph{Bayes}}(d\eprior)$ be as above. Then the following hold.
    \begin{enumerate}[label={\emph{(\Alph*)}}, ref={(\Alph*)}]
        \item $\excost(\sigma, d\eprior) = S(q_0,0,0,0)$.
        \item Let $\sigma'$ be any other strategy. Then
        \[
        \excost(\sigma',d\eprior) \ge \excost(\sigma, d\eprior).
        \]
    \label{thm: obs b}
    \end{enumerate}
\end{thm}

For a class of ``tame strategies'' $\sigma'$, we can sharpen \ref{thm: obs b} above to a quantitative result. A \emph{tame strategy} $\sigma'$ satisfies the estimate
\[
|u^{\sigma'}(t)| \le \hat{C} [ |q^{\sigma'}(t)| + 1]\;\text{(all }t \in [0,T])
\]
with probability 1, for a constant $\hat{C}$, called a \emph{tame constant for $\sigma'$}. The quantitative version of \ref{thm: obs b} is as follows.
\begin{thm}[Quantitative Uniqueness]\label{thm: nintro 2}
Let $d\eprior$, $\sigma = \sigma_{\emph{Bayes}}(d\eprior)$ be as in Theorem \ref{thm: nintro 1}. Given $\varepsilon>0$ and given a constant $\hat{C}$, there exists $\delta >0$ for which the following holds:

Let $\sigma'$ be a tame strategy with tame constant $\hat{C}$.

If $\excost(\sigma',d\eprior) \le \excost(\sigma, d\eprior) + \delta$, then the expected value of 
\[
\int_0^T \{ |\qsig(t) - q^{\sigma'}(t)|^2 + |\usig(t) - u^{\sigma'}(t)|^2 \} \ dt
\]
is less than $\varepsilon$.
\end{thm}
Quantitative Uniqueness will play a crucial r\^{o}le in our analysis of agnostic control.

Our main result for agnostic control is as follows.
\begin{thm}\label{thm: nintro 3}
    Fix $[-a_\emx,a_\emx]$, $q_0$, $T$ (and $\gamma$ if we use hybrid regret). Then there exist a probability measure $d\eprior$, a finite subset $E\subset [-a_\emx, a_\emx]$, and a strategy $\sigma$, for which the following hold.
    \begin{itemize}
        \item[\emph{(I)}] $\sigma$ is the optimal Bayesian strategy for the prior probability distribution $d\eprior$.
        \item[\emph{(II)}] $d\eprior$ is supported in the finite set $E$.
        \item[\emph{(III)}] $E$ is precisely the set of points $a \in [-a_\emx,a_\emx]$ at which the function $[-a_\emx,a_\emx] \ni a \mapsto \regret(\sigma,a)$ achieves its maximum.
        \item[\emph{(IV)}] $\regret^*(\sigma) \le \regret^*(\sigma')$ for any other strategy $\sigma'$.
    \end{itemize}
\end{thm}
So, for optimal agnostic control, we should pretend to believe that the unknown $a$ is confined to a finite set $E$ and governed by the probability distribution $d\prior$, even though in fact we know nothing about $a$ except that it lies in $[-a_\mx,a_\mx]$.

Let $d\prior$, $\sigma$, $E$ be as in (I), (II), (III) of Theorem \ref{thm: nintro 3}. Since $\sigma$ is the optimal Bayesian strategy for $d\prior$ (by (I)), and since $d\prior$ is supported on the finite set $E$ (by (II)), we have for any other strategy $\sigma'$ that
\[
\excost(\sigma, a_0) \le \excost(\sigma', a_0) \;\text{for some} \; a_0 \in E.
\]
In particular, we have
\[
\regret(\sigma, a_0) \le \regret(\sigma' , a_0)\;\text{for some}\; a_0 \in E.
\]
Combining this with (III), we see that for any $a \in [-a_\mx, a_\mx]$ we have
\[
\regret(\sigma, a) \le \regret(\sigma',a_0).
\]
Therefore (I), (II), (III) of Theorem \ref{thm: nintro 3} easily imply (IV). The hard part of Theorem \ref{thm: nintro 3} is the assertion that there exist $d\prior$, $E$, $\sigma$, satisfying (I), (II), (III). 

Theorem \ref{thm: nintro 3} lets us search for optimal agnostic strategies: We first guess a finite set $E$ and a probability measure $d\prior$ concentrated on $E$. By solving the Bellman equation, we produce the optimal Bayesian strategy $\sigma= \sigma_{\bayes}(d\prior)$, which allows us to compute the function $[-a_\mx,a_\mx] \ni a \mapsto \regret(\sigma,a)$. If the maximum of that function occurs precisely at the points of $E$, then $\sigma$ is the desired optimal agnostic strategy. Otherwise, we modify our guess $(E, d\prior)$. We have carried this out numerically for several $[-a_\mx,a_\mx], q_0, T$.

This concludes our introductory discussion of agnostic control for bounded $a$ (i.e. $a \in [-a_\mx, a_\mx]$).

We briefly touch on another variant of the control problem \eqref{eq: nintro 1}: \emph{agnostic control for unbounded $a$}.

Suppose we assume absolutely nothing about our unknown $a$; it might be any real number. For any strategy $\sigma$ we define $\regret^*(\sigma)$ as in \eqref{eq: nintro 5}, except that now the $\sup$ is taken over all $a\in\R$. Our task is to pick $\sigma$ to minimize $\regret^*(\sigma)$.

Our companion paper \cite{almostoptimal2023} analyzes this problem by comparing optimal agnostic control for arbitrary $a$ with the case in which $a$ is confined to a large interval $[-a_\mx(\varepsilon), +a_\mx(\varepsilon)]$, depending on a small parameter $\varepsilon>0$. Roughly speaking, \cite{almostoptimal2023} shows that any strategy for $a$ confined to $[-a_\mx(\varepsilon), +a_\mx(\varepsilon)]$ may be modified to produce a strategy for arbitrary $a \in \R$, with an increase in worst-case hybrid regret of at most $\varepsilon$. (See \cite{almostoptimal2023} for precise statements.)

\subsection*{Recap}
Let us summarize what we have achieved. Suppose our goal is to minimize worst-case hybrid regret in the setting in which $a$ may be any real number. Modulo an arbitrarily small increase in regret, we may reduce matters to the case in which $a$ is confined to a bounded interval $[-a_\mx, a_\mx]$. We then look for a probability measure $d\prior$ living on a finite set $E\subset [-a_\mx, a_\mx]$, such that the regret of the optimal Bayesian strategy for $d\prior$ is maximized precisely on $E$. We can calculate the optimal Bayesian strategy for a given prior probability measure by solving a Bellman equation. However, our results are conditional; we have to make an assumption on the existence, smoothness, and size of solutions to the Bellman equation. In numerical simulations, we have produced evidence for our PDE Assumptions, and we have produced optimal agnostic strategies for cases in which the unknown $a$ is confined to an interval.

\subsection*{Ideas from the Proofs}
We mention one significant technical point regarding the proofs of Theorems \ref{thm: nintro 1}, \ref{thm: nintro 2}, \ref{thm: nintro 3}: We need a rigorous definition of a strategy. Certainly the phrase ``a rule for determining $u(t)$ from past history'' isn't precise.

We want to allow $u(t)$ to depend discontinuously on past history $(q(s))_{s \in [0,t]}$. For instance, we should be allowed to set 
\[
u(t)=\begin{cases}
    -q(t) &\text{if}\; |q(t)|>1,\\
    0 &\text{otherwise}.
\end{cases}
\]
On the other hand, we had better make sure that we can produce solutions of our stochastic ODE
\[
dq = (aq + u) dt + dW.
\]
We proceed as follows.

At first we fix a partition
\begin{equation}\label{eq: nintro sharp}
    0 = t_0 < t_1 < \dots < t_N =T
\end{equation}
of the time interval $[0,T]$. We restrict ourselves to strategies $\sigma$ in which the control $u(t)$ is constant in each interval $[t_\nu, t_{\nu+1})$, and in which, for each $\nu$, $u(\tnu)$ is determined by $q(t_1),\dots,q(\tnu)$, together with ``coin flips'' $\vxi = (\xi_1, \xi_2, \cdots) \in \{0,1\}^\N$. For all $\nu$, we assume that $u(\tnu)$ is a Borel measurable function of $(q(t_1),\dots, q(\tnu),\vxi)$, and that
\[
|u(\tnu)| \le C_\tame [|q(\tnu)| + 1].
\]
A strategy as above is called a \emph{tame strategy associated to the partition} \eqref{eq: nintro sharp}, with a \emph{tame constant} $C_\tame$. For such strategies, it is easy to define the solutions $\qsig(t)$, $\usig(t)$ of our stochastic ODE \eqref{eq: nintro 1}.

Most of our work lies in controlling and optimizing tame strategies associated to a sufficiently fine partition. In particular, we will prove approximate versions of Theorems \ref{thm: nintro 1}, \ref{thm: nintro 2} in the setting of such strategies.

We will then define a tame strategy (not associated to any partition) by considering a sequence $\pi_1, \pi_2,\dots$  of ever-finer partitions of $[0,T]$. To each partition $\pi_n$ we associate a tame strategy $\sigma_n$ with a tame constant $C_\tame$ independent of $n$. If the resulting $q^{\sigma_n}(t)$ and $u^{\sigma_n}(t)$ tend to limits, in an appropriate sense, as $n \rightarrow \infty$, then we declare these limits $q(t)$, $u(t)$ to arise from a \emph{tame strategy} $\sigma$ with a \emph{tame constant} $C_\tame$.

Finally, we drop the restriction to tame strategies and consider general strategies. To do so, we consider a sequence $(\sigma_n)_{n = 0,1,2,\dots}$ of tame strategies, NOT assumed to have a tame constant independent of $n$. If the relevant $q^{\sigma_n}(t)$ and $u^{\sigma_n}(t)$ converge, in a suitable sense, as $n \rightarrow \infty$, then we say that the limits $q(t)$, $u(t)$ arise from a strategy $\sigma$.

It isn't hard to pass from tame strategies associated to partitions of $[0,T]$ to general tame strategies, and then to pass from such tame strategies to general strategies. The work in proving Theorems \ref{thm: nintro 1} and \ref{thm: nintro 2} lies in our close study of tame strategies associated to fine partitions.

We provide only a few comments on the proof of Theorem \ref{thm: nintro 3}. The main work lies in proving an analogue of Theorem \ref{thm: nintro 3} in which the unknown $a$ is confined to a finite set $A \subset [-a_\mx, a_\mx]$, rather than to the whole of $[-a_\mx, a_\mx]$. We apply that analogue to a sequence $A_1, A_2, \dots$ of fine nets in $[-a_\mx, a_\mx]$, e.g., $A_n = [-a_\mx, a_\mx] \cap 2^{-n} \Z$, and pass to the limit as $n \rightarrow \infty$ using a weak compactness argument. To establish the result for finite $A$, we proceed by induction on the number of elements of $A$. Details may be found in Chapter \ref{chap: 5}.

\subsection*{Future Directions}
Our work suggests several unsolved problems, among which we mention:
\begin{itemize}
    \item Prove (or disprove) the PDE Assumption.
    \item Consider problems in which the particle lives in $\R^N$, not just in $\R^1$; and in which the dynamics of the particle depend on more than one unknown parameter. Can that be done without rendering the relevant numerics hopelessly impractical?
    \item Even for the model problem considered in this paper, improve the numerics to let us produce optimal agnostic strategies for a larger range of $a_\mx$ and $T$ than we can deal with today.
\end{itemize}

We speculate briefly on a particular model problem in which we don't know a priori what our control does.

Consider a particle governed by the stochastic ODE
\begin{equation}\label{eq: nintro !}
    dq(t) = au(t)dt + dW(t), \qquad q(0) = 0.
\end{equation}
As usual, $q(t)$ denotes position, $u(t)$ is our control, $W(t)$ is Brownian motion, and we incur a cost
\[
\int_0^T\{(q(t))^2 +(u(t))^2\} \ dt.
\]
In the simplest case, suppose we know a priori that $a=1$ or $a=-1$, each with probability $1/2$. We write $\excost(\sigma)$ to denote the expected cost incurred by executing a strategy $\sigma$, and we set
\begin{equation}\label{eq: nintro !!}
    \excost^* = \inf\{\excost(\sigma) : \text{All strategies } \sigma\}.
\end{equation}

For this simple model problem we conjecture that the inf in \eqref{eq: nintro !!} is not achieved by any strategy $\sigma$, because heuristic reasoning suggests that there is a regime in which we would like to set $u = \pm \infty$ to gain instant information about $a$.

Clearly there is much to be done before we can claim to understand agnostic control theory.

\subsection*{Survey of Prior Literature}

Literature that considers adaptive control of a simple linear system similar to the one considered in this paper commonly consists of one or more of the following features: (\emph{i}) unknown governing dynamics, (\emph{ii}) unknown cost function and (\emph{iii}) adversarial noise.  Examples of such work include \cite{kumar2022online,mania2019certainty,wagenmaker2020active,furieri2020learning,Cohen:2019,malik2019derivative,duchi2011adaptive} as well as our own prior work \cite{fefferman2021optimal}, \cite{carruth2022controlling}.

Initial work in obtaining regret bounds in the infinite time horizon for the related LQR (linear-quadratic regulator) problem was undertaken in \cite{abbasi2011regret}, which proved that under certain assumptions, the expected additive regret of the adaptive controller is bounded by $\tilde{O} (\sqrt{T})$. Further progress was made on this problem in \cite{chen2021black}. Assuming controllability of the system, the authors gave the first efficient algorithm capable of attaining sublinear additive regret in a single trajectory in the setting of online nonstochastic control. See also the related \cite{minasyan2021online}, which obtained sublinear adaptive regret bounds, a stronger metric than standard regret and more suitable for time-varying systems. Additional adaptive control approaches include \cite{dean2018regret,dean2019safely} using the system level synthesis. This expands on ideas in \cite{simchowitz2018learning}, which showed that the ordinary least-squares estimator learns a linear system nearly optimally in one shot. Other work uses Thompson Sampling \cite{abeille2017thompson,kargin2022thompson} or deep learning \cite{chen2023regret}.  Perhaps most related to the work performed in this study is \cite{jedra2022minimal}, which designed an online learning algorithm with sublinear expected regret that moves away from episodic estimates of the state dynamics (meaning that no boundedness or initially stabilizing control needed to be assumed).

In \cite{fefferman2021optimal}, the third and fourth authors of the present paper, along with B. Guill\'{e}n Pegueroles and M. Weber, found regret minimizing strategies for a problem with simple unknown dynamics (a particle moving in one-dimension at a constant, unknown velocity subject to Brownian motion). In \cite{gurevich2022optimal}, along with D. Goswami and D. Gurevich, they generalized these results to an analogous, higher-dimensional system with the addition of sensor noise. In \cite{fefferman2021optimal}, they also posed the problem of finding regret minimizing strategies for the more complicated dynamics \eqref{eq: nintro 1}. In \cite{carruth2022controlling}, the authors of the present paper, along with M. Weber, took the first steps toward resolving this problem. Specifically, we exhibited a strategy for the dynamics \eqref{eq: nintro 1} with bounded multiplicative regret.

Historically, significant work has been undertaken in the closely related ``multi-armed bandit'' problem; see, for instance, the classic papers \cite{robbins1952some,vermorel2005multi}. Recent work considering this paradigm includes \cite{wei2021non}, which used reinforcement learning to obtain dynamic regret whose order of magnitude is optimal, and \cite{faury2021regret}, which studied the more general Generalized Linear Bandits (GLBs) and obtained similar regret bounds.

We finally want to point out the parallel field of adversarial control, where the noise profile is chosen by an adversary instead of randomly. This includes \cite{martin2022safe}, which attained minimum dynamic regret and guaranteed compliance with hard safety constraints in the face of uncertain disturbance realizations using the system level synthesis framework, and \cite{goel2021competitive}, which studied the problem of competitive control.

As this list of references is by no means exhaustive and does not do justice to the wealth of studies in the literature, we point the reader to the book \cite{hazancontrol} and the references therein for a more thorough overview of online control.

We emphasize that our approach in \cite{almostoptimal2023}, \cite{carruth2022controlling}, \cite{fefferman2021optimal}, and the present paper differs from the other work cited above in that:
\begin{itemize}
    \item We seek strategies that minimize the worst-case regret for a fixed time horizon $T$, whereas the literature is mainly concerned with $T \rightarrow \infty$.
    \item \sloppy Typically in the literature one assumes either that the dynamics are bounded or that one is given a stabilizing control. We make no such assumptions in \cite{almostoptimal2023} and so we must control a system that is arbitrarily unstable.
    \item However, we achieve the above ambitious goals only for a simple model system.
\end{itemize}

We thank Amir Ali Ahmadi, Brittany Hamfeldt, Elad Hazan, Robert Kohn, Sam Otto, Allan Sly, and Melanie Weber for helpful conversations. We are grateful to the Air Force Office of Scientific Research, especially Frederick Leve, for their generous support. The second named author thanks the Joachim Herz foundation for support.

\chapter{The Game}

We will deal with random variables
\begin{flalign}
    & a_\tru \in [-a_\mx, + a_\mx],&\label{eq: bps 1}\\
    &\text{``Coin flips'' } \xi_1, \xi_2, \dots \in \{0,1\}\text{ (we write } \vxi \text{ for } (\xi_1,\xi_2,\dots)), \text{and}\\
    &\text{Brownian motion } W(t), \text{ starting at }W(0) = 0.
\end{flalign}

The random variable $a_\tru$ is deterministic and known in Chapter \ref{chap: 2}, unknown but subject to a known prior in Chapters \ref{chap: 3} and \ref{chap: 4}, and unknown without a known prior in Chapter \ref{chap: 5}.

The $\xi_\nu$, the real number $a_\tru$, and the Brownian motion are mutually independent. The variable $a_\tru$ has a prior probability distribution given by the measure $d \prior(a)$ in Chapters \ref{chap: 3} and \ref{chap: 4}; and each $\xi_\nu$ is equal to 0 with probability $1/2$, and to 1 with probability $1/2$.

When $a_\tru$ has a prior probability distribution, we write $\prob[\cE]$ to denote the probability of an event $\cE$ with respect to the above probability space, and we write $\E[X]$ for the expected value of $X$ with respect to that probability space.

Given $a \in [-a_\mx, + a_\mx ] $, we write $\prob_a[\cE]$ for the probability of event $\cE$ conditioned on the event $a_\tru = a$, and we write $\E_a[X]$ for the expectation of $X$ conditioned on $a_\tru = a$. $\prob_a[\cdots]$ and $\E_a[\cdots]$ make sense without an assumed prior for $a$.

Similarly, given $a \in [-a_\mx, +a_\mx]$ and $\veta = (\eta_1, \eta_2, \dots) \in \{0,1\}^\N$, we write $\prob_{a,\veta}[\cE]$ and $\E_{a,\veta}[X]$ to denote the probability and expectation, respectively, conditioned on the event $a_\tru = a$ and $\xi_\nu = \eta_\nu$ for all $\nu$.

Also, we write $\E_{\veta}[X]$ and $\prob_{\veta}[\cE]$ to denote expectation and probability, respectively, conditioned on $\vxi = \veta$.

The above conditional expectations make sense even if, e.g., $a$ isn't in the support of $d\prior$.

Fix a \emph{terminal time} $T>0$ and a partition
\[
0 = t_0 < t_1 < \dots < t_N = T \;\text{of}\; [0,T].
\]
Fix a \emph{starting position} $q_0 \in \R$. A \emph{tame rule} at time $t_\nu$ is a Borel measurable function $\sigma_{\tnu}: \R^\nu \times \{0,1\}^\N\rightarrow \R$, satisfying the estimate
\begin{equation}\label{eq: tg 1}
|\sigma_{\tnu}(q_1, \dots, \qnu, \vxi)| \le C_\tame [|\qnu| + 1]
\end{equation}
for all $(q_1,\dots, q_\nu, \vxi) \in \R^\nu \times \{0,1\}^\N.$
If $\nu =0$, then $\signu$ is simply a function on $\{0,1\}^\N$. (We use the product topology on $\R^\nu \times \{0,1\}^\N$ to define Borel measurability. We require Borel measurability to avoid technicalities. In particular, the composition of Borel measurable functions is Borel measurable, whereas the composition of Lebesgue measurable functions needn't be Lebesgue measurable.)

A \emph{tame strategy} is an array $\sigma = (\signu)_{\nu=0,1,\dots,N-1}$, where, for each $\nu$, $\signu$ is a tame rule at time $t_\nu$ with the same $C_\tame$ serving in \eqref{eq: tg 1} for all the $\tnu$. We call $C_\tame$ a \emph{tame constant} for the strategy $\sigma$. Until further notice, we say simply ``strategy'' in place of ``tame strategy.'' If the $\signu$ don't depend on the coin flips $\vxi$, we call $\sigma$ a \emph{deterministic strategy}. We will often write $\sigma_\nu$ in place of $\sigma_{\tnu}$.

Given a strategy $\sigma = (\signu)_{\nu=0,1,\dots,N-1}$, we define random variables $\qsig(t)$ for $t \in [0,T]$ and $\usig(t)$ for $t \in [0,T)$, as follows.

By induction on $\nu$, we define $\qsig(t)$ for $t \in [0, \tnu]$ and $\usig(t)$ for $ t\in [0,\tnu)$.

In the base case $\nu=0$, we set $\qsig(t) = q_0$ for $t \in [0,\tnu]=\{0\}$. Since $[0,\tnu) = [0,0)$ is empty, there is no need to define $\usig$ in the base case.

For the induction step, we fix $\nu \ge 0$, and assume that we have defined $\qsig(t)$ for $t \in [0,\tnu]$ and $\usig(t)$ for $t \in [0,\tnu)$. We extend the definition of $\qsig(t)$ to $t \in [0, t_{\nu+1}]$, and that of $\usig(t)$ to $t \in [0,t_{\nu+1})$, as follows.
\begin{itemize}
\item For $t \in [\tnu, t_{\nu+1})$, we set
\[
\usig(t) = \signu(\qsig(t_1),\dots, \qsig(\tnu),\vxi).
\]
\item For $t \in [\tnu, t_{\nu+1}]$, we define $\qsig(t)$ as the solution of the stochastic ODE
\[
d\qsig(t) = (a_\tru \qsig(t) + \usig(t) )dt + dW(t),
\]
with the initial value $\qsig(t_\nu)$ already given by our induction hypothesis.
\end{itemize}
This completes our induction on $\nu$, so we have defined the random variables $\qsig(t), \usig(t)$.

In addition to $\qsig(t), \usig(t)$, we define random variables $\zosig(\tnu)$, $\ztsig(\tnu)$ $(0 \le \nu < N)$ by the following induction.
\begin{flalign*}
    &\zosig(t_0) = \ztsig(t_0) = 0. \quad \text{(Recall, $t_0=0$.)}\\
    &\zosig(t_{\nu+1}) = \zosig(\tnu) + \qsig(\tnu)\cdot[\Delta \qsig_\nu - \usig(\tnu)\Delta \tnu],\;\text{where}\\
    &\qquad \Delta \qsig_\nu  = \qsig(t_{\nu+1}) - \qsig(\tnu)\;\text{and}\; \Delta \tnu = t_{\nu+1} - t_\nu.\\
    & \ztsig(t_{\nu+1}) = \ztsig(\tnu) + (\qsig(\tnu))^2\Delta \tnu.
\end{flalign*}
Thus,
\[
\zosig(\tnu) = \sum_{0 \le \mu < \nu} \qsig(t_\mu)(\Delta \qsig_\mu - \usig(t_\mu)\Delta t_\mu)
\]
and
\[
\ztsig(\tnu) = \sum_{0 \le \mu < \nu} (\qsig(t_\mu))^2 \Delta t_\mu.
\]
We will try to pick our strategy $\sigma$ to make the expected value of
\[
\int_0^T [ (\qsig(t))^2 + (\usig(t))^2] \ dt 
\]
as small as possible.

\chapter{Tame Strategies Associated to Partitions}\label{chap: 2}

\section{Setup}

In this chapter, we take $a_\tru$ to be fixed, $a_\tru = a$, and we suppose that our strategy makes no use of coin flips.

We fix a partition
\begin{equation}\label{eq: 1.1}
    0 = t_0 < t_1 < \dots < t_N = T
\end{equation}
of a time interval $[0,T]$.

We fix a (deterministic) strategy $\sigma$ for the game with starting position $q_0$. We assume that our strategy is tame, i.e., 
\begin{equation}\label{eq: 1.2}
|\usig(\tnu)| \le C_{\tame}[|\qsig(t_\nu)|+1]
\end{equation}
for a constant $C_\tame$.

We write $c$, $C$, $C'$, etc.\ to denote constants determined by
\begin{itemize}
    \item $C_\tame$ in \eqref{eq: 1.2}
    \item An upper bound for the time horizon $T$
    \item An upper bound for $a_\tru$
    \item An upper bound for $|q_0|$.
\end{itemize}
These symbols may denote different constants in different occurrences.

We define
\[
\Delta t_\nu := t_{\nu+1} - t_\nu\;\text{for all}\; \nu \;(0 \le \nu < N)
\]
and we assume that
\begin{equation}\label{eq: 1.3}
(\Delta t_\mx) := \max_\nu \Delta t_\nu \; \text{is less than a small enough constant } c.
\end{equation}

We write $X=O(Y)$ to denote the estimate $|X| \le C Y$.

We write $\qnu$ to denote $\qsig(t_\nu)$, $\unu$ to denote $\usig(\tnu)$, and $\dqnu$ to denote $q_{\nu+1} - \qnu$.

Note that
\begin{equation}\label{eq: 1.3.5}
    \unu = \sigma_\nu(q_1,\dots,\qnu),
\end{equation}
where $\sigma_\nu$ is given by the strategy  $\sigma$ for decisions at time $\tnu$.

Thanks to \eqref{eq: 1.2}, we have
\begin{equation}\label{eq: 1.4}
    |\unu| \le C [|\qnu|+1].
\end{equation}
Recall that the $\qnu$ evolve as follows.

$q_0$ is given, and $u_0$ is specified by the strategy $\sigma$. $q_{\nu+1}$ and $u_{\nu+1}$ are then determined from $\qnu$ and $\unu$ as follows.
\begin{itemize}
    \item We solve the stochastic ODE
\begin{equation}\label{eq: 1.5}
    dq(t) = (aq(t) + \unu)dt + dW(t) \qquad \text{for } t \in [\tnu, t_{\nu+1}],
\end{equation}
with initial condition $q(\tnu) = \qnu$. Here, $a$ is the (given) value of $a_\tru$, and $W(t)$ denotes Brownian motion at time $t$.
\item We set $q_{\nu+1} = q(t_{\nu+1})$.
\item We set $u_{\nu+1} = \sigma_{\nu+1}(q_1,\dots,q_{\nu+1})$ (compare with \eqref{eq: 1.3.5}).
\end{itemize}
Thus, the $\qnu$, $\unu$ are random variables defined by induction on $\nu$. 

Solving the ODE \eqref{eq: 1.5} using an integrating factor, we find that
\begin{equation}\label{eq: 1.6}
    q(t) - q(\tnu) = (a\qnu + \unu) \bigg[ \frac{e^{a(t-\tnu)}-1}{a}\bigg] + \int_{\tnu}^t e^{a(t-s)}\ dW(s)
\end{equation}
for $ t \in [t_\nu, t_{\nu+1}]$.

In particular,
\begin{equation}\label{eq: 1.7}
\dqnu = q_{\nu+1} - \qnu = (a\qnu + \unu) \Delta t_\nu^* + \Delta W_\nu,
\end{equation}
where 
\begin{equation}\label{eq: 1.8}
    \Delta t_\nu^* = \bigg[ \frac{e^{a\Delta t_\nu}-1}{a}\bigg]
\end{equation}
and
\begin{equation}\label{eq: 1.9}
    \Delta W_\nu = \int_{\tnu}^{t_{\nu+1}} e^{a(t_{\nu+1}-s)}\ dW(s).
\end{equation}
(If $a=0$, we interpret the above fractions in square brackets as $(t-t_\nu)$ in \eqref{eq: 1.6}, and $\Delta t_\nu$ in \eqref{eq: 1.8}.) We warn the reader that $\Delta W_\nu \ne W(t_{\nu+1})-W(t_\nu)$.

Note that $\Delta W_\nu$ is a normal random variable with mean 0 and variance
\begin{equation}\label{eq: 1.10}
\Delta \tilde{t}_\nu = \bigg[ \frac{e^{2a\Delta \tnu}-1}{2a}\bigg]
\end{equation}
(again, equal to $\Delta t_\nu$ if $a=0$).

Note that
\begin{equation}\label{eq: 1.11}
    \Delta \tilde{t}_\nu, \Delta t_\nu^* = \Delta t_\nu + O((\Delta t_\nu)^2).
\end{equation}
Note also that $\Delta \tilde{t}_\nu$, $\Delta t_\nu^*$, $\Delta W_\nu$ depend on $a$.

We introduce the sigma algebras $\cF_\nu$, defined as the algebra of events determined by the $\Delta W_\mu$ for $0 \le \mu < \nu$. Note that $q_\nu$, $u_\nu$ and $\Delta W_\mu$ ($\mu < \nu$) are $\cF_\nu$-measurable (i.e., they are deterministic once we condition on $\cF_\nu$), while the $\Delta W_\mu$ for $\mu \ge \nu$ are independent of $\cF_\nu$.

\begin{rmk}
    Thanks to equation \eqref{eq: 1.7}, the sigma algebra $\cF_\nu$ may be equivalently defined to consist of all events determined by by $q_1, \dots, q_\nu$. This equivalence holds because $a_\etru$ has been fixed $(a_\etru = a)$. 
\end{rmk}

\section{Estimates for probabilities of outliers}\label{sec: rare events}

We suppose
\begin{equation}\label{eq: 1.12}
    Q>C \;\text{for a large enough} \; C,
\end{equation}
and we estimate the probability that $\max_\nu |q_\nu| > Q$. To do so, we set
\begin{flalign*}
    &\unu^1 = \unu/\qnu \;\text{and}\; \unu^0=0\;\text{if}\;|\qnu| > 1;&\\
    &\unu^1 = 0\;\text{and}\; \unu^0 = \unu\;\text{otherwise}.
\end{flalign*}
Thus,
\begin{equation}\label{eq: 1.12.5}
    \unu = \unu^1 q_\nu + \unu^0
\end{equation}
and
\begin{equation}\label{eq: 1.13}
|\unu^1|, |\unu^0| \le C.
\end{equation}
Also, $\unu^1$ and $\unu^0$ are $\cF_\nu$-measurable.

Thanks to \eqref{eq: 1.12.5}, we can rewrite \eqref{eq: 1.7} in the form
\[
\Delta \qnu = (a + \unu^1) \qnu (\Delta \tnu^*) +\unu^0(\Delta \tnu^*) + \Delta W_\nu,
\]
or equivalently,
\[
[e^{-at_{\nu+1}}q_{\nu+1}] = (1 + e^{-a\Delta t_\nu} \Delta \tnu^* \unu^1)[e^{-a\tnu}\qnu] +e^{-at_{\nu+1}}\Delta t_\nu^* \unu^0 + e^{-at_{\nu+1}}\Delta W_\nu.
\]
(See \eqref{eq: 1.8}.)

Setting
\begin{flalign}
&M_\nu = \prod_{0 \le \mu < \nu} (1+e^{-a\Delta t_\mu} \Delta t_\mu^* u_\mu^1)^{-1},\label{eq: 1.14}&\\
& m_\nu = e^{-at_{\nu+1}} \Delta t_\nu^* \unu^0, \;\text{and}\label{eq: 1.15}\\
& \qnu^* = M_\nu e^{-at_\nu} q_\nu,\label{eq: 1.16}
\end{flalign}
we see that
\begin{flalign}\label{eq: 1.17}
    & q_{\nu+1}^* = q_\nu^* + M_{\nu+1} m_\nu + M_{\nu+1} e^{-at_{\nu+1}}\Delta W_\nu,
\end{flalign}
and that
\begin{flalign}\label{eq: 1.18}
    &q_0^* = q_0.
\end{flalign}

Since $|e^{-a\Delta t_\mu} u_\mu^i| \le C \; (i=0,1)$ and 
\[
\sum_\mu (\Delta t_\mu^*) \le C,
\]
we have
\begin{equation}\label{eq: 1.19}
    c < M_\nu < C\;\text{and}\; |m_\nu| < C \Delta t_\nu.
\end{equation}
Moreover, $M_{\nu+1}$ is $\cF_\nu$-measurable.

Let $\lambda \in \R$, to be fixed below. From \eqref{eq: 1.17}, \eqref{eq: 1.19}, we have
\begin{equation}\label{eq: 1.20}
\exp(\lambda q_{\nu+1}^*) \le \{ \exp(\lambda q_\nu^*) \cdot \exp( C |\lambda| \Delta t_\nu)\} \exp(\{M_{\nu+1}e^{-at_{\nu+1}}\}\lambda \Delta W_\nu).
\end{equation}
Here, the quantities in curly brackets are $\cF_\nu$-measurable, while $\Delta W_\nu$ is independent of $\cF_\nu$.

Recalling that $\Delta W_\nu$ is normal, with mean 0 and variance $O(\Delta t_\nu)$, we deduce from \eqref{eq: 1.19}, \eqref{eq: 1.20} that
\[
\E[\exp(\lambda q_{\nu+1}^*) | \cF_\nu] \le \exp(\lambda q_\nu^*) \exp(C |\lambda| \Delta \tnu)\exp(C \lambda^2 \Delta \tnu).
\]

Thus, the random variables
\begin{equation}\label{eq: 1.21}
    Z_\nu = \exp(- C [|\lambda| + \lambda^2] \tnu) \exp(\lambda q_\nu^*)\quad (0 \le \nu \le N)
\end{equation}
form a supermartingale, with
\[
Z_0 = \exp(\lambda q_0) \le \exp(C |\lambda|).
\]

Consequently, for any $Q>0$ we have
\[
\prob[ \max_\nu Z_\nu > \exp(|\lambda| Q)] \le \exp(|\lambda| (C-Q)).
\]
By definition \eqref{eq: 1.21}, this means that
\begin{equation}\label{eq: 1.22}
    \prob[\lambda q_\nu^* - C [|\lambda| + \lambda^2] \tnu > |\lambda| Q\;\text{for some}\; \nu] \le \exp(|\lambda| (C-Q)).
\end{equation}
Taking $Q$ greater that $2C$ in \eqref{eq: 1.22} (see \eqref{eq: 1.12}), and picking $\lambda = \pm Q$, we learn from \eqref{eq: 1.22} that 
\[
\prob[ |q_\nu^*| > C Q\;\text{for some}\; \nu] \le C\exp(-c Q^2).
\]
Recalling \eqref{eq: 1.16} and \eqref{eq: 1.19}, we conclude that
\begin{equation}\label{eq: 1.23}
    \prob[ \max_\nu | q_\nu| > Q ] \le C \exp(-c Q^2)
\end{equation}
if $Q$ satisfies \eqref{eq: 1.12}.

Thus, we have succeeded in estimating the probability that $\max_\nu |q_\nu|$ is large.

Immediately from \eqref{eq: 1.4} and \eqref{eq: 1.23} we have also
\begin{equation}\label{eq: 1.24}
    \prob[ \max_\nu |u_\nu| > Q] \le C \exp(-c Q^2)
\end{equation}
if $Q$ satisfies \eqref{eq: 1.12}.

We now turn our attention to
\begin{equation}\label{eq: 1.25}
    \zo(\tnu) = \sum_{0 \le \mu < \nu} q_\mu [ \Delta q_\mu - u_\mu \Delta t_\mu]
\end{equation}
and
\begin{equation}\label{eq: 1.26}
    \zt(\tnu) = \sum_{0\le \mu < \nu} q_\mu^2 \Delta t_\mu.
\end{equation}
Note that $\zo(\tnu)$ can be rewritten as
\begin{equation}\label{eq: 1.27}
\begin{split}
    \zo(\tnu) = &\sum_{0 \le \mu < \nu}\Big\{\frac{1}{2} (q_{\mu+1}^2 - q_\mu^2) - \frac{1}{2} (\Delta q_\mu)^2 \Big\} - \sum_{0 \le \mu < \nu} u_\mu q_\mu \Delta t_\mu\\
    = & \frac{1}{2} \qnu^2 - \frac{1}{2} q_0^2 - \frac{1}{2} \sum_{\mu = 0}^{\nu-1}(\Delta q_\mu)^2 - \sum_{0 \le \mu < \nu} u_\mu q_\mu \Delta t_\mu.
    \end{split}
\end{equation}

Now suppose that
\begin{equation}\label{eq: 1.28}
    \max_\nu |q_\nu|, \max_\nu |u_\nu| \le CQ\;\text{with}\; Q\; \text{as in \eqref{eq: 1.12}}.
\end{equation}
Then from \eqref{eq: 1.26}, \eqref{eq: 1.27}, we have
\begin{equation}\label{eq: 1.29}
    |\zt(\tnu)|\le C Q^2\quad \text{(all }\nu),
\end{equation}
and
\begin{equation}\label{eq: 1.30}
    |\zo(\tnu)| \le C Q^2 + \Big| \sum_{0 \le \mu < \nu} \{ (\Delta q_\mu)^2 - \Delta \tilde{t}_\mu \} \Big| \quad \text{(all }\nu),
\end{equation}
since also
\[
\sum_{\mu=0}^N (\Delta \tilde{t}_\mu) \le C \sum_{\mu=0}^N \Delta t_\mu \le C'.
\]

We will show that
\begin{multline}\label{eq: 1.31}
    \prob \Big[ \max_\nu \Big| \sum_{0 \le \mu < \nu} \big\{ (\Delta q_\mu)^2 - \Delta \tilde{t}_\mu\big\} \Big| > C Q^2 (\Delta t_\mx)^{1/2} \Big]\\ \le C \exp(-c Q^2).
\end{multline}
In view of \eqref{eq: 1.23}, \eqref{eq: 1.24}, and \eqref{eq: 1.31}, estimates \eqref{eq: 1.29}, \eqref{eq: 1.30} imply the inequalities
\begin{equation}\label{eq: 1.32}
    \prob[ \max_\nu |\zo(\tnu)| > C Q^2] \le C \exp(-c Q^2)
\end{equation}
and
\begin{equation}\label{eq: 1.33}
    \prob[ \max_\nu | \zt(\tnu)| > C Q^2] \le C \exp(-c Q^2)
\end{equation}
for $Q$ as in \eqref{eq: 1.12}. Thus, to prove \eqref{eq: 1.32} and \eqref{eq: 1.33}, it remains only to prove \eqref{eq: 1.31}. Estimate \eqref{eq: 1.31} will have further applications in a later section.

We now prove \eqref{eq: 1.31}.

From \eqref{eq: 1.7}, \eqref{eq: 1.11} we have
\begin{equation}\label{eq: 1.34}
    \begin{split}
        \sum_{\mu < \nu} \{(\Delta q_\mu)^2 - \Delta \tilde{t}_\mu\} = & \sum_{\mu < \nu}(a q_\mu + u_\mu)^2 (\Delta t_\mu^*)^2 \\ &+ 2 \sum_{\mu< \nu} ( a q_\mu + u_\mu)(\Delta t_\mu^*) \Delta W_\mu
        \\
        &+ \sum_{\mu < \nu} \{ (\Delta W_\mu)^2 - \Delta \tilde{t}_\mu\}\\
        \equiv& \text{TERM 1}(\nu) + \text{TERM 2}(\nu) + \text{TERM 3}(\nu),
    \end{split}
\end{equation}
with
\begin{equation}\label{eq: 1.35}
    0 \le \text{TERM 1} (\nu) \le C \max_\mu \{ |q_\mu| +  |u_\mu|\}^2 \cdot (\Delta t_\mx), \;\text{all }\nu.
\end{equation}

To estimate $\text{TERM 2}(\nu)$, we fix $Q>C$ as in \eqref{eq: 1.12} and study the random variables
\begin{multline}\label{eq: 1.36}
    Y_\nu = \exp(-\hat{C}\lambda^2 Q^2 (\Delta t_\mx)^2 \tnu) \\ \cdot\exp\Big(\lambda \sum_{\mu < \nu} (a q_\mu + u_\mu) \cdot \mathbbm{1}_{|a q_\mu + u_\mu|< CQ}(\Delta t_\mu^*) \Delta W_\mu\Big)
\end{multline}
for a large enough constant $\hat{C}$, and for $\lambda \in \R$ to be picked below. Since $(aq_\mu + u_\mu)\cdot \mathbbm{1}_{|aq_\mu + u_\mu|< CQ}$ is $\cF_\nu$-measurable for $\mu \le \nu$, we have
\begin{equation}\label{eq: 1.37}
\begin{split}
    \E[Y_{\nu+1}|\cF_\nu] = & Y_\nu \cdot \exp(- \hat{C} \lambda^2 Q^2 (\Delta t_\mx)^2 \Delta t_\nu) \cdot \\ &\E[\exp([\lambda (a \qnu + \unu)\cdot \mathbbm{1}_{|a \qnu + \unu| < CQ} \cdot (\Delta t_\nu^*)] \Delta W_\nu | \cF_\nu]\\
    \le & Y_\nu \cdot \exp(- \hat{C} \lambda^2 Q^2 (\Delta t_\mx)^2 \Delta t_\nu)\\
    &\cdot \exp(C [\lambda(a \qnu + \unu) \cdot \mathbbm{1}_{|a \qnu + \unu|<CQ}\cdot (\Delta t_\nu^*)]^2 \Delta \tilde{t}_\nu).
\end{split}
\end{equation}
(Here, we use the fact that $\Delta W_\nu$ is independent of $\cF_\nu$, and normal with mean 0 and variance $\Delta \tilde{t}_\nu$.)

If we take $\hat{C}$ large enough, then the product of the exponentials on the right in \eqref{eq: 1.37} is less than 1. Thus,
\[
\E[Y_{\nu+1} | \cF_\nu] \le Y_\nu, \;\text{with}\; Y_0 \equiv 1.
\]
So the $Y_\nu$ form a supermartingale. Consequently,
\begin{equation}\label{eq: 1.38}
    \prob(\exists \nu \;\text{s.t.}\; Y_\nu > \exp(cQ^2)) \le \exp(-c Q^2).
\end{equation}
We now pick $\lambda = [\text{sgn}]\tilde{c}[(\Delta t_\mx)]^{-1}$ with $\text{sgn} = \pm 1$ and $\tilde{c}>0$ a small enough constant.  Combining \eqref{eq: 1.36} and \eqref{eq: 1.38} then yields the estimate
\begin{multline*}
\prob\Big(\exists \nu \;\text{s.t}\; [(\Delta t_\mx)]^{-1} \Big| \sum_{\mu < \nu} ( a q_\mu + u_\mu) \cdot \mathbbm{1}_{|aq_\mu + u_\mu|<CQ} (\Delta t_\mu^*) \Delta W_\mu \Big| > CQ^2\Big) \\ \le \exp(-cQ^2).
\end{multline*}
Consequently,
\begin{multline*}
    \prob\Big(\exists \nu \;\text{s.t.}\; \Big| \sum_{\mu < \nu}(a q_\mu + u_\mu)(\Delta t_\mu^*) \Delta W_\mu \Big| > C Q^2 (\Delta t_\mx)\Big)\\ \le \exp(-cQ^2) + \prob( \exists \mu \;\text{s.t.}\; |a q_\mu + u_\mu| > CQ).
\end{multline*}
Recalling \eqref{eq: 1.34}, we have
\begin{multline}\label{eq: 1.39f}
    \prob(\max_\nu | \text{TERM 2}(\nu)| > C Q^2 (\Delta t_\mx))\\ \le \exp(-c Q^2) + \prob( \max_\mu \{ | q_\mu | + |u_\mu|\} > c Q).
\end{multline}
We turn our attention to $\text{TERM 3}(\nu)$. Let 
\begin{equation}\label{eq: 1.40f}
    \cZ_\nu = \exp(- \hat{C} \lambda^2 (\Delta t_\mx) t_\nu) \cdot \exp\bigg(\lambda \sum_{\mu < \nu} \{ (\Delta W_\mu)^2 - \Delta \tilde{t}_\mu\}\bigg)
\end{equation}
for a large enough constant $\hat{C}$, and for $\lambda \in \R$ to be picked below.

Then $\cZ_0 \equiv 1$, and
\begin{equation}\label{eq: 1.41f}
    \E[\cZ_{\nu+1} | \cF_\nu] = \cZ_\nu \exp(- \hat{C} \lambda^2 (\Delta t_\mx) \Delta t_\nu) \cdot \E[\exp(\lambda\{ (\Delta W_\nu)^2 - \Delta \tilde{t}_\nu \} ) ].
\end{equation}

For $|\lambda | < c (\Delta t_\mx)^{-1}$, we have
\begin{equation}\label{eq: 1.42f}
\begin{split}
    \E[\exp(\lambda \{ (\Delta W_\nu)^2 - &\Delta \tilde{t}_\nu\})] \\
    &= \exp(-\lambda \Delta \tilde{t}_\nu) \cdot \frac{1}{\sqrt{2\pi \Delta \tilde{t}_\nu}} \int_{-\infty}^\infty e^{\lambda x^2} e^{- x^2/2 \Delta \tilde{t}_\nu}\ dx\\
    & = \exp(- \lambda \Delta \tilde{t}_\nu) \cdot (1-2\lambda(\Delta \tilde{t}_\nu))^{-1/2}\\
    & \le \exp( C \lambda^2 (\Delta t_\nu)^2) \le \exp(C\lambda^2 ( \Delta t_\mx) \Delta t_\nu).
\end{split}
\end{equation}
Substituting \eqref{eq: 1.42f} into \eqref{eq: 1.41f}, and taking $\hat{C}$ large enough, we find that
\[
\E[\cZ_{\nu+1} | \cF_\nu] \le \cZ_\nu,
\]
i.e., the $\cZ_\nu$ form a supermartingale. This holds for $|\lambda| < c (\Delta t_\mx)^{-1}$. Since $\cZ_0 \equiv 1$, it follows that
\[
\prob(\exists \nu \;\text{s.t.}\; \cZ_\nu > \exp(c Q^2)) \le \exp(-c Q^2).
\]
Taking $\lambda = (\Delta t_\mx)^{-1/2}\cdot [\text{sgn}]$ with $\text{sgn} = \pm 1$, we conclude that
\begin{multline*}
        \prob\Big( \exists \nu \;\text{s.t.}\; \exp\Big(-\hat{C} t_\nu + (\Delta t_\mx)^{-1/2} \Big| \sum_{\mu < \nu} \{ (\Delta W_\mu)^2 - \Delta \tilde{t}_\mu\} \Big| \Big) > \exp(c Q^2)\Big) \\ \le \exp(-c Q^2),
\end{multline*}
so that
\[
\prob\Big( \max_\nu \Big| \sum_{\mu < \nu}\{ (\Delta W_\mu)^2 - \Delta \tilde{t}_\mu \} \Big| > C Q^2 (\Delta t_\mx)^{1/2}\Big) \le \exp(-c Q^2).
\]
Recalling \eqref{eq: 1.34}, we conclude that
\begin{equation}\label{eq: 1.43f}
    \prob(\max_\nu | \text{TERM 3}(\nu) | > C Q^2 (\Delta t_\mx)^{1/2})) \le \exp(-c Q^2).
\end{equation}
Estimates \eqref{eq: 1.35}, \eqref{eq: 1.39f}, \eqref{eq: 1.43f} control TERM 1($\nu$), TERM 2$(\nu)$, TERM 3($\nu$). Substituting these estimates into \eqref{eq: 1.34}, we learn that
\begin{multline*}
    \prob\Big(\max_\nu \Big| \sum_{\mu < \nu} \{ (\Delta q_\mu)^2 - \Delta \tilde{t}_\mu\}\Big| > C' Q^2 (\Delta t_\mx)^{1/2}\Big) \\ 
    \le C \exp(-cQ^2) + C \prob(\max_\mu \{ |q_\mu| + |u_\mu|\} > c Q).
\end{multline*}
Finally, recalling \eqref{eq: 1.23} and \eqref{eq: 1.24}, we see that
\begin{equation*}
    \prob\Big( \max_\nu \Big| \sum_{\mu < \nu} \{ (\Delta q_\mu)^2 - \Delta \tilde{t}_\mu\} \Big| > C'' Q^2 (\Delta t_\mx)^{1/2}\Big) \le C \exp(-c Q^2),
\end{equation*}
completing the proof of \eqref{eq: 1.31}.

Next, we estimate
\begin{flalign*}
    &\Delta \qnusig = \qsig(t_{\nu+1}) - \qsig(\tnu),&\\
    &\Delta \zonusig = \zosig(t_{\nu+1}) - \zosig(\tnu),\\
    &\Delta \ztnusig = \ztsig(t_{\nu+1}) - \ztsig(\tnu).
\end{flalign*}
Recall that
\begin{flalign*}
    &\Delta \qnusig = (a \qnusig + \unusig) \Delta \tnu^* + \Delta W_\nu,&\\
    &\Delta \zonusig = \qnusig(\Delta \qnusig - \unusig \Delta \tnu),\\
    &\Delta \ztnusig = (\qnusig)^2 \Delta \tnu.
\end{flalign*}
Let $Q \ge C$ for large enough $C$, and suppose $|\qnusig| \le Q.$ Then also $|\unusig| \le C [|\qnusig| + 1]\le C'Q$, so
\begin{flalign*}
    &|\Delta \qnusig | \le C Q(\Delta t_\nu) + |\Delta W_\nu|,&\\
    &|\Delta \zonusig| \le Q (|\Delta \qnusig| + CQ\Delta \tnu) \le C' Q^2(\Delta \tnu) + Q |\Delta W_\nu|,\;\text{and}\\
    &|\Delta \ztnusig | \le Q^2 (\Delta t_\nu),
\end{flalign*}
hence for $p\ge 1$, we have
\[
(|\Delta \qnusig | + |\Delta \zonusig| + |\Delta \ztnusig|)^p \le C_p Q^{2p}(\Delta t_\nu)^p + C_p Q^p | \Delta W_\nu|^p.
\]
Recall that $\Delta W_\nu$ is independent of $\cF_\nu$ and normal, with mean 0 and variance at most $C \Delta \tnu$. it follows that, for any $p \ge 1$, we have
\begin{equation}\label{eq: 1.star.1}
\E[(|\Delta \qnusig | + |\Delta \zonusig| + |\Delta \ztnusig|)^p | \cF_\nu] \le C_p Q^{2p}(\Delta t_\nu)^p + C_p Q^p (\Delta t_\nu)^{p/2}
\end{equation}
whenever $|\qnusig| \le Q$. (Recall, $\qnusig$ is deterministic once we condition on $\cF_\nu$.) In particular, \eqref{eq: 1.star.1} implies that
\begin{equation}\label{eq: 1.star.2}
    \prob[ | \Delta \qnusig | + | \Delta \zonusig| + |\Delta \ztnusig| > (\Delta t_\nu)^{2/5} | \cF_\nu] \le C (\Delta \tnu)^{1000}
\end{equation}
if $|\qnusig| \le Q$ and $C \le Q \le (\Delta \tnu)^{-1/1000}$.

Together with \eqref{eq: 1.star.1} for $p=2,4$ and Cauchy-Schwartz, \eqref{eq: 1.star.2} implies the estimate
\begin{multline*}
\E[(|\Delta \qnusig| + | \Delta \zonusig| + |\Delta \ztnusig|)^p \cdot \mathbbm{1}_{|\Delta \qnusig| + |\Delta \zonusig| + |\Delta \ztnusig| > (\Delta \tnu)^{2/5}} | \cF_\nu]\\ \le C (\Delta \tnu)^{100}\;\text{for}\; p =1,2,
\end{multline*}
provided $|\qnusig| \le Q$ and $C' \le Q \le (\Delta \tnu)^{-1/1000}$ for large enough $C'$.

Next, we estimate $|\qsig(t) - \qnusig| = |\qsig(t) - \qsig(\tnu)|$ for $t \in [\tnu, t_{\nu+1}]$.

Recall that
\[
\qsig(t) - \qnusig = (a\qnusig + \unusig) \cdot \bigg[ \frac{e^{a(t-t_\nu)}-1}{a}\bigg] + \int_{t_\nu}^t e^{a(t-s)}\ dW(s)
\]
for $t \in [t_\nu, t_{\nu+1}]$.

If $|\qnusig| \le Q$ with $Q \ge C$ (for large enough $C$), then also $|\unusig| \le CQ$, hence
\[
|\qsig(t) - \qnusig | \le C Q(\Delta \tnu) + \bigg| \int_{\tnu}^t e^{a(t-s)}\ dW(s)\bigg|
\]
for $t \in [\tnu, t_{\nu+1}]$.

Applying the reflection principle \cite{feller} to the Gaussian process
\[
W^\#(t) = \int_{t_\nu}^t e^{-as} \ dW(s)\qquad (t \ge t_\nu),
\]
we see that
\[
\prob\Big[ \max_{t \in [t_\nu, t_{\nu+1}]}\Big| \int_{t_\nu}^t e^{a(t-s)}\ dW(s)\Big| > CQ(\Delta \tnu)^{1/2} \Big]\le C \exp(-cQ^2).
\]
Since $\int_{t_\nu}^t e^{a(t-s)}\ dW(s)$ $(t\ge \tnu)$ is independent of $\cF_\nu$, it now follows that
\[
\prob\Big[\max_{t \in [\tnu, t_{\nu+1}]}|\qsig(t) - \qnusig| > C'Q(\Delta \tnu)^{1/2} \Big| \cF_\nu\Big] \le C \exp(-c Q^2)
\]
provided $|\qnusig| \le Q $ and $C \le Q$ for large enough $C$. Taking 
\[
Q = (\Delta \tnu)^{2/5}/[C'(\Delta \tnu)^{1/2}],
\]
we find that
\begin{equation}
    \prob\Big[ \max_{t \in [t_\nu, t_{\nu+1}]} | \qsig(t) - \qnusig| > (\Delta \tnu)^{2/5} \Big| \cF_\nu\Big] \le C (\Delta \tnu)^{1000}
\end{equation}
provided $|\qnusig| \le c \cdot (\Delta \tnu)^{-1/10}$.

Let us summarize the results of the above discussion.

\begin{lem}[Lemma on Rare Events]\label{lem: rare events}
    We condition on $a_\etru = a$, $\vxi = \veta$. Fix a strategy $\sigma$. For constants $c,C$ depending only on upper bounds for $|q_0|$, $a_\emx$, $C_\etame$, and $T$, the following holds.

    Suppose $\Delta t_\emx \equiv \max_\nu ( t_{\nu+1} - t_\nu) < c$.

    Then, for $Q>C$, the following hold with probability $> 1 - \exp(-cQ^2)$:
    \begin{itemize}
        \item $|\qsig(\tnu)|, |\usig(\tnu)| \le Q$ for all $\nu$
        \item $|\zosig(\tnu)|, |\ztsig(\tnu)| \le Q^2$ for all $\nu$
        \item $\big| \sum_{0 \le \mu < \nu} (\qsig(t_{\mu+1}) - \qsig(t_\mu))^2 - \tnu \big| \le Q^2 (\Delta t_\emx)^{1/2}$ for all $\nu$.
    \end{itemize}
    Moreover, suppose we fix $\nu$ and condition on $\cF_\nu$, the sigma algebra of events determined by the $\qsig(t_\mu)$ $(0 \le \mu \le \nu)$. Suppose that $|\qnu|  < (\Delta \tnu)^{-1/1000}$. Then
    \[
    \eE[(|\Delta \qnusig| + |\Delta \zonusig| + |\Delta \ztnusig|)^p \cdot \mathbbm{1}_{|\Delta \qnusig| + |\Delta \zonusig| + |\Delta \ztnusig| > (\Delta \tnu)^{2/5}} | \cF_\nu] \le C (\Delta t_\nu)^{100}
    \]
    for $p = 1,2$, and
    \[
    \eprob[ |\Delta \qnusig| + |\Delta \zonusig| + |\Delta \ztnusig| > (\Delta \tnu)^{2/5} | \cF_\nu] \le C(\Delta \tnu)^{1000}.
    \]
    Also, we have
    \[
    \eprob\Big[ \max_{t \in [\tnu, t_{\nu+1}]} |\qsig(t) - \qnusig | > (\Delta \tnu)^{2/5} | \cF_\nu\Big] \le C (\Delta \tnu)^{1000}
    \]
    provided $|\qnusig| \le c \cdot (\Delta \tnu)^{-1/10}$. Finally, for $Q\ge C$, we have
    \[
    \eprob \Big[ \max_{t\in [\tnu, t_{\nu+1}]} |\qsig(t) - \qnusig| > C' Q(\Delta \tnu)^{1/2} | \cF_\nu\Big] \le C \exp(-cQ^2) \;\text{if} \; |q_\nu | \le Q.
    \]
\end{lem}
\begin{proof}
    To deduce the third bullet point from \eqref{eq: 1.31}, we note that
    \[
    \sum_{0 \le \mu < \nu} (\Delta \tilde{t}_\mu) = \sum_{0 \le \mu < \nu} [ (\Delta t_\mu) + O((\Delta t_\mu)^2)] = \tnu\cdot (1+O(\Delta t_\mx)).
    \]
    The remaining assertions of the lemma have already been proved as stated.
\end{proof}

\section{The Probability Density}

We continue to adopt the assumptions and notation of Section \ref{sec: rare events}. Our goal is to derive an approximate formula for the joint probability density of $(q_1, \dots, q_\barN) = (\qsig(t_1),\dots, \qsig(t_\barN))$ for fixed $\barN \le N$. Let us denote this joint probability density by $\Phi(\barq_1, \dots, \barq_\barN).$ Thus,
\begin{equation}\label{eq: 1.39}
    \prob((\qsig(t_1), \dots, \qsig(t_\barN)) \in E) = \int_E \Phi(\barq_1, \dots, \barq_\barN)\ d\barq_1\cdots d\barq_\barN
\end{equation}
for measurable sets $E\subset \R^\barN$.

By formula \eqref{eq: 1.7},
\[
\Delta q_\nu = (a q_\nu + u_\nu) \Delta t_\nu^* + \Delta W_\nu
\]
with $\Delta W_\nu$ mutually independent and normal, with mean 0 and variance $\Delta \tilde{t}_\nu$; consequently, the joint probability $\Phi$ is given by
\begin{equation}\label{eq: 1.40}
    \Phi(\barq_1, \dots, \barq_\barN) = \prod_{\nu = 0}^{\barN-1} \phi_\nu
\end{equation}
with
\begin{equation}\label{eq: 1.41}
    \phi_\nu = \frac{1}{\sqrt{2\pi \Delta \tilde{t}_\nu}} \exp\bigg( - \frac{1}{2\Delta \tilde{t}_\nu} [ \Delta \barq_\nu - (a \barq_\nu + \baru_\nu) \Delta t_\nu^*]^2 \bigg).
\end{equation}
Here, $\Delta \barq_\nu \equiv \barq_{\nu+1} - \barq_\nu$, $\barq_0 \equiv q_0$, and $\baru_\nu$ denotes the control exercised by the strategy $\sigma$ at time $t_\nu$ given that $\qsig(t_\mu) = \barq_\mu$ for $0 \le \mu \le \nu$ and $\vxi = \veta$. Note that $\baru_\nu$ is determined by $\barq_1, \dots, \barq_\nu$ (and $q_0)$.

We make the following assumptions on $(\barq_1, \dots, \barq_\barN)$:
\begin{flalign}
    &\max_\nu ( |\barq_\nu| + |\baru_\nu|) \le Q,\label{eq: 1.42} &\\
    & \Big| \sum_\nu (\Delta \barq_\nu)^2 - t_\barN \Big| \le Q^2 (\Delta t_\mx)^{1/4}, \label{eq: 1.43}
\end{flalign}
with $Q \ge C $ given.

Thanks to Lemma \ref{lem: rare events}, \eqref{eq: 1.42} and \eqref{eq: 1.43} are very likely true for $(\barq_1, \dots, \barq_\barN) = (\qsig(t_1), \dots, \qsig(t_\barN)).$ Under assumptions \eqref{eq: 1.42} and \eqref{eq: 1.43} we will simplify \eqref{eq: 1.40}, \eqref{eq: 1.41}.

First of all, since $\Delta \tnu^* = \Delta \tnu + O((\Delta \tnu)^2)$ we have
\[
[\Delta \barq_\nu - (a \barq_\nu + \baru_\nu) \Delta \tnu^*] = [ \Delta \barq_\nu - (a \barq_\nu + \baru_\nu) \Delta \tnu] + \text{ERR}_\nu,
\]
with
\[
\text{ERR}_\nu = O(Q( \Delta t_\nu)^2)
\]
thanks to \eqref{eq: 1.42}.

Hence,
\begin{align*}
    [\Delta \barq_\nu - ( a\barq_\nu + \baru_\nu) \Delta t_\nu^*]^2 = & [ \Delta \barq_\nu - (a\barq_\nu + \baru_\nu)^2 \Delta t_\nu]^2 + \text{ERR}_\nu^2\\
    & + 2 \text{ERR}_\nu [ \Delta \barq_\nu - ( a \barq_\nu + \baru_\nu) \Delta t_\nu]\\
    = & [ \Delta \barq_\nu - (a \barq_\nu + \baru_\nu) \Delta t_\nu]^2 + O(Q^2 (\Delta \tnu)^3) \\
    &+ 2 \text{ERR}_\nu (\Delta \barq_\nu).
\end{align*}
Therefore,
\begin{equation}\label{eq: 1.44}
    \begin{split}
        \sum_\nu \frac{1}{2\Delta \tilde{t}_\nu} [ \Delta \barq_\nu - &( a\barq_\nu + \baru_\nu) \Delta t_\nu^*]^2 \\ = & \sum_{\nu} \frac{1}{2\Delta \tilde{t}_\nu} [ \Delta \barq_\nu - (a\barq_\nu + \baru_\nu) \Delta \tnu]^2 + \sum_\nu O(Q^2(\Delta \tnu)^2) \\
        &+ \sum_\nu \frac{\text{ERR}_\nu}{\Delta \tilde{t}_\nu} ( \Delta \barq_\nu).
    \end{split}
\end{equation}
The last sum has absolute value at most
\begin{equation}\label{eq: 1.44.5}
C \sum_\nu ( \Delta \tnu)^{-1/2}\bigg\{ \frac{\text{ERR}_\nu}{\Delta \tilde{t}_\nu}\bigg\}^2 + C \sum_\nu ( \Delta \tnu)^{1/2}(\Delta \barq_\nu)^2.
\end{equation}
The expression \eqref{eq: 1.44.5} is, in turn, at most
\begin{align*}
    \sum_\nu O (Q^2(\Delta \tnu)^{3/2}) +& C (\Delta t_\mx)^{1/2} \sum_\nu ( \Delta \barq_\nu)^2\\
    =& O(Q^2 (\Delta t_\mx)^{1/2}) + C (\Delta t_\mx)^{1/2} \Big| \sum_\nu ( \Delta \barq_\nu)^2 - t_\barN\Big|\\
    &+ C t_\barN (\Delta t_\mx)^{1/2}\\
    =& O(Q^2 (\Delta t_\mx)^{1/2}),
\end{align*}
thanks to \eqref{eq: 1.43}. Therefore, \eqref{eq: 1.44} implies that
\begin{multline}\label{eq: 1.45}
    \sum_\nu \frac{1}{2\Delta \tilde{t}_\nu} [ \Delta \barq_\nu - ( a \barq_\nu + \baru_\nu) \Delta \tnu^*]^2\\
    = \sum_\nu \frac{1}{2 \Delta \tilde{t}_\nu} [ \Delta \barq_\nu - (a\barq_\nu + \baru_\nu) \Delta \tnu]^2 + O(Q^2 ( \Delta t_\mx)^{1/2}).
\end{multline}
We want to replace $\Delta \tilde{t}_\nu$ by $\Delta \tnu$ on the right in \eqref{eq: 1.45}. To do so, note that
\[
\frac{1}{2\Delta \tilde{t}_\nu} - \frac{1}{2\Delta \tnu} = - \frac{a}{2} + O(\Delta \tnu),
\]
thanks to \eqref{eq: 1.10}. Consequently,
\begin{equation}\label{eq: 1.46}
    \begin{split}
        \sum_\nu  \Big( \frac{1}{2\Delta \tilde{t}_\nu} - \frac{1}{2\Delta \tnu}\Big) [\Delta \barq_\nu &- (a\barq_\nu + \baru_\nu) \Delta \tnu]^2\\
        = & \sum_\nu [ a + O(\Delta \tnu)](a \barq_\nu + \baru_\nu) \Delta \tnu (\Delta \barq_\nu) \\&
        + \sum_\nu \Big[ - \frac{a}{2} + O(\Delta \tnu)\Big](\Delta \barq_\nu)^2\\
        & + \sum_\nu \Big[ - \frac{a}{2} + O(\Delta \tnu)\Big](a\barq_\nu + \baru_\nu)^2(\Delta \tnu)^2\\
        \equiv & \terma + \termb + \termc
    \end{split}
\end{equation}

Now
\begin{align*}
\termb &= - \frac{a}{2} \sum_\nu ( \Delta \barq_\nu)^2 + O(\Delta t_\mx) \sum_\nu (\Delta \barq_\nu)^2\\
& = - \frac{a}{2} t_\barN + O(Q^2 (\Delta t_\mx)^{1/4}) 
\end{align*}
by \eqref{eq: 1.43}, while
\[
\termc = \sum_\nu O(Q^2 (\Delta \tnu)^2) = O(Q^2 (\Delta t_\mx))
\]
by \eqref{eq: 1.42}. To estimate $\terma$, we apply \eqref{eq: 1.42} and \eqref{eq: 1.43} to write:
\begin{align*}
    | \terma | \le& C \sum_\nu ( a\barq_\nu + \baru_\nu)^2 ( \Delta \tnu)^{3/2} + C \sum_\nu (\Delta \tnu)^{1/2} ( \Delta \barq_\nu)^2\\
    \le & O(Q^2 (\Delta t_\mx)^{1/2}) + C (\Delta t_\mx)^{1/2} \sum_\nu (\Delta \barq_\nu)^2\\
    = & O(Q^2 (\Delta t_\mx)^{1/2}) + C (\Delta t_\mx)^{1/2}\Big[ \sum_\nu ( \Delta \barq_\nu)^2 - t_\barN\Big] \\
    = & O(Q^2 (\Delta t_\mx)^{1/2}).
\end{align*}

Combining our estimates for Terms $\alpha, \beta, \gamma$, and recalling \eqref{eq: 1.46}, we learn that
\[
\sum_\nu \Big( \frac{1}{2\Delta \tilde{t}_\nu} - \frac{1}{2\Delta \tnu}\Big) [ \Delta \barq_\nu - (a  \barq_\nu + \baru_\nu) \Delta \tnu]^2 = - \frac{a}{2} t_\barN + O(Q^2 (\Delta t_\mx)^{1/4}).
\]
Consequently, \eqref{eq: 1.45} implies that
\begin{multline}\label{eq: 1.47}
    \sum_\nu \frac{1}{2\Delta \tilde{t}_\nu} [ \Delta \barq_\nu - (a \barq_\nu + \baru_\nu) \Delta t_\nu^*]^2 \\= - \frac{a}{2} t_\barN + \sum_\nu \frac{1}{2\Delta \tnu} [ \Delta \barq_\nu - (a\barq_\nu +\baru_\nu) \Delta \tnu]^2 +  O(Q^2 (\Delta t_\mx)^{1/4}).
\end{multline}
Again applying \eqref{eq: 1.10}, we see that
\begin{multline*}
\frac{1}{\sqrt{2\pi \Delta \tilde{t}_\nu}} = \frac{1}{\sqrt{2\pi\Delta \tnu}} \cdot ( 1- \frac{1}{2}a \Delta \tnu + O(\Delta \tnu)^2) \\= \frac{1}{\sqrt{2\pi\Delta \tnu}} \exp(- \frac{1}{2}a \Delta \tnu + O(\Delta \tnu)^2)
\end{multline*}
so that
\begin{equation}\label{eq: 1.48}
    \begin{split}
        \prod_{\nu=0}^{\barN - 1} \frac{1}{\sqrt{2\pi\Delta \tilde{t}_\nu}} =& \bigg( \prod_{\nu=0}^{\barN-1} \frac{1}{\sqrt{2\pi \Delta \tnu}}\bigg) \exp\Big(- \frac{a}{2}\sum_\nu \{ (\Delta \tnu) + O(\Delta \tnu)^2\}\Big)\\
        = & \bigg( \prod_{\nu=0}^{\barN-1} \frac{1}{\sqrt{2\pi\Delta \tnu}}\bigg) \exp\Big( - \frac{a}{2} t_\barN + O(\Delta t_\mx)\Big).
    \end{split}
\end{equation}
Putting \eqref{eq: 1.47} and \eqref{eq: 1.48} into \eqref{eq: 1.40} and \eqref{eq: 1.41}, we find that
\begin{equation}\label{eq: 1.49}
\begin{split}
    \Phi(\barq_1,\dots,\barq_\barN) = \prod_{\nu=0}^{\barN-1} &\Big\{ \frac{1}{\sqrt{2\pi\Delta \tnu}}\exp\Big(- \frac{1}{2\Delta \tnu} [ \Delta \barq_\nu - (a \barq_\nu + \baru_\nu) \Delta \tnu]^2\Big)\Big\}\\ &\cdot (1 + O(Q^2 (\Delta t_\mx)^{1/4})).
\end{split}
\end{equation}
In particular, the factors $\exp(\frac{a}{2} t_\barN)$ arising from \eqref{eq: 1.47} and \eqref{eq: 1.48} cancel.

We record our result \eqref{eq: 1.49} as a lemma.

\begin{lem}[Lemma on the Probability Distribution]\label{lem: prob dist}
    We condition on $a_\etru = a$, $\vxi = \veta$.

    Then, for constants $c,C$ determined by $q_0$, $a_\emx$, $C_\etame,$ and an upper bound for $T$, the following holds.

    Suppose $\Delta t_\emx = \max_\nu (t_{\nu+1}- \tnu) < c$. Fix $\barN \le N$. Let $\Phi(\barq_1, \dots, \barq_\barN)$ be the joint probability density for $(\qsig(t_1), \dots, \qsig(t_\barN))$.

    Let $Q> C$, and suppose $(\barq_1,\dots, \barq_\barN)$ satisfies
    \begin{flalign*}
        &\max_\nu ( |\barq_\nu| + |\baru_\nu|) \le Q \;\text{and}&\\
        & \Big| \sum_\nu ( \barq_{\nu+1} - \barq_\nu)^2 - t_\barN \Big| \le Q^2 (\Delta t_\emx)^{1/4},
    \end{flalign*}
    where $\baru_\nu$ is the control exercised by the strategy $\sigma$ at time $\tnu$ when 
    \[
    (\qsig(t_1), \dots, \qsig(t_\nu))= (\barq_1, \dots, \barq_\nu)\;\text{and}\; \vxi = \veta.
    \]
    Then
    \begin{align*}
    \Phi(\barq_1,\dots,\barq_\barN) = \prod_{\nu=0}^{\barN-1} &\Big\{ \frac{1}{\sqrt{2\pi\Delta \tnu}}\exp\Big(- \frac{1}{2\Delta \tnu} [ \Delta \barq_\nu - (a \barq_\nu + \baru_\nu) \Delta \tnu]^2\Big)\Big\}\\ &\cdot (1 + O(Q^2 (\Delta t_\emx)^{1/4})).
    \end{align*}
    Here, $\Delta \barq_\nu = \barq_{\nu+1} - \barq_\nu$ (with $\barq_0 = q_0$), $\Delta \tnu = t_{\nu+1} - \tnu$, and $O(Q^2 (\Delta t_\emx)^{1/4})$ denotes a quantity whose absolute value is at most $CQ^2(\Delta t_\emx)^{1/4}$.
\end{lem}

\section{Analytic Continuation}
In this section we prepare to make an analytic continuation of the function mapping $a \in [-a_\mx, + a_\mx]$ to the expected cost incurred by a strategy $\sigma$ assuming that $a_\mx = a$. We set up notation.

We fix a tame strategy $\sigma = (\sigma_\nu)_{0 \le \nu < N}$ and coin flips $\veta$; we write $\baru_\nu$ to denote the control exercised by the strategy $\sigma$ at time $\tnu$ assuming that $\qsig(t_\mu) = \barq_\mu$ for $1 \le \mu \le \nu$ and $\vxi = \veta$ (i.e., $\baru_\nu = \usig(\tnu) = \sigma_\nu(\bar{q}_1, \dots, \bar{q}_\nu, \veta)$).

We define functions
\begin{equation}\label{eq: ac1}
    \psi_\nu(\Delta \barq, \barq, \baru,a) := \frac{1}{\sqrt{2\pi\Delta \tnu}}\exp\Big( - \frac{[\Delta \barq - (a\barq+ \baru)\Delta \tnu]^2}{2\Delta \tnu}\Big)
\end{equation}
and
\begin{equation}\label{eq: ac2}
    \Psi(\barq_1, \dots, \barq_N,a) = \prod_{\nu=0}^{N-1} \psi_\nu(\Delta \barq_\nu, \barq_\nu, \baru_\nu(\barq_1,\dots,\barq_\nu),a).
\end{equation}
In \eqref{eq: ac1}, $\Delta \barq$, $\barq$, $\baru$ are real variables, while $a \in \C$. In \eqref{eq: ac2}, $\Delta \barq_\nu := \barq_{\nu+1} - \barq_\nu$, with $\barq_0 := q_0$. Again, in \eqref{eq: ac2}, $a \in \C$.

We introduce a set
\begin{equation}\label{eq: ac3}
\begin{split}
    E  = \{(\barq_1,\dots,\barq_N) \in \R^N: & \max_\nu |\barq_\nu| \le (\Delta t_\mx)^{-1/16}\\ &\text{and}\; |\sum_\nu (\barq_{\nu+1}-\barq_\nu)^2 - T| \le (\Delta t_\mx)^{1/8}\}.
\end{split}
\end{equation}

We denote by

\[
\Phi(\barq_1,\dots,\barq_N,a)\qquad (a \in [-a_\mx, +a_\mx])
\]
the probability density for $(\qsig(t_1),\dots,\qsig(t_N))$ assuming that $a_\tru = a $ and $\vxi = \veta$.

According to Lemma \ref{lem: rare events} and Lemma \ref{lem: prob dist}, we have
\begin{equation}\label{eq: ac4}
    \Phi(\barq_1,\dots,\barq_N,a) = \Psi(\barq_1,\dots,\barq_N,a) \cdot (1+\err(\barq_1,\dots,\barq_N,a))
\end{equation}
for $(\barq_1,\dots,\barq_N)\in E$ and $a \in [-a_\mx, a_\mx]$, with
\begin{equation}\label{eq: ac5}
    |\err(\barq_1,\dots,\barq_N,a)|\le (\Delta t_\mx)^{1/8};
\end{equation}
and
\begin{equation}\label{eq: ac6}
    \E_{a,\veta} [ \mathbbm{1}_{(\qsig(t_1),\dots, \qsig(t_N)) \notin E}] \le C \exp(-c(\Delta t_\mx)^{-1/8}).
\end{equation}
Since $|\usig| \le C [|\qsig| + 1]$ for a tame rule $\sigma$, we have, for $a \in [-a_\mx, a_\mx]$,\
\begin{equation}\label{eq: ac7}
    \sum_\nu \{ (\usig_\nu)^2 + (\qsig_\nu)^2\} \Delta \tnu \le C \max_\nu |\qnu|^2 + C.
\end{equation}

By Lemma \ref{lem: rare events}, we have
\begin{equation}\label{eq: ac8}
    \E_{a,\veta}[ \{\max_\nu |\qsig(\tnu)|^2 + |\usig(\tnu)|^2\}^2] \le C.
\end{equation}

We study the function
\[
[-a_\mx, +a_\mx] \ni a \mapsto \E_{{a,\veta}} \Big[ \sum_{\nu=0}^{N-1} \{ (\qsig(\tnu))^2 + (\usig(\tnu))^2\}\Delta \tnu\Big];
\]
we denote this function by $\ecost(a)$.

The above definitions and estimates yield:
\begin{equation}\label{eq: ac9}
\begin{split}
    \ecost(a) =& \E_{a,\veta}\Big[ \sum_{\nu=0}^{N-1}\{ (\qsig(\tnu))^2 + (\usig(\tnu))^2\}\Delta \tnu \cdot \mathbbm{1}_{(\qsig(t_1),\dots, \qsig(t_N))\in E}\Big]\\
    &+\erroro(a),
\end{split}
\end{equation}
with
\begin{equation}\label{eq: ac10}
    \begin{split}
        |\erroro(a)| \le & \E_{a,\veta}[ C \max_\nu \{ |\qsig(\tnu)|^2 + |\usig(\tnu)|^2\}\cdot\mathbbm{1}_{(\qsig(t_1),\dots,\qsig(\tnu))\notin E}]\\
        \le & \Big( \E_{a,\veta} [C \max\{ |\qsig(\tnu)|^2 + |\usig(\tnu)|^2\}^2])^{1/2}\\
        &\cdot \Big( \prob_{a,\veta} ((\qsig(t_1),\dots,\qsig(t_N))\notin E)\Big)^{1/2}\\
        \le & C' \exp(-c' (\Delta t_\mx)^{-1/8}).
    \end{split}
\end{equation}

Moreover,
\begin{align*}
    \E_{a,\veta}\Big[ \sum_\nu & \{ (\qsig(\tnu))^2 + (\usig(\tnu))^2\} \Delta \tnu \cdot \mathbbm{1}_{(\qsig(t_1),\dots,\qsig(\tnu))\in E} \Big]\\
    = &\int_{(\barq_1,\dots,\barq_N) \in E} \Big(\sum_\nu \{ \barq_\nu^2 + \baru_\nu^2\}\Delta \tnu\Big) \Phi(\barq_1, \dots, \barq_N,a) \ d\barq_1\cdots d\barq_N\\
    = &\int_{(\barq_1,\dots,\barq_N)\in E} \Big\{\Big( \sum_\nu \{ \barq_\nu^2 + \baru_\nu^2\}\Delta \tnu\Big) \Psi(\barq_1,\dots,\barq_N,a)\\
    & \quad \cdot (1 + \err(\barq_1,\dots,\barq_N,a))\ d\barq_1\cdots d\barq_N\Big\}\\
    = & (1 + \errort(a))\cdot \int_{(\barq_1,\dots,\barq_N)\in E} \Big\{\Big( \sum_\nu \{ \barq_\nu^2 + \baru_\nu^2\} (\Delta \tnu)\Big) \\ &
    \qquad \qquad \qquad \qquad\; \cdot \Psi(\barq_1,\dots,\barq_N,a) \ d\barq_1\cdots d\barq_N\Big\},
\end{align*}
with
\begin{equation}\label{eq: ac11}
    |\errort(a)| \le C  ( \Delta t_\mx)^{1/8},
\end{equation}
thanks to \eqref{eq: ac4}, \eqref{eq: ac5}. Together with \eqref{eq: ac9}, \eqref{eq: ac10}, this yields
\begin{align*}
\ecost(a) =& \erroro(a) \\
&+ (1+\errort(a))\cdot \int_E \Big\{ \Big(\sum_\nu \{\barq_\nu^2 + \baru_\nu^2\} \Delta \tnu\Big) \\
&\qquad \qquad\qquad \qquad\qquad \;\; \;\cdot \Psi(\barq_1,\dots,\barq_N,a)\Big\}\ d\barq_1\cdots d\barq_N
\end{align*}
with $\erroro(a)$, $\errort(a)$ controlled by \eqref{eq: ac10}, \eqref{eq: ac11}. Since also
\[
0 \le \ecost(a) \le C
\]
by Lemma \ref{lem: rare events}, it follows that
\begin{equation}\label{eq: ac12}
\begin{split}
    \ecost(a) = & \int_{(\barq_1, \dots, \barq_N) \in E}\Big\{ \Big( \sum_\nu \{\barq_\nu^2 + \baru_\nu^2\} \Delta \tnu\Big)\\ &\quad \cdot \Psi(\barq_1,\dots,\barq_N,a)\Big\} \;d\barq_1\cdots d\barq_N + \errorth(a),
    \end{split}
\end{equation}
with
\begin{equation}\label{eq: ac13}
    |\errorth(a)|\le C (\Delta t_\mx)^{1/8}.
\end{equation}
Equation \eqref{eq: ac12} and estimate \eqref{eq: ac13} hold for $a \in [-a_\mx,+a_\mx]$. We make an analytic continuation of the integral in \eqref{eq: ac12} from $a \in [-a_\mx,+a_\mx]$ to $a = a_R + ia_I$ in the rectangle
\[
\cR = \{ a_R + i a_I : a_R \in [-a_\mx, a_\mx], a_I \in [-\delta, \delta]\}
\]
for a small $\delta>0$ to be picked below.

We write $I(a)$ to denote the integral in \eqref{eq: ac12}. Thus,
\begin{equation}\label{eq: ac14}
    I(a) = \int_E \Big( \sum_\nu \{ \barq_\nu^2 + \baru_\nu^2\} \Delta \tnu\Big) \Psi(\barq_1,\dots,\barq_N,a) \ d\barq_1\cdots d\barq_N
\end{equation}
for $a = a_R + i a_I \in \cR$, and
\begin{equation}\label{eq: ac15}
    |\ecost(a) - I(a)| \le C (\Delta t_\mx)^{1/8}\;\text{for} \; a \in [-a_\mx, + a_\mx].
\end{equation}

A glance at \eqref{eq: ac1}, \eqref{eq: ac2} shows that the integrand in \eqref{eq: ac14} has the form
\[
B(\barq_1,\dots,\barq_N) \exp(a^2 G_2(\barq_1,\dots,\barq_N) + a G_1(\barq_1, \dots, \barq_N) + G_0(\barq_1,\dots,\barq_N)),
\]
where $ B, G_0, G_1, G_2$ are bounded measurable functions of $(\barq_1, \dots, \barq_N)$ on $E$. Moreover, the region of integration, $E$, is bounded; see \eqref{eq: ac3}. Therefore, $I(a)$ is an analytic function on $\cR$.

Next, we estimate
\begin{equation}\label{eq: ac17}
    \int_E \Big( \sum_\nu \{ \barq_\nu^2 + \baru_\nu^2\}\Delta t_\nu \Big) |\Psi(\barq_1,\dots,\barq_N, a_R + i a_I)|\; d\barq_1\cdots d\barq_N
\end{equation}
for $a_R + i a_I \in \cR$. From \eqref{eq: ac1} we have
\[
|\psi_\nu(\Delta \barq, \barq, \baru, a_R + i a_I)| = \exp(\frac{a_I^2}{2}\barq^2 \Delta \tnu) \psi_\nu(\Delta \barq, \barq, \baru, a_R),
\]
hence \eqref{eq: ac2} implies that
\[
|\Psi(\barq_1,\dots,\barq_N, a_R + ia_I)| = \exp\Big(\sum_\nu \frac{a_I^2}{2} \barq_\nu^2 \Delta \tnu\Big) \Psi(\barq_1,\dots,\barq_N,a_R).
\]
Moreover, for $(\barq_1, \dots, \barq_N) \in E$ and $a_R \in [-a_\mx,+a_\mx]$, \eqref{eq: ac4} and \eqref{eq: ac5}  yield
\[
\Psi(\barq_1,\dots,\barq_N,a_R) \le 2 \Phi(\barq_1, \dots, \barq_N,a_R).
\]
Consequently, the integrand in \eqref{eq: ac17} is at most
\[
2\Big( \sum_\nu \{ \barq_\nu^2 + \baru_\nu^2\} \Delta \tnu\Big)\exp\Big(\frac{a_I^2}{2}\sum_\nu \barq_\nu^2 \Delta \tnu\Big) \Phi(\barq_1,\dots,\barq_N,a_R) 
\]
for $a_R \in [-a_\mx, a_\mx]$. Since $\Phi(\barq_1,\dots,\barq_N,a_R)$ is the probability density for $(\qsig(t_1), \dots, \qsig(t_N))$ assuming $a_\tru = a_R$ and $\vxi = \veta$, it follows that the integral \eqref{eq: ac17} is at most
\begin{multline*}
    2 \E_{a_R, \veta} \Big[ \mathbbm{1}_{(\qsig(t_1),\dots,\qsig(t_N)) \in E} \cdot \Big( \sum_\nu \{ \qsig(\tnu))^2 + (\usig(\tnu))^2\} \Delta \tnu \Big)\\ \cdot \exp\Big( \frac{a_I^2}{2} \sum_\nu (\qsig(\tnu))^2 \Delta \tnu\Big) \Big]\;\text{for}\; a \in \cR.
\end{multline*}
Recall that $|\usig(\tnu)| \le C [|\qsig(\tnu)|+1]$ and that $|a_I|\le \delta$ for $a = a_R + i a_I \in \cR$.

Consequently, for $a_R + i a_I \in \cR$, we have
\begin{multline*}
    \Big( \sum_\nu \{ (\qsig(\tnu))^2 + (\usig(\tnu))^2\} \Delta \tnu\Big) \exp\Big( \frac{a_I^2}{2} \sum_\nu (\qsig(\tnu))^2 \Delta \tnu\Big)\\
    \le C_\delta \cdot \exp\Big( C \delta^2 \max_\nu | \qsig(\tnu)|^2\Big).
\end{multline*}
So the integral \eqref{eq: ac17} is at most
\begin{equation}\label{eq: ac18}
    C_\delta \E_{a_R,\veta} \Big[ \exp \Big( C \delta^2 \max_\nu | \qsig(t_\nu)|^2\Big)\Big]\;\text{for}\; a \in \cR.
\end{equation}
On the other hand, Lemma \ref{lem: rare events} gives
\begin{equation}\label{eq: ac19}
    \E_{a_R, \veta} [ \exp(c \max_\nu | \qsig(\tnu)|^2)] \le C
\end{equation}
for $a_R \in [-a_\mx,a_\mx]$ and $c>0$ small enough.

We now fix $\delta = \hat{c}$ small enough that $C\delta^2 < c$ with $C,c$ as in \eqref{eq: ac18}, \eqref{eq: ac19}. We conclude that the integral \eqref{eq: ac17} is less than a large constant $C$, independent of $a \in \cR$. 

So we have shown that $I(a)$ is analytic and bounded for 
\[
a \in (-a_\mx, + a_\mx) \times (-\delta, \delta).
\]
Recalling that we have taken $\delta = \hat{c}$ and that \eqref{eq: ac15} holds, we obtain the following result.

\begin{lem}[Analytic Continuation Lemma]\label{lem: ac}
    Let $\sigma$ be a tame strategy, and let $\veta \in \{0,1\}^\N$.

    Then there exists an analytic function $I_{\veta}(a)$ on the rectangle
    \[
    \cR = \{a_R + ia_I : a_R \in (-a_\emx, a_\emx), a_I \in (-\hat{c},\hat{c})\}
    \]
    such that
    \[
    |I_{\veta}(a)| \le C \;\text{on}\; \cR
    \]
    and
    \[
    \Big| \eE_{a,\veta} \Big[ \sum_{\nu=0}^{N-1} \{ (\qsig(\tnu))^2 + (\usig(\tnu))^2\} \Delta \tnu\Big] - I_{\veta}(a) \Big| \le C (\Delta t_\emx)^{1/8}
    \]
    for $ a \in (-a_\emx, + a_\emx)$. Explicitly,
    \begin{align*}
    I_{\veta}(a) =& \int_E \Big( \sum_{\nu=0}^{N-1}\{\qnu^2 + \unu^2\} \Delta \tnu\Big)\\& \cdot \prod_{\nu=0}^{N-1} \Big\{ \frac{1}{\sqrt{2\pi \Delta \tnu}} \exp \Big(  - \frac{[q_{\nu+1}-\qnu - (a\qnu + \unu)\Delta \tnu]^2}{2 \Delta \tnu} \Big) \Big\}\ dq_1 \cdots dq_N
    \end{align*}
    where
    \begin{equation}
\begin{split}
    E  = \{(\barq_1,\dots,\barq_N) \in \R^N: & \max_\nu |\barq_\nu| \le (\Delta t_\emx)^{-1/16}\\ &\emph{and}\; |\sum_\nu (\barq_{\nu+1}-\barq_\nu)^2 - T| \le (\Delta t_\emx)^{1/8}\}.
\end{split}
\end{equation}
and $\unu$ denotes the value assigned by $\sigma$ to the control at time $\tnu$ assuming that $\qsig(t_\mu) = q_\mu$ for $0 \le \mu \le \nu$ and $\vxi = \veta$.
\end{lem}

\section{Moments of Increments}
We retain the assumptions and notation of Section \ref{sec: rare events}. Recall that $\cF_\nu$ is the sigma algebra of events determined by $\qsig(t_\mu)$ ($0 \le \mu \le \nu)$, and that
\[
\begin{aligned}
  &\Delta \tnu = t_{\nu+1} - \tnu, & &\Delta \qnu = \qsig(t_{\nu+1}) - \qsig(\tnu),\\
  & \Delta \zonu = \zosig(t_{\nu+1}) - \zosig(\tnu), & &\Delta \ztnu = \ztsig(t_{\nu+1}) - \ztsig(\tnu).
\end{aligned}
\]
We suppose that $a_\tru =  a$ and $\vxi = \veta$.

Recall also that
\begin{flalign}
    &\Delta \qnu = (a \qnu + \unu) \Delta \tnu^* + \Delta W_\nu,\label{eq: moi 1}&\\
    &\Delta \zonu = \qnu (\Delta \qnu - \unu \Delta \tnu) = a \qnu^2 \Delta \tnu^* + \qnu \unu (\Delta \tnu^* - \Delta \tnu) + \qnu \Delta W_\nu,\label{eq: moi 2}\\
    &\Delta \ztnu = \qnu^2 \Delta \tnu. \label{eq: moi 3}
\end{flalign}

We condition on $\cF_\nu$; thus, $\qnu$ and $\unu$ are deterministic, while $\Delta W_\nu$ is normal, with mean 0 and variance $\Delta \tilde{t}_\nu$. We suppose that
\begin{equation}\label{eq: moi 4}
    |\qnu|, |\unu| \le Q, \;\text{with}\; Q\ge C \;\text{given}.
\end{equation}
Then
\begin{flalign*}
    & \Delta \qnu = O(Q) \Delta \tnu + \Delta W_\nu,&\\
& \Delta \zonu = O(Q^2) \Delta \tnu + \qnu \Delta W_\nu,\\
& \Delta \ztnu = O(Q^2) \Delta \tnu,
\end{flalign*}
so that
\begin{flalign}
    & (\Delta \qnu)^2 = O(Q^2) (\Delta \tnu)^2 + 2 O(Q) \Delta \tnu \Delta W_\nu + (\Delta W_\nu)^2,\label{eq: moi 5} &\\
    & (\Delta \zonu)(\Delta \qnu) = O(Q^3) (\Delta \tnu)^2 + O(Q^2) \Delta \tnu \Delta W_\nu + \qnu (\Delta W_\nu)^2,\label{eq: moi 6}\\
    & (\Delta \ztnu) (\Delta \qnu) = O(Q^3)(\Delta \tnu)^2 + O(Q^2) \Delta \tnu \Delta W_\nu,\label{eq: moi 7}\\
    & (\Delta \zonu)^2 =O(Q^4)(\Delta \tnu)^2 + O(Q^3) (\Delta \tnu) (\Delta W_\nu) + \qnu^2 (\Delta W_\nu)^2,\label{eq: moi 8}\\
& (\Delta \ztnu) (\Delta \zonu) = O(Q^4)(\Delta \tnu)^2 + O(Q^3)(\Delta \tnu)(\Delta W_\nu)\label{eq: moi 9}\\
& (\Delta \ztnu)^2 = O(Q^4)(\Delta \tnu)^2.\label{eq: moi 10}
\end{flalign}
Moreover, all the quantities $O(Q^\text{power})$ above are deterministic once we condition on $\cF_\nu$.

Therefore, \eqref{eq: moi 1}--\eqref{eq: moi 3} and \eqref{eq: moi 5}--\eqref{eq: moi 10} yield the following.
\begin{flalign}
    & \E[\Delta \qnu | \cF_\nu] = (a\qnu + \unu) \Delta \tnu + O(Q(\Delta \tnu)^2),\label{eq: moi 11} &\\
    & \E[\Delta \zonu | \cF_\nu] = a\qnu^2(\Delta \tnu) + O(Q^2 (\Delta \tnu)^2),\label{eq: moi 12}\\
    & \E[ \Delta \ztnu | \cF_\nu] = \qnu^2 \Delta \tnu, \label{eq: moi 13}\\
    & \E[ (\Delta \zonu)^2 | \cF_\nu] = \qnu^2 \Delta \tnu + O(Q^4 (\Delta \tnu)^2),\label{eq: moi 14}\\
    & \E[(\Delta \zonu) (\Delta \ztnu) | \cF_\nu ] = O(Q^4 (\Delta \tnu)^2),\label{eq: moi 16},\\
    & \E[(\Delta \ztnu)^2 | \cF_\nu] = O(Q^4 (\Delta \tnu)^2),\label{eq: moi 17}\\
    & \E[ (\Delta \qnu)^2 | \cF_\nu] = \Delta \tnu + O(Q^2 (\Delta \tnu)^2), \label{eq: moi 18}\\
    & \E[(\Delta \qnu) (\Delta \zonu) | \cF_\nu] = \qnu \Delta \tnu + O(Q^3 (\Delta \tnu)^2), \label{eq: moi 19}\\
    & \E[(\Delta \qnu) (\Delta \ztnu) | \cF_\nu] = O(Q^3 (\Delta \tnu)^2).\label{eq: moi 20}
\end{flalign}
(Here we have used the fact that $\Delta \tnu^*$ and $\Delta \tilde{t}_\nu$ are $\Delta \tnu + O((\Delta \tnu)^2).$)

We define an event
\begin{multline}
    \tame_\nu = \{ |\Delta \qnu | \le 2 (\Delta \tnu)^{2/5},  |\Delta \zonu| \le 2 (\Delta \tnu)^{2/5}, \\
    | \Delta \ztnu| \le 2 (\Delta \tnu)^{2/5}\}.
\end{multline}
From Lemma \ref{lem: rare events} we obtain the estimate
\[
\E[(\Delta \qnu)^{\alpha_0}(\Delta \zonu)^{\alpha_1}(\Delta \ztnu)^{\alpha_2} \cdot \mathbbm{1}_{\nottame_\nu} | \cF_\nu] = O((\Delta \tnu)^{100})
\]
for integers $\alpha_0,\alpha_1, \alpha_2 \ge 0 $ with $\alpha_0 + \alpha_1 + \alpha_2 \le 2$. Also from Lemma \ref{lem: rare events}, we recall that
\[
\prob[\nottame_\nu | \cF_\nu] \le C \cdot (\Delta \tnu)^{1000}.
\]
Consequently, \eqref{eq: moi 11}--\eqref{eq: moi 20} imply the conclusions of the following lemma.

\begin{lem}[Lemma on Moments of Increments]\label{lem: moments}
    We fix $a_\etru = a$ and $\vxi = \veta$.

    We suppose that
    \[
    Q \ge C \text{ and } (\Delta t_\emx) < Q^{-1000}.
    \]
    Define the event
    \[
    \etame(\nu) = \{ |\Delta \qnusig| \le 2 (\Delta \tnu)^{2/5}, |\Delta \zonusig| \le 2 (\Delta \tnu)^{2/5}, | \Delta \ztnusig| \le 2 (\Delta \tnu)^{2/5}\}.
    \]
    Let $\cF_\nu$ be the sigma algebra of events determined by $\qsig(t_\mu)$ ($0 \le \mu \le \nu)$.

    Fix $\nu$, and suppose that
    \[
    |\qsig(t_\nu)|, |\usig(t_\nu)| \le Q.
    \]
    Then the following hold.
    \begin{flalign*}
        &\eE[(\Delta \qnusig) \mathbbm{1}_{\etame(\nu)} | \cF_\nu] = ( a \qnusig + \unusig)(\Delta \tnu) + \eerr\; 1, &\\
        & \eE[(\Delta \zonusig) \mathbbm{1}_{\etame(\nu)}|\cF_\nu] = a(\qnusig)^2(\Delta \tnu) + \eerr\; 2, \\
        & \eE[(\Delta \ztnusig)\mathbbm{1}_{\etame(\nu)}|\cF_\nu] = (\qnusig)^2(\Delta \tnu) + \eerr\; 3\\
        & \eE[(\Delta \qnusig)^2 \mathbbm{1}_{\etame(\nu)}|\cF_\nu] = (\Delta \tnu) + \eerr\; 4,\\
        &\eE[(\Delta \qnusig)(\Delta \zonusig)\mathbbm{1}_{\etame(\nu)}|\cF_\nu] = \qnusig(\Delta \tnu) + \eerr\; 5,\\
        &\eE[(\Delta \qnusig) (\Delta \ztnusig) \mathbbm{1}_{\etame(\nu)}|\cF_\nu] = \eerr\; 6,\\
        & \eE[(\Delta \zonusig)^2 \mathbbm{1}_{\etame(\nu)}|\cF_\nu] = (\qnusig)^2(\Delta \tnu) + \eerr \; 7,\\
        & \eE[(\Delta \zonusig)(\Delta \ztnusig)\mathbbm{1}_{\etame(\nu)}|\cF_\nu] = \eerr\; 8,\\
        & \eE[(\Delta \ztnusig)^2 \mathbbm{1}_{\etame(\nu)}|\cF_\nu] = \eerr\; 9,
    \end{flalign*}
    where
    \[
    |\eerr\; 1|, \dots, |\eerr\; 9| \le C' Q^4 (\Delta \tnu)^2.
    \]
    Also, under the above assumptions, we have
    \[
    \eprob[\enottame(\nu)|\cF_\nu] \le (\Delta \tnu)^{20}.
    \]
    Here, of course,
    \[
    \begin{aligned}
        & \qnusig = \qsig(\tnu),& & \unusig = \usig(\tnu),&\\
        & \Delta \tnu = t_{\nu+1} - \tnu,& & \Delta \qnusig = \qsig(t_{\nu+1}) - \qsig(t_\nu),&\\ & \Delta \zonusig = \zosig(t_{\nu+1}) - \zosig(\tnu),& & \Delta \ztnusig = \ztsig(t_{\nu+1}) - \ztsig(\tnu).
    \end{aligned}
    \]
    The constants $C,C'$ are determined by $q_0, a_\emx, C_\etame$ and an upper bound for $T$.
\end{lem}

\section{Stability Under Change of Assumption}
Let $f(\barq_1,\dots,\barq_N)$ be a nonnegative function on $\R^N$. For $\veta \in \{0,1\}^\N$ and $a_1,a_2 \in [-a_\mx, +a_\mx]$, we compare
\[
\E_{a_1,\veta}[ f(\qsig(t_1), \dots, \qsig(t_N))]
\]
with
\[
\E_{a_2,\veta}[f(\qsig(t_1),\dots,\qsig(t_N))]
\]
for a tame strategy $\sigma$.

To do so, let $\Phi(\barq_1,\dots, \barq_N,a)$ denote the probability density of 
\[
(\qsig(t_1),\dots,\qsig(t_N))
\]
assuming that $\vxi = \veta$ and $a_\tru = a$.

According to Lemma \ref{lem: prob dist}, the following holds for $a = a_1, a_2$:

Let $Q>C$ and $Q\le (\Delta t_\mx)^{-1/1000}$, and suppose that
\begin{equation}\label{eq: suca 1}
\begin{split}
& \max_\nu(|\barq_\nu| + |\baru_\nu|) \le Q, \\
    &\Big| \sum_{0 \le \nu < N} (\barq_{\nu+1} - \barq_\nu)^2 - T \Big| \le Q^2 (\Delta t_\mx)^{1/4},
\end{split}
\end{equation}
where $\baru_\nu$ denotes the value of $\usig(\tnu)$ assuming $\qsig(t_\mu) = \barq_\mu$ for $\mu \le \nu$. (Recall that the $\baru_\nu$ don't depend on $a$.) Then
\begin{equation}\label{eq: suca 2}
\begin{split}
    \Phi(\barq_1,\dots,\barq_N,a&) \\= \prod_{0 \le \nu < N} \Big\{&\frac{1}{\sqrt{2\pi\Delta \tnu}} \exp\Big( - \frac{1}{2\Delta \tnu}[(\barq_{\nu+1}-\barq_\nu) - (a\barq_\nu + \baru_\nu) \Delta \tnu]^2\Big)\Big\}\\ &\cdot ( 1 + \err(\barq_1,\dots, \barq_N, a)
\end{split}
\end{equation}
with
\begin{equation}\label{eq: suca 3}
    |\err(\barq_1,\dots,\barq_N,a)| \le C Q^2 (\Delta t_\mx)^{1/4}
\end{equation}
and $\barq_0 \equiv q_0$.

Applying \eqref{eq: suca 2}, \eqref{eq: suca 3} with $a=a_1$ and with $a=a_2$, we see that if \eqref{eq: suca 1} holds, then
\begin{equation}\label{eq: suca 4}
\begin{split}
    \Phi(\barq_1, \dots, \barq_N, a_2) = &\Phi(\barq_1, \dots, \barq_N, a_1) \\ &\cdot \exp\Big( - \frac{1}{2}[a_2^2 - a_1^2] \barztsig(T) + [a_2 - a_1] \barzosig(T)\Big)\\
    &\cdot (1 + \err(\barq_1,\dots,\barq_N,a_1,a_2))
\end{split}
\end{equation}
with
\begin{equation}\label{eq: suca 5}
    \barzosig(T):= \sum_{0 \le \nu < N} \barq_\nu([\barq_{\nu+1} - \barq_\nu] - \baru_\nu \Delta \tnu), \qquad \barztsig(T):= \sum_{0 \le \nu < N} \barq_\nu^2 \Delta \tnu,
\end{equation}
and
\begin{equation}\label{eq: suca 6}
    |\err(\barq_1,\dots,\barq_N,a_1,a_2)| \le C Q^2 (\Delta t_\mx)^{1/4}.
\end{equation}
If \eqref{eq: suca 1} holds, then
\begin{equation}\label{eq: suca 7}
    |\barzosig(T)|, |\barztsig(T)| \le C Q^2
\end{equation}
(see \eqref{eq: 1.27}).

Letting $\cE$ denote the event that $(\qsig(t_1) , \dots, \qsig(t_N))$ satisfies \eqref{eq: suca 1}, we conclude from \eqref{eq: suca 4}, \eqref{eq: suca 6}, \eqref{eq: suca 7} that
\begin{equation}\label{eq: suca 8}
    \begin{split}
        \E_{a_2,\veta}[f(\qsig(t_1),&\dots,\qsig(t_N))\cdot \mathbbm{1}_{\cE}]\\ \le &\E_{a_1,\veta}[f(\qsig(t_1),\dots,\qsig(t_N))\cdot \mathbbm{1}_\cE]\cdot \exp(CQ^2 |a_2-a_1|)\\ & \cdot (1 + \err_f(a_1,a_2))
    \end{split}
\end{equation}
with
\begin{equation}\label{eq: suca 9}
    |\err_f(a_1,a_2)|\le C Q^2(\Delta t_\mx)^{1/4}.
\end{equation}
On the other hand, Lemma \ref{lem: rare events} shows that the complement of $\cE$, denoted $\compE$, satisfies
\[
\prob_{a_2,\veta}[\compE] \le \exp(-cQ^2),
\]
and therefore, by Cauchy-Schwarz, we have
\begin{equation}\label{eq: suca 10}
\begin{split}
    \E_{a_2,\veta}[f(\qsig(t_1),\dots, &\qsig(t_N))\cdot \mathbbm{1}_{\compE}]\\ &\le \exp(-c'Q^2) \Big( \E_{a_2, \veta}[ f^2(\qsig(t_1),\dots,\qsig(t_N))]\Big)^{1/2}.
\end{split}
\end{equation}
Combining \eqref{eq: suca 9} and \eqref{eq: suca 10}, we obtain the following result.

\begin{lem}[Lemma on Change of Assumption]\label{lem: change of assumption}
    Suppose $Q \ge C $ for a large enough constant $C$.

    Let $a_1, a_2 \in [-a_\emx, + a_\emx]$, let $\veta \in \{0,1\}^\N$, and let $f(\barq_1, \dots, \barq_N)$ be a nonnegative function on $\R^\N$. Let $\sigma$ be a tame strategy. Assume that $Q \le (\Delta t_\emx)^{-1/1000}$. Then
    \begin{align*}
        \eE_{a_2, \veta} [ &f(\qsig(t_1), \dots, \qsig(t_N))] \\ \le & \exp(C Q^2 |a_2 - a_1|) (1 + CQ^2 (\Delta t_\emx)^{1/4}) \eE_{a_1, \veta} [ f(\qsig(t_1),\dots, \qsig(t_N))]\\ &+ \exp(-c Q^2) \Big( \eE_{a_2, \veta} [ f^2(\qsig(t_1),\dots, \qsig(t_N))]\Big)^{1/2}.
    \end{align*}
\end{lem}

\section{Disasters Due to Undercontrol}

Our tame strategies $\sigma$ are defined to guarantee that
\begin{equation}\label{eq: du 1}
    |\usig(t_\nu)| \le C [|\qsig(\tnu)| + 1],
\end{equation}
which is reasonable, given our assumption that $a_\tru \in [-a_\mx, a_\mx]$. However, if $a_\tru \gg a_\mx$, then we expect \eqref{eq: du 1} to undercontrol, leading to exponentially large expected cost.

The following lemma confirms that intuition.

\begin{lem}[Lemma on Undercontrol]\label{lem: undercontrol}
    Let $\sigma$ be a deterministic strategy satisfying \eqref{eq: du 1}, and suppose $a_\etru =a $, where $a$ exceeds a large enough constant $C_*$. Write $\eE_a[\dots]$ for the corresponding expectation. Assume that $\Delta t_\emx$ is less than a small enough positive number determined by $a$ and $T$. Then
    \[
    \eE_a \Big[ \sum_{0 \le \nu < N} \{ \qsig(\tnu))^2 + (\usig(\tnu))^2\}\Delta \tnu\Big] \ge c T^2 \exp(c a T).
    \]
\end{lem}
\begin{proof}
    We write $\qnu$ for $\qsig(\tnu)$, $\Delta \qnu$ for $q_{\nu+1}-\qnu$, $\unu$ for $\usig(\tnu)$, and $\Delta \tnu$ for $t_{\nu+1}-\tnu$. We let $\cF_\nu$ denote the sigma algebra of events determined by $q_1,\dots, q_\nu$. Thus, $\qnu$ and $\unu$ are deterministic once we condition on $\cF_\nu$.

    Recall that
    \begin{equation}\label{eq: du 2}
        \Delta \qnu = (a\qnu + \unu ) (\Delta \tnu^*) + \Delta W_\nu
    \end{equation}
    where $\Delta W_\nu$ is normal with mean 0 and variance
    \[
    \Delta \tilde{t}_\nu = \frac{\exp(2a\Delta \tnu) - 1}{2a} ;
    \]
    moreover $\Delta W_\nu$ is independent of $\cF_\nu$. Here,
    \[
    \Delta \tnu^* = \frac{\exp(a\Delta \tnu)-1}{a}.
    \]
    Since $\Delta \tnu < \Delta t_\mx$ is less than a small enough positive number determined by $a$ and $T$, we have
    \begin{equation}\label{eq: du 3}
        |\Delta \tilde{t}_\nu - \Delta t_\nu|, |\Delta \tnu^* - \Delta \tnu| < 10^{-3} \Delta \tnu.
    \end{equation}

From \eqref{eq: du 2} we have
\begin{align*}
    q_{\nu+1}^2 = & \qnu^2 + 2\qnu [ (a\qnu + \unu) (\Delta \tnu^*) + \Delta W_\nu] + (a\qnu + \unu)^2(\Delta \tnu^*)^2\\
    & + 2(a\qnu + \unu)(\Delta \tnu^*) \Delta W_\nu + (\Delta W_\nu)^2\\
    \ge & \qnu^2(1+2a\Delta \tnu^*) + 2 \qnu \unu (\Delta \tnu^*) + [2\qnu + 2(a\qnu + \unu)(\Delta \tnu^*)] \Delta W_\nu\\
    &+ (\Delta W_\nu)^2.
\end{align*}
Consequently,
\begin{equation}\label{eq: du 4}
    \E_a[q_{\nu+1}^2 | \cF_\nu] \ge \qnu^2(1+2a\Delta \tnu^*) + 2\qnu \unu(\Delta \tnu^*) + (\Delta \tilde{t}_\nu).
\end{equation}
For $\delta>0$ to be picked in a moment, we have
\[
|2\qnu \unu| \le \delta^{-2} \qnu^2 + \delta^2 \unu^2 \le C \delta^{-2} \qnu^2 + C \delta^2,
\]
thanks to \eqref{eq: du 1}.

Putting this inequality into \eqref{eq: du 4}, we find that
\begin{equation}\label{eq: du 5}
    \E_a[q_{\nu+1}^2| \cF_\nu] \ge \qnu^2(1+[2a-C\delta^{-2}]\Delta \tnu^*) + (\Delta \tilde{t}_\nu - C \delta^2 \Delta \tnu^*).
\end{equation}
We take $\delta$ to be a small enough constant $c$ such that
\[
\Delta \tilde{t}_\nu - C \delta^2 \Delta \tnu^* \ge \frac{1}{2} \Delta \tnu;
\]
see \eqref{eq: du 3}. Since $a$ exceeds a large enough constant $C$, we then have
\[
[2a-C\delta^{-2}] > a,
\]
so from \eqref{eq: du 5} we obtain
\[
\E_a[q_{\nu+1}^2 | \cF_\nu] \ge \qnu^2(1+a\Delta \tnu^*) + \frac{1}{2}\Delta \tnu.
\]
Again applying \eqref{eq: du 3}, and recalling that $\Delta \tnu < \Delta t_\mx$ is less than a small enough positive nuumber determined by $a$ and $T$, we conclude that
\[
\E_a[q_{\nu+1}^2 | \cF_\nu] \ge \exp\Big( \frac{1}{2} a \Delta \tnu\Big) \qnu^2 + \frac{1}{2} \Delta \tnu,
\]
and therefore
\begin{equation}\label{eq: du 6}
    \E_a[q_{\nu+1}^2] \ge \exp\Big( \frac{1}{2} a \Delta \tnu\Big) \E_a[\qnu^2] + \frac{1}{2}\Delta \tnu.
\end{equation}
Since \eqref{eq: du 6} implies that
\[
\E_a[q_{\nu+1}^2] \ge \E_a [\qnu^2] + \frac{1}{2} \Delta \tnu \; \text{for each}\; \nu,
\]
we conclude that
\[
\E_a[\qnu^2] \ge \frac{1}{2} \tnu\;\text{for each}\; \nu.
\]
We pick $\nu_0$ so that
\[
\frac{1}{2}T < t_{\nu_0} < \frac{2}{3}T.
\]
(Our smallness assumption on $\Delta t_\mx$ implies that such a $\nu_0$ exists.) Then 
\[
\E_a[q_{\nu_0}^2] \ge \frac{1}{2} t_{\nu_0} > \frac{1}{4} T.
\]
Returning to \eqref{eq: du 6}, we have
\[
\E_a[q_{\nu+1}^2] \ge \exp\Big(\frac{1}{2} a \Delta \tnu\Big) \E_a[\qnu^2]\; \text{for each}\; \nu,
\]
hence for $\nu \ge \nu_0$ we have
\[
\E_a[\qnu^2] \ge \exp\Big( \frac{1}{2}a [t_\nu - t_{\nu_0}]\Big) \E[q_{\nu_0}^2] \ge \frac{1}{4}T\exp\Big(\frac{1}{2} a [t_\nu - t_{\nu_0}]\Big).
\]
In particular,
\[
\E_a[\qnu^2] \ge \frac{1}{4}T\exp(caT)\;\text{for}\; \tnu \in \Big[\frac{3}{4}T,T\Big].
\]
Consequently,
\begin{align*}
    \E_a\Big[ \sum_{0 \le \nu < N} \{ q_\nu^2 + \unu^2\} \Delta \tnu\Big] &\ge \E_a \Big[ \sum_{\tnu \in [\frac{3}{4}T,T]}\qnu^2 \Delta \tnu\Big]\\ &\ge \frac{1}{4}T \exp(ca T) \cdot \sum_{\tnu \in [ \frac{3}{4}T,T]} \Delta \tnu\\ &\ge c T^2 \exp(ca T),
\end{align*}
since each $\Delta \tnu < \Delta t_\mx$ is less than a small enough positive number determined by $a$ and $T$. The proof of the lemma is complete.
\end{proof}

\section{Costing by Integrals}\label{sec: costing by integrals}
Let $\sigma$ be a deterministic tame strategy. We condition on 
\[
a_\tru = a \in [-a_\mx, +a_\mx],
\]
and write $\prob_a[\cdot]$ and $E_a[\cdot]$ to denote the corresponding probability and expectation.

We want to compare
\[
\cost(\sigma) = \int_0^T \{(\qsig(t))^2 + (\usig(t))^2\} \ dt
\]
with
\[
\cost_D(\sigma) = \sum_{0 \le \nu < N}\{ (\qsig(\tnu))^2 + (\usig(\tnu))^2\} \ \Delta \tnu
\]
(``$D$'' for ``discrete'').

\begin{lem}[Lemma on Costing by Integrals]\label{lem: costing by integrals}
For any $m \ge 1$, we have
\[
\eE_a[ | \cost(\sigma) - \cost_D(\sigma)|^m] \le C_m (\Delta t_\emx)^{m/2}.
\]
\end{lem}
\begin{proof}
    Recall that $\usig(t) = \usig(\tnu)$ for $t \in [\tnu, t_{\nu+1}]$. Hence, 
    \begin{align*}
        \cost(\sigma) - &\cost_D(\sigma) = \sum_{0 \le \nu < N} \int_{\tnu}^{t_{\nu+1}}\{ (\qsig(t))^2 - (\qsig(\tnu))^2\}\ dt\\
        &= \sum_{0 \le \nu < N} \int_{\tnu}^{t_{\nu+1}} \{ [ \qsig(t) - \qsig(\tnu)]^2 + 2\qsig(\tnu)[\qsig(t) - \qsig(\tnu)]\}\ dt.
    \end{align*}

Setting
\[
\osc(\nu) = \max_{t \in [\tnu, t_{\nu+1}]} | \qsig(t) - \qsig(\tnu)|,
\]
we therefore have
\[
|\cost(\sigma) - \cost_D(\sigma)| \le \sum_{0 \le \nu < N}\{ (\osc(\nu))^2 + 2 | \qsig(\tnu)| (\osc(\nu))\} \Delta \tnu,
\]
so that by H\"older's inequality,
\begin{multline*}
|\cost(\sigma) - \cost_D(\sigma)|^m\\ \le C_m \sum_{0 \le \nu < N} \{ (\osc(\nu))^{2m} + |\qsig(\tnu)|^m (\osc(\nu))^m\} \Delta \tnu.
\end{multline*}
Consequently,
\begin{equation}\label{eq: ci 1}
\begin{split}
    \E_a[|\cost(\sigma) -& \cost_D(\sigma)|^m]\\ \le& C_m\sum_{0 \le \nu < N} \E_a[(\osc(\nu)^{2m}](\Delta \tnu) \\
    &+ C_m \sum_{0 \le \nu < N} (\E_a[|\qsig(\tnu)|^{2m}])^{1/2}(\E_a[(\osc(\nu))^{2m}])^{1/2} \Delta \tnu.
\end{split}
\end{equation}

We now estimate the right-hand side of \eqref{eq: ci 1}.

Fix a large enough constant $C_*$, and let $\cF_\nu$ denote the sigma algebra of events determined by $\qsig(t_\mu)$ for $\mu =1,\dots,\nu$.

Lemma \ref{lem: rare events} shows that
\begin{equation}\label{eq: ci 2}
    \E_a[|\qsig(\tnu)|^{2m}] \le C_m
\end{equation}
and that, given $Q_2 \ge Q_1 \ge C_*$, we have
\begin{equation}\label{eq: ci 3}
    \prob_a [ \osc(\nu) > Q_2 (\Delta \tnu)^{1/2} | \cF_\nu] \le C \exp(-c Q_2^2) \;\text{if}\; |\qsig(\tnu)| \le Q_1.
\end{equation}
From \eqref{eq: ci 3} we see that
\[
\E_a[ (\osc(\nu))^{2m} | \cF_\nu] \le C_m Q_1^{2m} (\Delta \tnu)^m\;\text{if}\; |\qsig(\tnu)| \le Q_1.
\]
In particular,
\begin{equation}\label{eq: ci 4}
    \E_a[(\osc(\nu))^{2m} \cdot \mathbbm{1}_{|\qsig(\tnu)| \le C_*}] \le C_m (\Delta \tnu)^m
\end{equation}
and, for $k \ge 0$,
\begin{multline}\label{eq: ci 5}
    \E_a[(\osc(\nu))^{2m}\cdot \mathbbm{1}_{|\qsig(\tnu)| \in [C_* 2^k, C_* 2^{(k+1)})}]\\ \le C_m \cdot (C_* 2^{(k+1)})^{2m}(\Delta \tnu)^m \cdot \prob_a [ |\qsig(\tnu)| > C_* 2^k].
\end{multline}
Another application of Lemma \ref{lem: rare events} gives
\[
\prob_a [ |\qsig(\tnu)| > C_* 2^k] \le C \exp(-c 2^{2k}),
\]
so that \eqref{eq: ci 5} implies that
\[
\E_a[ (\osc(\nu))^{2m}\cdot \mathbbm{1}_{|\qsig(\tnu)| \in [C_* 2^k, C_* 2^{(k+1)})}] \le C_m \cdot (2^{2mk}) (\Delta \tnu)^m \cdot \exp(-c2^{2k}).
\]
Summing over $k\ge 0$, and combining the result with \eqref{eq: ci 4}, we learn that
\begin{equation}\label{eq: ci 6}
    \E_a [ (\osc(\nu))^{2m}] \le C_m (\Delta \tnu)^m.
\end{equation}
From \eqref{eq: ci 1}, \eqref{eq: ci 2}, and \eqref{eq: ci 6} we conclude that
\[
\E_a[ | \cost(\sigma) - \cost_D (\sigma)|^m] \le C_m \sum_{0 \le \nu < N}(\Delta \tnu)^{m/2 + 1} \le C_m' (\Delta t_\mx)^{m/2},
\]
which is the conclusion of the lemma.
\end{proof}

\section{Continuous vs Discrete}\label{sec: continuous vs discrete}

Let $\sigma$ be a tame strategy, possibly depending on coin flips $\vxi$. In this section we compare the following random functions of time.
\begin{flalign*}
    &\qcsig(t) = \qsig(t)\;\text{for}\; t\in [0,T]&\\
    & \qdsig(t) = \qsig(\tnu)\;\text{for}\; t\in[\tnu, t_{\nu+1}), \;\text{each}\; \nu < N.\\
    &\zocsig(t) = \frac{1}{2}(\qsig(t))^2 - \frac{1}{2}q_0^2 -\frac{1}{2}t - \int_0^t \usig(s)\qsig(s)\ ds\;\text{for}\; t \in [0,T].\\
    &\zodsig(t) = \frac{1}{2}(\qsig(\tnu))^2 - \frac{1}{2} q_0^2 - \frac{1}{2}\sum_{0 \le \mu < \nu} (\Delta \qsig(t_\mu))^2 -  \sum_{0 \le \mu < \nu} \usig(t_\mu) \qsig(t_\mu)\ \Delta t_\mu\\ & \qquad \qquad \text{for}\; t \in [t_\nu, t_{\nu+1}),\;\text{each} \; \nu < N.\\
    & \ztcsig(t) = \int_0^t (\qsig(s))^2\ ds \;\text{for}\; t \in [0,T].\\
    & \ztdsig(t) = \sum_{0 \le \mu < \nu} (\qsig(t_\mu))^2 \Delta t_\mu\;\text{for}\; t \in [\tnu, t_{\nu+1}), \; \text{each} \; \nu < N.
\end{flalign*}
We recall that $\usig(t)$ is constant on $[\tnu, t_{\nu+1})$ for each $\nu$, so there is no need to introduce analogous quantities for $\usig$.

We establish the following result.

\begin{lem}[Lemma on Continuous Variants]\label{lem: continuous variants}
    Let $Q$ be greater than a large enough constant $C$. Then there exists an event $\bad(\sigma,Q)$ with the following properties.
    \begin{itemize}
        \item For each $ a \in [-a_\emx, + a_\emx]$ and $\veta \in \{0,1\}^\N$, we have
        \[
        \eprob_{a,\veta}[\bad(\sigma,Q)] \le C \exp(-cQ^2).
        \]
        \item If $\bad(\sigma,Q)$ does not occur, then
        \begin{flalign*}
            &\max_{t \in [0,T)} | \qcsig(t) - \qdsig(t)| \le C Q(\Delta t_\emx)^{1/4}&\\
            & \max_{t \in [0,T)} | \zocsig(t) - \zodsig(t) | \le C Q^2(\Delta t_\emx)^{1/4},\\
            & \max_{t \in [0,T)} | \ztcsig(t) - \ztdsig(t) | \le C Q^2 (\Delta t_\emx)^{1/4}.
        \end{flalign*}
    \end{itemize}
\end{lem}

\begin{proof}
    From Lemma \ref{lem: rare events}, the following have probability at most $C\exp(-cQ^2)$ when conditioned on $a_\tru = a$, $\vxi = \veta$ (any $a \in [-a_\mx, a_\mx]$, $\veta \in \{0,1\}^\N):$
    \begin{flalign}
        & \max_\nu \Big| \sum_{0 \le \mu < \nu} (\Delta \qsig_\mu)^2 - \tnu \Big| > Q^2 (\Delta t_\mx)^{1/2}&\label{eq: cvd 1}\\
        & \max_\mu |\qsig_\mu| > Q\label{eq: cvd 2}\\
        & \max_\mu |\usig_\mu| > Q.\label{eq: cvd 3}
    \end{flalign}
    Moreover, with
    \[
    \osc(\nu) := \max_{t \in [\tnu, t_{\nu+1}]} |\qsig(t) - \qsig(\tnu)|
    \]
    and $\cF_\nu^\sigma$ defined as the sigma algebra of events determined by the $\qsig(t_\mu)$ for $\mu \le \nu$, Lemma \ref{lem: rare events} gives also the following, for each fixed $\nu$.
    \[
    \prob_{a,\veta}[\osc(\nu) > C Q(\Delta \tnu)^{1/2} | \cF_\nu] \le C \exp(-cQ^2)\;\text{if}\; |\qnu| \le Q.
    \]
    This implies that
    \[
    \prob_{a,\veta} [ \osc(\nu) > CQ(\Delta \tnu)^{1/2}\;\text{and}\; |\qnu| \le Q] \le C\exp(-cQ^2).
    \]
    Since also
    \[
    \prob_{a,\veta}[ |\qnu| > Q] \le C \exp(-c Q^2),
    \]
    it follows that
    \[
    \prob_{a,\veta} [ \osc(\nu)>CQ(\Delta \tnu)^{1/2}] \le C \exp(-c Q^2)
    \]
    for each fixed $\nu$, and for all $Q \ge C$.

    Taking $Q(\Delta \tnu)^{-1/4}$ in place of $Q$ here, we find that
    \begin{multline*}
    \prob_{a,\veta}[\osc(\nu) > CQ(\Delta \tnu)^{1/4}]\\ \le C \exp(-cQ^2(\Delta \tnu)^{-1/2}) \le C' \exp(-cQ^2) (\Delta \tnu).
    \end{multline*}
    Summing over $\nu$, we find that the event
    \begin{equation}\label{eq: cvd 4}
        \osc(\nu) > CQ(\Delta \tnu)^{1/4}\;\text{for some}\; \nu,
    \end{equation}
    conditioned on $a_\tru = a$ and $\vxi = \veta$, has probability at most $C \exp(-cQ^2)$.

    We now define $\bad(Q,\sigma)$ to be the event that at least one of the conditions \eqref{eq: cvd 1}--\eqref{eq: cvd 4} holds. Then, as claimed,
    \[
    \prob_{a,\veta} [ \bad(Q,\sigma)] \le C \exp(-cQ^2)
    \]
    for any $a \in [-a_\mx, + a_\mx]$ and any $\veta \in \{0,1\}^\N$.

    We now suppose that $\bad(Q, \sigma)$ does not occur, and compare $\qcsig(t)$ with $\qdsig(t)$, $\zocsig(t)$ with $\zodsig(t)$, and $\ztcsig(t)$ with $\ztdsig(t)$.

    Since $\bad(Q,\sigma)$ does not occur, we have
    \begin{flalign}
        &\max_\nu | \qsig(\tnu) |, \max_\nu |\usig(\tnu)| \le Q,&\label{eq: cvd 5}\\
        & \max_\nu \Big| \sum_{0 \le \mu < \nu} (\Delta \qsig_\nu)^2 - \tnu \Big| \le Q^2 (\Delta t_\mx)^{1/2}, \label{eq: cvd 6}\\
        & \max_\nu \osc(\nu) \le C Q(\Delta t_\mx)^{1/4}.\label{eq: cvd 7}
    \end{flalign}

    For any $\nu$, and any $t \in [\tnu, t_{\nu+1})$, we have
    \[
    |\qcsig(t) - \qdsig(t)| = | \qsig(t) - \qsig(\tnu)| \le \osc(\nu) \le C Q(\Delta t_\mx)^{1/4}.
    \]
    Thus,
    \[
    \max_{t \in [0,T]} | \qcsig(t) - \qdsig(t)| \le C Q(\Delta t_\mx)^{1/4},
    \]
    as claimed. Next, for any $\nu$ and any $t \in [\tnu, t_{\nu+1})$, we have
    \begin{align*}
        \zocsig(t) - \zodsig(t) = &\frac{1}{2}[(\qsig(t))^2 - (\qsig(\tnu))^2] + \frac{1}{2} \Big[ \sum_{0 \le \mu < \nu}(\Delta \qsig_\mu)^2 - t_\nu\Big] \\
        &- \frac{1}{2}(t-\tnu) - \sum_{0 \le \mu < \nu} \int_{t_\mu}^{t_{\mu+1}} \{ \usig(t_\mu)[\qsig(s) - \qsig(t_\mu)]\}\ ds \\
        &- \int_{t_\nu}^t \usig(\tnu) \qsig(s)\ ds\\
        \equiv & \term\;1 + \term\;2-\term\;3-\term\;4-\term\;5.
    \end{align*}
    (Here, we have used the fact that $\usig(s) = \usig(t_\mu)$ for $s \in [t_\mu, t_{\mu+1})$.) We note that
    \begin{equation}\label{eq: cvd 8}
        \max_{t \in [0,T]} |\qsig(t)| \le \max_{0 \le \nu < N} \{ |\qsig(\tnu)| + \osc(\nu)\} \le C Q,
    \end{equation}
    by \eqref{eq: cvd 5} and \eqref{eq: cvd 7}. Hence, for $t \in [t_\nu, t_{\nu+1})$, we have
    \begin{align*}
        |\term\;1| &\le \max_{\tilde{t}\in[0,T]} | \qsig(\tilde{t})| \cdot | \qsig(t) - \qsig(\tnu)|\\
        &\le C Q \osc(\nu) \le C Q^2 (\Delta t_\mx)^{1/4}.
    \end{align*}
    So
    \[
    |\term\;1| \le C Q^2 (\Delta t_\mx)^{1/4}\;\text{for all}\; t \in [0,T).
    \]
    Next, \eqref{eq: cvd 6} tells us that
    \[
    |\term\;2| \le Q^2 (\Delta t_\mx)^{1/2}\;\text{for all}\; t\in[0,T).
    \]
    Clearly
    \[
    |\term\;3| \le (\Delta t_\mx).
    \]
    Furthermore,
    \begin{align*}
    |\term\;4| &\le \sum_{0 \le \mu < N}|\usig(t_\mu)| \int_{t_\mu}^{t_{\mu+1}}|\qsig(s) - \qsig(t_\mu)|\ ds\\
    & \le CQ \sum_{0 \le \mu < N} \osc(\mu) \Delta t_\mu \le C Q \cdot \max_\mu \osc(\mu) \le C Q^2 (\Delta t_\mx)^{1/4}
    \end{align*}
    thanks to \eqref{eq: cvd 5} and \eqref{eq: cvd 7}. Finally,
    \begin{align*}
    |\term\;5| &\le \max_\nu |\usig(\tnu)| \cdot \max_{\tilde{t} \in [0,T]} |\qsig(\tilde{t})| \cdot (\Delta t_\mx)\\ 
    & \le C Q^2 ( \Delta t_\mx),
    \end{align*}
    by \eqref{eq: cvd 5} and \eqref{eq: cvd 8}. Combining our estimates for $\textsc{Terms}$ 1--5, we find that
    \[
    |\zocsig(t) - \zodsig(t)| \le C Q^2 (\Delta t_\mx)^{1/4},
    \]
    as claimed. We pass to $\ztcsig$, $\ztdsig$.  For $t \in [\tnu, t_{\nu+1})$, we have
    \begin{align*}
        | \ztcsig(t) - \ztdsig(t) | = &\Big| \sum_{0 \le \mu < \nu} \int_{t_\mu}^{t_{\mu+1}} \{ (\qsig(s))^2 - (\qsig(t_\mu))^2 \} \ ds + \int_{t_\nu}^t (\qsig(s))^2 \ ds \Big|\\
        \le & \sum_{0 \le \mu < \nu} C \Big(\max_{\tilde{t}\in[0,T]} | \qsig(\tilde{t})|\Big) \int_{t_\mu}^{t_{\mu+1}} | \qsig(s) - \qsig(t_\mu)|\ ds \\
        &+ \max_{\tilde{t} \in [0,T]} |\qsig(\tilde{t})|^2 \cdot \Delta t_\mx\\
        \le & C Q \sum_{0 \le \mu < \nu} \osc(\mu) \Delta t_\mu + C Q^2 \Delta t_\mx\\
        \le & C' Q \max_{0 \le \mu < N}\osc(\mu) + C Q^2 \Delta t_\mx\\
        \le & C'' Q^2 (\Delta t_\mx)^{1/4},\;\text{thanks to \eqref{eq: cvd 7} and \eqref{eq: cvd 8}}.
    \end{align*}
    The proof of the lemma is complete.
\end{proof}

\section{Refining a Partition}
Let $\sigma$ be a tame strategy $\sigma = (\sigma_{\tnu})_{0 \le \nu < N}$ associated to a partition
\begin{equation}\label{eq: rap 1}
    0 = t_0 < t_1 < \dots < t_N = T.
\end{equation}
Suppose
\begin{equation}\label{eq: rap 2}
    0 = \hatt_0 < \hatt_1 < \dots < \hatt_\hatN = T
\end{equation}
is a refinement of the partition \eqref{eq: rap 1}.

We would like to associate a tame strategy $\hatsigma = (\hatsigma_{\hatt_\mu})_{0 \le \mu < \hatN}$ to the partition \eqref{eq: rap 2} in such a way that
\begin{equation}\label{eq: rap 3}
\begin{split}
    &\qhatsig(t) = \qsig(t)\;\text{for all} \; t \in [0,T],\; \text{and}\\
    &\uhatsig(t) = \usig(t)\;\text{for all} \; t \in [0,T).
\end{split}
\end{equation}
That would tell us that refining the partition \eqref{eq: rap 1} allows additional tame strategies, but doesn't rule out any tame strategies $\sigma$.

Unfortunately, no such $\hatsigma$ exists. The problem is that, in order to be a tame strategy, $\hatsigma$ must satisfy
\begin{equation}\label{eq: rap 4}
    |\uhatsig(\hatt_\mu)| \le C_\tame [ |\qhatsig(\hatt_\mu)| + 1]\;\text{with probability}\; 1.
\end{equation}
It may happen that
\[
|\usig(\hatt_\mu)| = |\usig(\tnu)| \gg C_\tame \cdot [|\qsig(\hatt_\mu)| + 1]
\]
for some $\hatt_\mu \in (t_\nu, t_{\nu+1})$, in which case \eqref{eq: rap 4} contradicts \eqref{eq: rap 3}. Accordingly, we modify \eqref{eq: rap 3}, as follows.

For each $\mu$ ($0 \le \mu < \hatN$), define $\bar{\nu}(\mu)$ to be the index $\nu$ for which $\tnu \le \hatt_\mu < t_{\nu+1}$.

Define a stopping time $\tau$ by setting
\[
\tau = \begin{cases}
    \text{Least } \hatt_\mu\; \text{s.t. } |\usig(t_{\bar{\nu}(\mu)})| > 2 C_\tame [ |\qsig(\hatt_\mu)| + 1], &\text{if such a } \hatt_\mu \; \text{exists},\\
    T &\text{otherwise}.
\end{cases}
\]

Then define random processes $\hatq(t)$, $\hatu(t)$ ($t \in [0,T]$) by setting
\[
\begin{cases}
    \hatq(t) = \qsig(t)\;\text{and} \; \hatu(t) = \usig(t) &\text{for}\; 0 \le t < \tau.\\
    \hatu(t)=0 &\text{for}\; \tau \le t \le T,
\end{cases}
\]
and on  $[\tau,T]$ defining $\hatq$ by
\[
\begin{cases}
    d\hatq(t) = \big(a_\tru\cdot \hatq(t)\big) dt + dW(t)\\
    \text{with initial condition}\; \hatq(\tau) = \qsig(\tau).
\end{cases}
\]
Note that
\begin{equation}\label{eq: rap 5}
    |\hatu(\hat{t}_\mu)| \le 2 C_\tame [ |\hatq(\hat{t}_\mu)| + 1] \;\text{for all}\; \hat{t}_\mu.
\end{equation}
It is a tedious exercise (left to the reader) to exhibit a tame strategy $\hatsigma = (\hatsigma_{\hatt_\mu})_{0 \le \mu < \hatN}$ associated to the partition \eqref{eq: rap 2}, such that $\hatq(t) = \qhatsig(t)$ for all $ t\in [0,T]$ and $\hatu(t) = \uhatsig(t)$ for all $ t\in [0,T)$.

Whereas our original tame strategy $\sigma$ satisfies
\begin{equation}
    |\sigma_{t_\nu}(q_1,\dots, q_\nu, \vxi)| \le C_\tame [|\qnu| + 1],
\end{equation}
the strategy $\hatsigma$ satisfies instead
\begin{equation}
    |\hatsigma_{\hat{t}_\mu}(q_1, \dots, q_\mu, \vxi) | \le 2 C_\tame [ |q_\mu| + 1];
\end{equation}
compare with \eqref{eq: rap 5}.

In place of \eqref{eq: rap 3}, we will show that $\qhatsig(t)$ and $\uhatsig(t)$ are likely very close to $\qsig(t)$, $\usig(t)$, respectively. To see this, we fix $a \in [-a_\mx, a_\mx]$ and $\veta \in \{0,1\}^\N$, and condition on $a_\tru = a$, $\vxi = \veta$.

We define random variables
\begin{equation}\label{eq: rap 8}
    \osc(\nu) = \max_{t\in[\tnu, t_{\nu+1}]}|\qsig(t) - \qsig(\tnu)|
\end{equation}
and an event
\begin{equation}\label{eq: rap 9}
    \disaster: \;\osc(\nu) \ge 1\;\text{for some}\; \nu.
\end{equation}
In Section \ref{sec: costing by integrals}, we proved that
\begin{equation}\label{eq: rap 10}
    \E_{a,\veta}[(\osc(\nu))^m] \le C_m (\Delta t_\nu)^{m/2}\;\text{for all}\; m \ge 1.
\end{equation}
Hence,
\[
\prob_{a,\veta}[\osc(\nu)\ge 1] \le C_{\bar{m}} (\Delta t_\nu)^{\bar{m}}\;\text{for all}\; \bar{m}\ge 1, 
\]
and consequently
\begin{equation}\label{eq: rap 11}
    \prob_{a,\veta}[\disaster] \le \sum_{0\le\nu < N} C_{\bar{m}} (\Delta t_\nu)^{\bar{m}}\le C_{\bar{m}-1}' (\Delta t_\mx)^{\bar{m}-1}
\end{equation}
for any $\bar{m} \ge 1$.

Next, we prepare to estimate
\[
\int_0^T\{ | \qhatsig(t) - \qsig(t)|^m + |\uhatsig(t) - \usig(t)|^m\}\ dt.
\]
Let 
\[
\qsig_0(t) = \qsig(\tnu)\;\text{for } t\in[\tnu, t_{\nu+1}),\; 0 \le \nu < N.
\]
Then
\[
\int_{\tnu}^{t_{\nu+1}} | \qsig_0(t) - \qsig(t)|^{2m}\; dt \le (\osc(\nu))^{2m} \Delta \tnu,
\]
hence
\begin{equation}\label{eq: rap 12}
\begin{split}
    \E_{a,\veta} \Big[ \int_0^T | \qsig_0(t) - \qsig(t)|^{2m}\ dt \Big] &\le \sum_\nu \E_{a,\veta} [(\osc(\nu))^{2m}]\Delta \tnu\\
    &\le C_m(\Delta t_\mx)^m,
\end{split}
\end{equation}
thanks to \eqref{eq: rap 10}. Similarly,
\begin{equation}\label{eq: rap 13}
    \E_{a,\veta} \Big[ \int_0^T |\qhatsig_0(t) - \qhatsig(t)|^{2m}\ dt\Big]\le C_m (\Delta t_\mx)^m,
\end{equation}
where $\qhatsig_0(t) = \qhatsig(\hatt_\mu)$ for $t \in [\hatt_\mu, \hatt_{\mu+1})$, $0 \le \mu < \hatN$. We have also
\begin{equation}\label{eq: rap 14}
    \E_{a,\veta} \Big[ \int_0^T | \qsig_0(t)|^{2m}\ dt\Big] = \E_{a,\veta} \Big[ \sum_{0 \le \nu < N} | \qsig(\tnu)|^{2m}\Delta \tnu\Big] \le C_m,
\end{equation}
thanks to Lemma \ref{lem: rare events}. Similarly,
\begin{equation}\label{eq: rap 15}
\E_{a,\veta} \Big[ \int_0^T | \qhatsig_0(t)|^{2m}\ dt\Big] \le C_m.
\end{equation}
From \eqref{eq: rap 12} and \eqref{eq: rap 14} we obtain
\begin{equation}\label{eq: rap 16}
    \E_{a,\veta} \Big[ \int_0^T |\qsig(t)|^{2m}\ dt\Big] \le C_m\;\text{for} \; m\ge 1.
\end{equation}
Similarly, from \eqref{eq: rap 13} and \eqref{eq: rap 15}, we have
\begin{equation}\label{eq: rap 17}
    \E_{a,\veta} \Big[ \int_0^T | \qhatsig(t)|^{2m}\ dt\Big] \le C_m\;\text{for } m \ge 1.
\end{equation}

Turning to $\usig$ and $\uhatsig$, we recall that $\usig(t) = \usig(t_\nu)$ for $t \in [\tnu, t_{\nu+1})$, $0 \le \nu < N$; hence, 
\begin{equation}\label{eq: rap 18}
    \E_{a,\veta}\Big[ \int_0^T |\usig(t)|^{2m}\ dt\Big] = \E_{a,\veta} \Big[ \sum_{0 \le \nu < N} |\usig(t_\nu)|^{2m} \Delta \tnu \Big] \le C_{m} \;(m\ge 1),
\end{equation}
by Lemma \ref{lem: rare events}. Similarly
\begin{equation}\label{eq: rap 19}
\E_{a,\veta}\Big[ \int_0^T |\uhatsig(t)|^{2m}\ dt\Big]  \le C_{m} \;(m\ge 1).
\end{equation}
From \eqref{eq: rap 16}--\eqref{eq: rap 18} we see that
\begin{equation}\label{eq: rap 20}
    \E_{a,\veta}\Big[ \int_0^T \{ |\qhatsig(t) - \qsig(t)| + |\uhatsig(t) - \usig(t)|\}^{2m}\ dt\Big]\le C_m\;(m\ge 1).
\end{equation}
Moreover, unless $\disaster$ occurs, we have $\tau=T$, hence
\begin{equation}\label{eq: rap 21}
\qhatsig(t) = \qsig(t)\;\text{and}\; \uhatsig(t) = \usig(t)\;\text{for all}\; t \in [0,T].
\end{equation}
Indeed, if $\disaster$ doesn't occur, then for $0 \le \nu < N$ and $t \in [\tnu, t_{\nu+1})$ we have
\[
|\qsig(t) - \qsig(\tnu)| \le 1,
\]
hence $[|\qsig(t)| + 1]$ and $[|\qsig(\tnu)| + 1]$ differ by at most a factor of 2. Since $|\usig(\tnu)| \le C_\tame[|\qsig(\tnu)| + 1]$, it follows that $|\usig(\tnu)| \le 2 C_\tame[|\qsig(t)| + 1]$ for $ t \in [\tnu, t_{\nu+1})$, $0 \le \nu < N$. In particular,
\begin{equation}\label{eq: rap 22}
|\usig(t_{\bar{\nu}(\mu)})| \le 2 C_\tame [ |\qsig(t_\mu)| + 1]\;\text{for all} \; \mu \; (0 \le \mu < \hatN).
\end{equation}
Comparing \eqref{eq: rap 22} with the definition of $\tau$, we see that, as claimed, $\tau = T$ unless $\disaster$ occurs.

Thus \eqref{eq: rap 21} holds unless $\disaster$ occurs.

From \eqref{eq: rap 20}, \eqref{eq: rap 21}, \eqref{eq: rap 11}, we now have
\begin{equation}
    \begin{split}
        \E_{a,\veta} \Big[ \int_0^T& \{ |\qhatsig(t) - \qsig(t)| + |\uhatsig(t) -\usig(t)|\}^m\ dt \Big]\\
        = & \E_{a,\veta} \Big[ \int_0^T \{ |\qhatsig(t) - \qsig(t)| + |\uhatsig(t) - \usig(t)|\}^m \mathbbm{1}_\disaster \ dt\Big]\\
        \le & \Big( \E_{a,\veta}\Big[ \int_0^T\{ |\qhatsig(t) - \qsig(t)| + |\uhatsig(t) - \usig(t)|\}^{2m}\ dt\Big]\Big)^{1/2}\\
        & \cdot (\prob_{a,\veta}[\disaster])^{1/2}\\
        \le & C_{m,\bar{m}}\cdot (\Delta t_\mx)^{(\bar{m}-1)/2}\;\text{for any}\; m,\bar{m} \ge 1.
    \end{split}
\end{equation}
We record this result as a lemma.
\begin{lem}[Refinement Lemma]\label{lem: refinement}
    Let $\sigma$ be a tame strategy associated to a partition
    \begin{equation}\tag{A}
        0 = t_0 < t_1 < \dots < t_N = T,
    \end{equation}
    and let
    \begin{equation}\tag{B}
        0 = \hatt_0 < \hatt_1 < \dots < \hatt_\hatN = T
    \end{equation}
    be a refinement of the partition \emph{(A)}.

    Then there exists a tame strategy $\hat{\sigma}$ associated to the partition \emph{(B)}, such that
    \[
    \eE_{a,\veta} \Big[ \int_0^T \{ | \qhatsig(t) - \qsig(t)| + |\uhatsig(t) - \usig(t)| \}^m\ dt\Big] \le C_{m,\bar{m}} (\Delta t_\emx)^{\bar{m}}
    \]
    for all $m,\bar{m} \ge 1$ and all $ a \in [-a_\emx, +a_\emx]$, $\veta \in \{0,1\}^\N$.

    The strategy $\hatsigma$ satisfies the same estimates as we assumed for $\sigma$ (see \ref{eq: du 1}), except that $C_\etame$ is replaced by $2 C_\etame$.
\end{lem}

\chapter{Bayesian Strategies Associated to Partitions}\label{chap: 3}

\section{Setup}

In this chapter, we take $a_\tru$ to be governed by a known prior probability distribution $d\prior(a)$, concentrated on an interval $[-a_\mx, + a_\mx]$.

We fix a partition
\begin{equation}\label{eq: 3.1}
    0 = t_0 < t_1 < \dots < t_N = T
\end{equation}
of the time interval $[0,T]$.

We fix a deterministic strategy $\sigma$ for the game starting at position $q_0$.

We assume that our strategy $\sigma$ is tame, i.e.,
\begin{equation}\label{eq: 3.2}
    |\usig(\tnu)|\le C_\tame^\sigma[|\qsig(\tnu)| + 1]
\end{equation}
for a constant $C_\tame^\sigma$. We call $C_\tame^\sigma$ the \emph{tame constant} for $\sigma$.

We write
\[
\begin{aligned}
   & q_\nu^\sigma = q^\sigma(\tnu),& &\Delta q_\nu^\sigma = q_{\nu+1}^\sigma - q_\nu^\sigma,&\\
   &\zeta_{1,\nu}^\sigma = \zo^\sigma(\tnu),& &\Delta \zonu^\sigma = \zeta_{1,\nu+1}^\sigma - \zonu^\sigma,&\\
   & \ztnu^\sigma = \zt^\sigma(\tnu),& & \Delta \ztnu^\sigma = \zeta_{2,\nu+1}^\sigma- \ztnu^\sigma,\\
   & \usig_\nu = \usig(\tnu).
\end{aligned}
\]

Until further notice, $c,C,C',$ etc.\ will denote constants determined by $C_\tame^\sigma$ in \eqref{eq: 3.2} together with $a_\mx$ and upper bounds for $T$ and $|q_0|$. The symbols $c,C,C',$ etc.\ may denote different constants in different occurrences. We assume that
\begin{equation}\label{eq: 3.3}
    \Delta t_\mx \equiv \max_{0 \le \nu < N} (t_{\nu+1} - t_\nu)<c\;\text{for a small enough constant } c.
\end{equation}

We write $X= O(Y)$ to indicate that $|X| \le CY$. We write $\cF_\nu^\sigma$ to denote the sigma algebra of events determined by $\qsig(t_\mu)$ for $\mu =0,1,\dots,\nu$. Note that $\cF_\nu^\sigma$ depends on $\sigma$, because $a_\tru$ isn't deterministic.

We write $\prob[\dots]$ to denote probability, and we write $\E[\dots]$ to denote expectation.

If we condition on $a_\tru = a$, then we write $\prob_a[\dots]$ and $\E_a[\dots]$ to denote the corresponding probability and expectation. Thus, for instance, $\E_a[X|\cF_\nu^\sigma]$ dentoes the expected value of $X$ conditioned on $\cF_\nu^\sigma$, given that $a_\tru = a$.

For any event $\cE$ we have
\begin{equation}\label{eq: 3.3.1}
    \prob[\cE] = \int_{-a_\mx}^{a_\mx} \prob_a[\cE] \ d\prior(a)
\end{equation}
and
\begin{equation}\label{eq: 3.4}
    \prob[\cE| \cF_\nu^\sigma] = \int_{-a_\mx}^{a_\mx} \prob_a[\cE | \cF_\nu^\sigma] \ d\post(a | \cF_\nu^\sigma),
\end{equation}
where $d\post(a|\cF_\nu^\sigma)$ is the posterior probability distribution for $a_\tru$ conditioned on $\cF_\nu^\sigma$.

Similarly, for any random variable $X$, we have
\begin{equation}\label{eq: 3.5}
    \E[X] = \int_{-a_\mx}^{a_\mx} \E_a[X]\ d\prior(a)
\end{equation}
and
\begin{equation}\label{eq: 3.6}
    \E[X | \cF_\nu^\sigma] = \int_{-a_\mx}^{a_\mx} \E_a[X | \cF_\nu^\sigma]\ d\post(a|\cF_\nu^\sigma).
\end{equation}
Thanks to \eqref{eq: 3.3.1}--\eqref{eq: 3.6}, the following results are immediate from Lemma \ref{lem: rare events}.

\begin{lem}\label{lem: bayesian rare events}[Bayesian Lemma on Rare Events]
    Suppose $Q>C$ for a large enough $C$. Then the following hold with probability $>1-\exp(-cQ^2)$:
    \begin{itemize}
        \item $|\qsig(\tnu)|, |\usig(\tnu)| \le Q$, for all $\nu$.
        \item $|\zosig(\tnu)|, |\ztsig(\tnu)| \le Q^2$, for all $\nu$.
        \item $| \sum_{0 \le \mu < \nu} (\qsig(t_{\mu+1}) - \qsig(t_\mu))^2 - \tnu| \le Q^2 (\Delta t_\emx)^{1/2}$ for all $\nu$.
    \end{itemize}
    Moreover, suppose we fix $\nu$ and condition on $\cF_\nu^\sigma$, the sigma algebra of events determined by $\qsig(t_\mu)$ for $0 \le \mu \le \nu$. Suppose $|\qsig(\tnu)|\le Q$, where $C\le Q \le (\Delta \tnu)^{-1/1000}$ (for large enough $C$).

    Then for $p=1,2$, we have
    \begin{multline*}
        \eE[(|\Delta \qnusig| + | \Delta \zonusig| + |\Delta \ztnusig|)^p\cdot \mathbbm{1}_{|\Delta \qnusig| + |\Delta \zonusig| + |\Delta \ztnusig|>(\Delta \tnu)^{2/5}} | \cF_\nu^\sigma] \\ \le C \cdot (\Delta \tnu)^{100};
    \end{multline*}
    also
    \[
    \eprob\Big[\max_{t\in[\tnu, t_{\nu+1}]} | \qsig(t) - \qnusig| > (\Delta \tnu)^{2/5} | \cF_\nu^\sigma\Big]\le C\cdot(\Delta \tnu)^{1000}
    \]
    and
    \[
    \eprob[ |\Delta \qnusig| + |\Delta  \zonusig| + |\Delta \ztnusig| > (\Delta \tnu)^{2/5} | \cF_\nu^\sigma] \le C \cdot (\Delta \tnu)^{1000}.
    \]
\end{lem}

\section{Posterior Probabilities and Expectations}
Fix $Q>C$ for large enough $C$.

For each $\nu$, let $\ok(\nu)$ denote the set of all $(\barq_1,\dots, \barq_\nu) \in \R^\nu$ s.t.
\begin{equation}\label{eq: 3.7}
    \max_{\mu\le\nu} |\barq_\mu| + |\baru_\mu| \le Q
\end{equation}
and
\begin{equation}\label{eq: 3.8}
    \Big| \sum_{\mu < \nu} (\Delta \barq_\mu)^2 - \tnu \Big| \le Q^2 (\Delta t_\mx)^{1/4},
\end{equation}
where $\Delta \barq_\mu = \barq_{\mu+1}-\barq_\mu$, $\barq_0 \equiv q_0$, and $\baru_\mu$ is defined to be the value assigned to $\usig(t_\mu)$ provided $\qsig(t_\gamma) = \barq_\gamma$ for $\gamma \le \mu$.

According to the Bayesian Lemma on Rare Events, we have
\begin{equation}\label{eq: 3.9}
    \prob[\text{NOT }\ok(\nu)] \le \exp(-cQ^2).
\end{equation}
Thanks to \eqref{eq: 3.3} and Lemma \ref{lem: prob dist}, the following holds for $a\in [-a_\mx, a_\mx]$ and $(\barq_1,\dots,\barq_\nu) \in \ok(\nu)$:
\begin{multline*}
    \prob[ a_\tru \in [a,a + da]\;\text{and}\; \qsig(t_\mu) \in [\barq_\mu,\barq_\mu + d\barq_\mu]\;\text{for}\; \mu \le \nu]\\
    = ( 1 + O(Q^2(\Delta t_\mx)^{1/4}))\cdot(*)\cdot d\barq_1\cdots d\barq_\nu,
\end{multline*}
where
\[
(*) = \Big[ d\prior(a)\cdot \prod_{0\le\mu < \nu} \Big\{\frac{1}{\sqrt{2\pi\Delta t_\mu}}\exp\Big( - \frac{1}{2\Delta t_\mu}(\Delta \barq_\mu - (a \barq_\mu + \baru_\mu)\Delta t_\mu)^2\Big)\Big\} \Big].
\]
Note that
\[
(*) = d\prior(a) \cdot \exp\Big(- \frac{1}{2}\barztnu a^2 + \barzonu a \Big)\cdot \{\text{Factor independent of }a\},
\]
where
\begin{equation}\label{eq: 3.10}
    \barztnu = \barztnu(\barq_1,\dots,\barq_\nu) = \sum_{0 \le \mu < \nu} \barq_\mu^2 \Delta t_\mu
\end{equation}
and
\begin{equation}\label{eq: 3.11}
    \barzonu = \barzonu(\barq_1,\dots, \barq_\nu) = \sum_{0\le\mu < \nu} \barq_\mu (\Delta \barq_\mu - \baru_\mu \Delta t_\mu).
\end{equation}
Hence, for 
\[
(\barq_1,\dots,\barq_\nu) \in \ok(\nu),
\]
the posterior probability distribution for $a_\tru,$ given that $\qsig(t_\mu) = \barq_\mu$ for $\mu = 1,\dots,\nu,$ is given by
\begin{multline}\label{eq: 3.12}
    d\post(a | \barq_1,\dots,\barq_\nu) \\= (1+O(Q^2(\Delta t_\mx)^{1/4}))\cdot \frac{d\prior(a)\cdot \exp\Big(-\frac{1}{2}\barztnu a^2 + \barzonu a\Big)}{Z}
\end{multline}
for a normalizing constant $Z$. Since
\[
\int_{-a_\mx}^{a_\mx}d\post(a | \barq_1,\dots,\barq_\nu) = 1,
\]
equation \eqref{eq: 3.12} holds with
\begin{equation}\label{eq: 3.13}
    Z = \int_{-a_\mx}^{a_\mx} d\prior(a)\cdot\exp\Big(-\frac{1}{2}\barztnu a^2 + \barzonu a\Big).
\end{equation}
We have thus proven the following result.

\begin{lem}[Lemma on Posterior Probabilities]\label{lem: post prob}
Suppose $Q \ge C $ for a large enough $C$. Fix $\nu$, and suppose $(\qsig(t_1),\dots,\qsig(t_\nu))$ satisfy
\begin{align}
&\max_{\mu\le \nu} |\qsig(t_\mu)| + |\usig(t_\mu)| \le Q,\label{eq: 3.14}\\
&\Big|\sum_{\mu < \nu} (\Delta \qsig(t_\mu))^2 - t_\nu \Big| \le Q^2 (\Delta t_\emx)^{1/4}.\label{eq: 3.16}
\end{align}
Then the posterior probability distribution for $a_\etru$ conditioned on $\cF_\nu^\sigma$ is given by
\begin{equation}\label{eq: 3.17}
\begin{split}
    d\epost(a | \cF_\nu^\sigma) = &(1+O(Q^2(\Delta t_\emx)^{1/4}))\\ &\cdot \frac{\exp\Big(-\frac{1}{2}\ztsig(\tnu)a^2 + \zosig(\tnu) a\Big)\cdot d\eprior(a)}{Z}
\end{split}
\end{equation}
with
\begin{equation}\label{eq: 3.18}
    Z = \int_{-a_\emx}^{a_\emx} \exp\Big( - \frac{1}{2} \ztsig(\tnu)b^2 + \zosig(\tnu)b\Big)\ d\eprior(b).
\end{equation}
\end{lem}

\begin{cor}
    Under the assumption of the above Lemma, we have
    \begin{equation}\label{eq: 3.19}
        \eE[a_\etru | \cF_\nu^\sigma] = O(Q^2(\Delta t_\emx)^{1/4}) + \bar{a}(\zosig(\tnu),\ztsig(\tnu)),
    \end{equation}
    where
    \begin{equation}\label{eq: 3.20}
        \bar{a}(\bar{\zo},\bar{\zt}) \equiv \frac{ \int_{-a_\emx}^{a_\emx}a \exp\Big( - \frac{a^2}{2}\bar{\zt} + a \bar{\zo}\Big)\ d\eprior(a) }{\int_{-a_\emx}^{a_\emx} \exp\Big( - \frac{a^2}{2}\bar{\zt} + a \bar{\zo}\Big)\ d\eprior(a) }
    \end{equation}
    for $(\bar{\zo},\bar{\zt})\in \R^2$.
\end{cor}
Thanks to the above Corollary, formula \eqref{eq: 3.6}, and Lemma \ref{lem: moments}, we now have the following results.

\begin{lem}[Lemma on Posterior Expectations]\label{lem: post exp}
    Suppose $Q \ge C $ for large enough $C$, and assume that $\Delta t_\emx \le Q^{-1000}$.

    Define the event
    \[
    \etame(\nu) = \{ |\Delta \qnusig| \le 2 (\Delta \tnu)^{2/5}, |\Delta \zonusig| \le 2 (\Delta \tnu)^{2/5}, |\Delta \ztnusig| \le 2 (\Delta \tnu)^{2/5}\}.
    \]
    Fix $\nu$, and suppose we have
    \begin{align*}
        &\max_{\mu \le \nu} |\qsig(t_\mu)| + |\usig(t_\mu)| \le Q,\;\text{and}\\
        &\Big| \sum_{0 \le \mu < \nu} (\Delta \qsig(t_\mu))^2 - \tnu \Big| \le Q^2 (\Delta t_\emx)^{1/4}.
    \end{align*}
    Then the following hold.
    \begin{align*}
        &\eE[(\Delta \qnusig)\cdot \mathbbm{1}_{\etame(\nu)} | \cF_\nu^\sigma] = [\bar{a}(\zosig(\tnu),\ztsig(\tnu))\qnusig + \unusig](\Delta \tnu) + \eerr\;1,\\
        &\eE[(\Delta \zonusig)\cdot\mathbbm{1}_{\etame(\nu)} | \cF_\nu^\sigma] = \bar{a}(\zonusig(\tnu),\ztnusig(\tnu))\cdot(\qnusig)^2(\Delta \tnu) + \eerr\;2,\\
        &\eE[(\Delta \ztnusig)\cdot\mathbbm{1}_{\etame(\nu)}|\cF_\nu^\sigma] = (\qnusig)^2 (\Delta \tnu) + \eerr\;3,\\
        &\eE[(\Delta \qnusig)^2 \cdot \mathbbm{1}_{\etame(\nu)} | \cF_\nu^\sigma] = (\Delta \tnu) + \eerr\; 4,\\
        &\eE[(\Delta \qnusig)\cdot(\Delta \zonusig)\cdot\mathbbm{1}_{\etame(\nu)} | \cF_\nu^\sigma] = \qnusig(\Delta \tnu) + \eerr\; 5,\\
        &\eE[(\Delta \qnusig)\cdot(\Delta \ztnusig)\cdot\mathbbm{1}_{\etame(\nu)}|\cF_\nu^\sigma] = \eerr\; 6,\\
        &\eE[(\Delta \zonusig)^2\cdot \mathbbm{1}_{\etame(\nu)} | \cF_\nu^\sigma ] = (\qnusig)^2 (\Delta \tnu) + \eerr\;7,\\
        &\eE[(\Delta \zonusig)(\Delta \ztnusig)\cdot\mathbbm{1}_{\etame(\nu)} | \cF_\nu] = \eerr \; 8,\\
        &\eE[(\Delta \ztnusig)^2\cdot\mathbbm{1}_{\etame(\nu)} | \cF_\nu^\sigma] = \eerr\; 9,
    \end{align*}
    where
    \[
    |\eerr\;1|,\dots,|\eerr\;9| \le C' Q^4 (\Delta t_\emx)^{1/4} \Delta \tnu.
    \]
    Moreover, under the above assumptions on $(\qsig(t_1),\dots,\qsig(\tnu))$, we have
    \[
    \eprob[\enottame(\nu)|\cF_\nu^\sigma] \le (\Delta \tnu)^{20}.
    \]
    Here,
    \[
    \bar{a}(\zo,\zt):= \frac{\int_{-a_\emx}^{a_\emx} a\exp\Big(-\frac{a^2}{2}\zt + a\zo\Big)\ d\eprior(a)}{\int_{-a_\emx}^{a_\emx}\exp\Big( - \frac{a^2}{2} \zt + a \zo\Big)\ d\eprior(a)}.
    \]
\end{lem}

\section{The PDE Assumption}\label{sec: pde}
For our given probability distribution $d\prior(a)$, we fix the function $\bar{a}(\zo,\zt)$ given in the preceding section, and we introduce the following PDE for an unknown function $S(q,t,\zo,\zt)$ defined on $\R\times[0,T]\times\R\times[0,\infty)$:
\begin{equation}\label{eq: pde 1}
    \begin{split}
        0 = &\partial_t S + (\bar{a}(\zo,\zt) q + u_\op)\partial_q S + \bar{a}(\zo,\zt) q^2 \partial_{\zo}S + q^2 \partial_{\zt}S+ \frac{1}{2}\partial_q^2 S \\& + q \partial_{q\zo} S  + \frac{1}{2}q^2 \partial_{\zo}^2 S + (q^2 + u_\op^2),
    \end{split}
\end{equation}
where
\begin{equation}\label{eq: pde 2}
    u_\op = - \frac{1}{2} \partial_q S.
\end{equation}
We \emph{assume} the existence of a solution of the above PDE, satisfying the following additional conditions.

The \textsc{terminal condition}:
\begin{equation}\label{eq: pde 3}
    S(q,t,\zo,\zt) = 0\;\text{at}\; t = T.
\end{equation}

\textsc{Positivity}:
\begin{equation}\label{eq: pde 4}
    S(q,t,\zo,\zt) \ge 0.
\end{equation}

\textsc{Estimates}: We assume $S \in C^{1,2}$. For $|\alpha| \le 3$, we have almost everywhere that
\begin{equation}\label{eq: pde 5}
    |\partial^\alpha S(q,t,\zo,\zt)| \le K \cdot[ 1 + |q| + |\zo| + |\zt|]^{m_0}
\end{equation}
for constants $K$, $m_0 \ge 1$. Moreover,
\begin{equation}\label{eq: pde 6}
    |u_\op(q,t,\zo,\zt)| \le C_\tame^\op[|q| + 1].
\end{equation}
Note that \eqref{eq: pde 5} holds everywhere, not just almost everywhere, when $|\alpha| \le 2$.

We call $K$, $m_0$, $C_\tame^\op$, $a_\mx$, and our upper bounds for $T$ and $|q_0|$ the \textsc{boilerplate constants}. We now broaden our definition of constants $c$, $C$, $C'$, etc., to allow them to depend on the \textsc{boilerplate constants}. As usual, these symbols may denote different constants in different occurrences.

We strengthen our large $Q$ assumption. More precisely, we assume from now on that $Q\ge C$ for a large enough constant $C$. Since the meaning of the constant $C$ has changed, the above is stronger than our previous large $Q$ assumption.

We assume that
\[
\Delta t_\mx \le Q^{-1000}.
\]

\section{The Allegedly Optimal Strategy}
Let $u_\op = u_\op(q,t,\zo,\zt)$ be as in Section \ref{sec: pde}.

We define a strategy $\tsig$ based on the function $u_\op$.

Given $\nu$ ($1\le \nu \le N)$ and given real numbers $q_1,\dots,q_\nu$, we define numbers $u_\mu$, $\zomu$, $\ztmu$ by induction on $\mu = 0,1,\dots,\nu$ so that $\zeta_{1,0} = \zeta_{2,0} = 0$, and
\begin{flalign*}
    &u_\mu = u_\op(q_\mu,t_\mu,\zomu,\ztmu)&\\
    &\zomu = \sum_{0 \le \gamma < \mu} q_\gamma ( [q_{\gamma+1} - q_\gamma] - u_\gamma \Delta t_\gamma)\\
    &\ztmu = \sum_{0 \le \gamma < \mu} q_\gamma^2 \Delta t_\gamma
\end{flalign*}
for each $\mu$.

We then set $\tsig_\nu(q_1, \dots, q_\nu)$ equal to the above $u_\mu$ for $\mu = \nu$.

We define our \textsc{allegedly optimal strategy} $\tsig$ to be the collection of tame rules
\[
\tsig = (\tsig_\nu)_{\nu=0,1,\dots,N-1}.
\]
Thanks to our PDE Assumption (see \eqref{eq: pde 6}), each $\tsig_\nu$ is indeed a tame rule with tame constant $C_\tame^\op$, hence $\tsig$ is a strategy.

\section{Performance of Competing Strategies}\label{sec: pcs}

We have just defined the allegedly optimal strategy $\tsig$, based on the functions $u_\op(q,t,\zo,\zt)$, $S(q,t,\zo,\zt)$, and $\bar{a}(\zo,\zt)$.

In this section, we compare the performance of $\tsig$ with that of an arbitrary competing (deterministic, tame) strategy $\sigma$, also defined with respect to the given partition $0 = t_0 < t_1 < \dots < t_N = T$. We assume that
\begin{equation}\label{eq: pcs 1}
    (\Delta t_\mx) \le Q^{-2000m_0}.
\end{equation}
(See Section \ref{sec: pde} for $m_0$.) Thanks to our assumption \ref{eq: pde 6}, the strategy $\tsig$ satisfies
\[
|u^{\tsig}(\tnu)| \le C_\tame^\op[|q^{\tsig}(\tnu)| + 1],
\]
i.e., $\tsig$ is tame with constant $C_\tame^\op$. We assume that the strategy $\sigma$ is tame with constant $C_\tame^\sigma$, i.e., we assume that
\[
|u^\sigma(\tnu)| \le C_\tame^\sigma[|q^\sigma(\tnu)| + 1].
\]
In this section, we write $c$, $C$, $C'$, etc.\ to denote constants determined by $C_\tame^\sigma$ and the \textsc{boilerplate constants} (one of which is $C_\tame^\op$). We strengthen our large $Q$ assumption by supposing that $Q>C$ for a large enough $C$. Since the meaning of constants $C$ has changed, this indeed strengthens our previous large $Q$ assumption.

We define
\begin{equation}\label{eq: pcs 2}
    \cU = \{(q,\zo,\zt)\in\R^3 : |q|, |\zo|, |\zt| \le Q\}
\end{equation}
and
\begin{equation}\label{eq: pcs 3}
    \cU^+ = \{ (q,\zo,\zt) \in \R^3 : |q|, |\zo|, |\zt| \le 2 Q\}.
\end{equation}
For each $\mu$ ($0 \le \mu \le N$) we introduce the event
\begin{equation}\label{eq: pcs 4}
\begin{split}
    \ok_\mu = \Big\{& (\qsig(t_\mu), \zosig(t_\mu), \ztsig(t_\mu)) \in \cU, \\
    & \Big| \sum_{0 \le \gamma < \mu} (\Delta \qsig(t_\gamma))^2 - t_\mu \Big| \le Q^2 (\Delta t_\mx)^{1/4}\Big\}.
\end{split}
\end{equation}
We define the event
\begin{equation}\label{eq: pcs 5}
    \disaster = \{ \ok_\mu\;\text{fails for some }\mu \le N\}
\end{equation}
and the stopping time
\begin{equation}\label{eq: pcs 6}
    \tau = \begin{cases}
        t_\mu \;\text{for the least }\mu\text{ for which }\ok_\mu\;\text{fails,} &\text{if there is such a }\mu,\\
        T = t_N &\text{otherwise}.
    \end{cases}
\end{equation}
Note that $\tau$ is indeed a stopping time with respect to the $\cF_\nu^\sigma$, i.e., the event $\tau > \tnu$ is $\cF_\nu^\sigma$-measurable, for each $\nu$.

We define a \emph{cost-to-go} by setting
\begin{equation}\label{eq: pcs 7}
    \ctg^\sigma(\tnu) = \sum_{\tnu \le t_\mu < \tau}[ (\qsig(t_\mu))^2 + (\usig(t_\mu))^2] \Delta t_\mu.
\end{equation}
To measure the difference between the strategies $\sigma$ and $\tsig$, we introduce the random variable
\begin{equation}\label{eq: pcs 8}
    \discrep^\sigma(t_\mu) = \usig(t_\mu) - u_\op(\qsig(t_\mu),t_\mu,\zosig(t_\mu),\ztsig(t_\mu)).
\end{equation}
Note that 
\[
u_\op(\qsig(t_\mu),t_\mu,\zosig(t_\mu),\ztsig(t_\mu))
\]
is not the same as 
\[
u^{\tsig}(t_\mu) = u_\op(q^{\tsig}(t_\mu),t_\mu,\zo^{\tsig}(t_\mu),\zt^{\tsig}(t_\mu)).
\]

In the next lemma, we compare the cost-to-go of $\sigma$ with that of $\tsig$. We will be conditioning on $\cF_\nu^\sigma$, so the quantities $\qsig(\tnu)$, $\usig(\tnu)$, $\zosig(\tnu)$, $\ztsig(\tnu)$ may be regarded as determinisitc.

\begin{lem}[Main Lemma on Competing Strategies]\label{lem: comp strategies}
Fix $\nu$ $(0 \le \nu \le N)$. Suppose
\begin{equation}\label{eq: pcs star}\tag{$\star$}
(\qsig(\tnu),\zosig(\tnu),\ztsig(\tnu)) \in \cU.
\end{equation}
    Then
    \begin{equation}\tag{A}\label{eq: pcs A}
    \begin{split}
        \eE[\ectg^\sigma(\tnu) | \cF_\nu^\sigma] + &\hat{C}Q^{2m_0}\eprob[\disaster | \cF_\nu^\sigma]\\
        \ge &S(\qsig(\tnu),\tnu,\zosig(\tnu),\ztsig(\tnu)) \\
        &+ \eE\Big[ \sum_{\tnu \le t_\mu < \tau} (\discrep^\sigma(t_\mu))^2 \Delta t_\mu | \cF_\nu^\sigma\Big]\\
        & - Q^{2m_0} (T-\tnu)\cdot (\Delta t_\emx)^{1/20}.
    \end{split}
    \end{equation}
    If $\sigma = \tsig$, then
    \begin{equation}\tag{B}\label{eq: pcs B}
        \begin{split}
            \eE[\ectg^\sigma(\tnu) | \cF_\nu^\sigma] \le &S(\qsig(\tnu), \tnu, \zosig(\tnu), \ztsig(\tnu)) \\
            &+ \hat{C} Q^{2m_0} \eprob[\disaster | \cF_\nu^\sigma] \\
            & + Q^{2m_0}(T-\tnu)\cdot (\Delta t_\emx)^{1/20}.
        \end{split}
    \end{equation}
\end{lem}

We fix the large constant $\hat{C}$ throughout this section.

\begin{proof}
    We proceed by downward induction on $\nu$.

    \underline{In the base case}, $\nu = N$. Since $\tau \le T = t_N$, we have $\ctg^\sigma(\tnu)=0$. Also, $S(\qsig(\tnu),\tnu,\zosig(\tnu),\ztsig(\tnu)) =0$ by the terminal condition for our PDE. Moreover, the sum
    \[
    \sum_{\tnu \le t_\mu < \tau}(\discrep^\sigma(t_\mu))^2 \Delta t_\mu
    \]
    is empty, and $T-\tnu = 0$. Therefore, \eqref{eq: pcs A} asserts that
    \[
    \hat{C} Q^{2m_0}\prob[\disaster | \cF_\nu^\sigma] \ge 0,
    \]
    while \eqref{eq: pcs B} asserts that if $\sigma = \tsig$, then
    \[
    0 \le \hat{C}Q^{2m_0}\prob[\disaster|\cF_\nu^\sigma].
    \]
    These two (equivalent) inequalities are obviously correct, so our lemma holds in the base case $\nu = N$.

    \underline{For the induction step}, we fix $\nu < N$, and assume that \eqref{eq: pcs A} and \eqref{eq: pcs B} hold with $(\nu+1)$ in place of $\nu$. We will deduce \eqref{eq: pcs A} and \eqref{eq: pcs B} for the given $\nu$. To do so, we first dispose of a trivial case.

    \underline{Suppose for a moment that $\ok_\mu$ fails for some $\mu \le \nu$.} Then $\disaster$ occurs, and $\tau \le \tnu$; consequently, $\ctg^\sigma(\tnu) = 0$, and
    \[
    \sum_{\tnu \le t_\mu < \tau} (\discrep^\sigma(t_\mu))^2 \Delta t_\mu = 0.
    \]
    Therefore, \eqref{eq: pcs A} asserts that
    \begin{align*}
        \hat{C}Q^{2m_0} \ge S(\qsig(\tnu),\tnu,\zosig(\tnu),\ztsig(\tnu)) - Q^{2m_0}(T-\tnu)\cdot (\Delta t_\mx)^{1/2},
    \end{align*}
    while \eqref{eq: pcs B} asserts that if $\sigma = \tsig$ then
    \[
    0 \le S(\qsig(\tnu),\tnu,\zosig(\tnu),\ztsig(\tnu)) + \hat{C}Q^{2m_0} + Q^{2m_0}\cdot (T-\tnu)\cdot(\Delta t_\mx)^{1/20}.
    \]
These inequalities are immediate from our assumptions on the PDE solution $S$, together with our hypothesis
\[
(\qsig(\tnu),\zosig(\tnu),\ztsig(\tnu))\in\cU.
\]
Thus, our induction step is complete in the trivial case in which $\ok_\mu$ fails for some $\mu \le \nu$.

\underline{From now on, we assume that}
\begin{equation}\label{eq: pcs 9}
    \underline{\ok_\mu\;\text{holds for all}\; \mu \le \nu.}
\end{equation}
Thus,
\begin{equation}\label{eq: pcs 10}
    \tau \ge t_{\nu+1}
\end{equation}

For the moment, we condition on $\cF_{\nu+1}^\sigma,$ and distinguish two cases.
\[
\begin{aligned}
    &\underline{\text{Case I}}:& &(\qsig(t_{\nu+1}), \zosig(t_{\nu+1}), \ztsig(t_{\nu+1})) \in \cU,&\\
    &\underline{\text{Case II}}:& &(\qsig(t_{\nu+1}),\zosig(t_{\nu+1}),\ztsig(t_{\nu+1}))\notin \cU&.
\end{aligned}
\]
In Case I, our inductive hypothesis tells us the following.
\begin{align*}
    \E[\ctg^\sigma(t_{\nu+1}) | \cF_{\nu+1}^\sigma] + \hat{C} &Q^{2m_0} \prob[\disaster | \cF_{\nu+1}^\sigma]\\
    \ge &S(\qsig(t_{\nu+1}),t_{\nu+1},\zosig(t_{\nu+1}),\ztsig(t_{\nu+1}))\\
    &+ \E\Big[ \sum_{t_{\nu+1} \le t_\mu < \tau} (\discrep^\sigma(t_\mu))^2 \Delta t_\mu | \cF_{\nu+1}^\sigma \Big]\\
    & - Q^{2m_0} (T-t_{\nu+1})\cdot (\Delta t_\mx)^{1/20};
\end{align*}
and if $\sigma = \tsig$ then
\begin{align*}
    \E[\ctg^\sigma(t_{\nu+1}) | \cF_{\nu+1}^\sigma] \le &S(\qsig(t_{\nu+1}), t_{\nu+1}, \zosig(t_{\nu+1}), \ztsig(t_{\nu+1})) \\
    &+ \hat{C} Q^{2m_0} \prob[\disaster | \cF_{\nu+1}^\sigma]\\
    & + Q^{2m_0} (T-t_{\nu+1})\cdot (\Delta t_\mx)^{1/20}.
\end{align*}
    Since
    \[
    \ctg^\sigma(\tnu) = \ctg^\sigma(t_{\nu+1}) + [(\qsig(\tnu))^2 + (\usig(\tnu))^2]\Delta \tnu,
    \]
    thanks to \eqref{eq: pcs 10}, the above inequalities yield at once that
\begin{equation}\label{eq: pcs AI}\tag{AI}
\begin{split}
\E[\ctg^\sigma(\tnu)|\cF_{\nu+1}^\sigma] + \hat{C}&Q^{2m_0}\prob[\disaster | \cF_{\nu+1}^\sigma]\\
 \ge& S(\qsig(t_{\nu+1}),t_{\nu+1},\zosig(t_{\nu+1}),\ztsig(t_{\nu+1})) \\
 &+ [(\qsig(\tnu))^2 + (\usig(\tnu))^2]\Delta \tnu\\
 & + \E\Big[ \sum_{t_{\nu+1} \le t_\mu < \tau} (\discrep^\sigma(t_\mu))^2 \Delta t_\mu | \cF_{\nu+1}^\sigma \Big]\\
 & - Q^{2m_0}\cdot (T-t_{\nu+1}) \cdot (\Delta t_\mx)^{1/20};
\end{split}
\end{equation}
and if $\sigma = \tsig$ then
\begin{equation}\label{eq: pcs BI}\tag{BI}
    \begin{split}
        \E[\ctg^\sigma(\tnu) | \cF_{\nu+1}^\sigma] \le &S(\qsig(t_{\nu+1}),t_{\nu+1},\zosig(t_{\nu+1}),\ztsig(t_{\nu+1}))\\
        &+[(\qsig(\tnu))^2 + (\usig(\tnu))^2]\Delta \tnu\\
        &+\hat{C}Q^{2m_0}\prob[\disaster | \cF_{\nu+1}^\sigma]\\
        &+Q^{2m_0}(T-t_{\nu+1})\cdot (\Delta t_\mx)^{1/20}.
    \end{split}
\end{equation}
Estimates \eqref{eq: pcs AI} and \eqref{eq: pcs BI} hold in Case I.

We turn to \underline{Case II}. In that case, $\ok_{\nu+1}$ fails, $\disaster$ occurs, and $\tau = t_{\nu+1}$, hence
\begin{flalign*}
&\ctg^\sigma(\tnu) = [(\qsig(\tnu))^2 + (\usig(\tnu))^2]\Delta \tnu&\\
&\prob[\disaster | \cF_{\nu+1}^\sigma] = 1\\
& \sum_{t_{\nu+1} \le t_\mu < \tau} (\discrep^\sigma(t_\mu))^2 \Delta t_\mu = 0.
\end{flalign*}
So we have the following estimates:
\begin{equation}\label{eq: pcs AII}\tag{AII}
    \begin{split}
        \E[\ctg^\sigma(\tnu) | &\cF_{\nu+1}^\sigma] + \hat{C}Q^{2m_0}\prob[\disaster | \cF_{\nu+1}^\sigma]\\
         \ge & \hat{C}Q^{2m_0} + \E\Big[ \sum_{t_{\nu+1} \le t_\mu < \tau} (\discrep^\sigma(t_\mu))^2 \Delta t_\mu | \cF_{\nu+1}^\sigma\Big]\\
         & - Q^{2m_0}(T-t_{\nu+1})\cdot (\Delta t_\mx)^{1/20};
    \end{split}
\end{equation}
and if $\sigma = \tsig$ then
\begin{equation}\label{eq: pcs BII}\tag{BII}
\begin{split}
    \E[\ctg^\sigma(\tnu) | \cF_{\nu+1}^\sigma] \le &[(\qsig(\tnu))^2 + (\usig(\tnu))^2]\Delta \tnu\\
    &+\hat{C}Q^{2m_0}\prob[\disaster | \cF_{\nu+1}^\sigma] \\
    &+ Q^{2m_0}\cdot(T-t_{\nu+1})\cdot(\Delta t_\mx)^{1/20}.
\end{split}
\end{equation}
Estimates \eqref{eq: pcs AII} and \eqref{eq: pcs BII} hold in Case II.

Combining estimates \eqref{eq: pcs AI}, \eqref{eq: pcs BI}, \eqref{eq: pcs AII}, \eqref{eq: pcs BII}, we obtain the following inequalities, which hold in both Cases I and II:
\begin{align*}
    \E[\ctg^\sigma(\tnu) | \cF_{\nu+1}^\sigma] + \hat{C}&Q^{2m_0}\prob[\disaster | \cF_{\nu+1}^\sigma]\\
    \ge & S(\qsig(t_{\nu+1}),t_{\nu+1},\zosig(t_{\nu+1}),\ztsig(t_{\nu+1}))\cdot \mathbbm{1}_{\casei}\\
    & + C\hat{Q}^{2m_0}\mathbbm{1}_{\caseii} + [(\qsig(\tnu))^2 + (\usig(\tnu))^2]\Delta \tnu\cdot \mathbbm{1}_{\casei}\\
    & + \E\Big[ \sum_{t_{\nu+1}\le t_\mu < \tau} (\discrep^\sigma(t_\mu))^2 \Delta t_\mu | \cF_{\nu+1}^\sigma\Big]\\
    & - Q^{2m_0}\cdot (T-t_{\nu+1})\cdot(\Delta t_\mx)^{1/20};
\end{align*}
and if $\sigma = \tsig$, then
\begin{align*}
    \E[\ctg^\sigma(\tnu) | \cF_{\nu+1}^\sigma] \le &S(\qsig(t_{\nu+1}),t_{\nu+1},\zosig(t_{\nu+1}),\ztsig(t_{\nu+1}))\cdot\mathbbm{1}_{\casei} \\
    &+ [(\qsig(\tnu))^2 + (\usig(\tnu))^2]\Delta \tnu \\
    &+ \hat{C}Q^{2m_0}\prob[\disaster | \cF_{\nu+1}^\sigma] \\
    &+ Q^{2m_0}\cdot (T- t_{\nu+1}) \cdot (\Delta t_\mx)^{1/20}.
\end{align*}

We now cease conditioning on $\cF_{\nu+1}^\sigma$ and instead condition on $\cF_{\nu}^\sigma$. From the two preceding inequalities we obtain the following estimates, valid whenever \eqref{eq: pcs 9} holds.
\begin{equation}\label{eq: pcs A*}\tag{A*}
    \begin{split}
        \E[\ctg^\sigma(\tnu) | \cF_\nu^\sigma]& + \hat{C} Q^{2m_0} \prob[\disaster | \cF_\nu^\sigma]\\
         \ge &\E[S(\qsig(t_{\nu+1}),t_{\nu+1},\zosig(t_{\nu+1}),\ztsig(t_{\nu+1}))\cdot \mathbbm{1}_{\casei}|\cF_\nu^\sigma]\\
         &+\E[\hat{C}Q^{2m_0}\mathbbm{1}_{\caseii} | \cF_\nu^\sigma] \\
         &+ [(\qsig(\tnu))^2 + (\usig(\tnu))^2]\Delta \tnu \cdot \E[\mathbbm{1}_{\casei} | \cF_\nu^\sigma]\\
         &+\E\Big[ \sum_{t_{\nu+1}\le t_\mu < \tau} (\discrep^\sigma(t_\mu))^2 \Delta t_\mu | \cF_\nu^\sigma\Big]\\
         &-Q^{2m_0}\cdot(T-t_{\nu+1})\cdot (\Delta t_\mx)^{1/20};
    \end{split}
\end{equation}
and if $\sigma = \tsig$ then
\begin{equation}\label{eq: pcs B*}\tag{B*}
    \begin{split}
        \E[\ctg^\sigma(\tnu) | \cF_\nu^\sigma] \le&\E[S(\qsig(t_{\nu+1}), t_{\nu+1}, \zosig(t_{\nu+1}), \ztsig(t_{\nu+1}))]\cdot \mathbbm{1}_{\casei} | \cF_\nu^\sigma]\\
        &+[(\qsig(\tnu))^2 + (\usig(\tnu))^2]\Delta \tnu \\
        &+ \hat{C}Q^{2m_0}\prob[\disaster | \cF_\nu^\sigma]\\
        &+Q^{2m_0}\cdot(T-t_{\nu+1})\cdot (\Delta t_\mx)^{1/20}.
    \end{split}
\end{equation}

Next, from Lemma \ref{lem: post exp}, recall the event
\begin{equation}\label{eq: pcs 11}
    \tame(\nu) = \{ |\Delta \qsig(\tnu)|, |\Delta \zosig(\tnu)|, |\Delta \ztsig(\tnu)| \le 2(\Delta \tnu)^{2/5}\},
\end{equation}
and the estimate
\begin{equation}\label{eq: pcs 12}
    \prob[\nottame(\nu) | \cF_\nu^\sigma] \le (\Delta \tnu)^{20}.
\end{equation}
(Note that Lemma \ref{lem: post exp} applies, thanks to \eqref{eq: pcs 9}.)
If $\tame(\nu)$ occurs, then, since
\[
(\qsig(\tnu),\zosig(\tnu),\ztsig(\tnu))\in\cU
\]
(by the hypothesis of the present lemma), we have 
\[
(\qsig(t_{\nu+1}),\zosig(t_{\nu+1}),\ztsig(t_{\nu+1}))\in\cU^+,
\]
hence
\[
0 \le S(\qsig(t_{\nu+1}),t_{\nu+1},\zosig(t_{\nu+1}),\ztsig(t_{\nu+1})) \le C Q^{m_0}.
\]
Therefore, if we take $\hat{C}$ large enough, then
\begin{multline}\label{eq: pcs 13}
S(\qsig(t_{\nu+1}),t_{\nu+1}, \zosig(t_{\nu+1}), \ztsig(t_{\nu+1})) \cdot \mathbbm{1}_{\casei} + \frac{1}{2}\hat{C} Q^{2m_0}\cdot \mathbbm{1}_{\caseii}\\ \ge S(\qsig(t_{\nu+1}),t_{\nu+1},\zosig(t_{\nu+1}),\ztsig(t_{\nu+1}))\cdot \mathbbm{1}_{\tame(\nu)}
\end{multline}
and
\begin{equation}\label{eq: pcs 14}
    \frac{1}{2}\hat{C}Q^{2m_0}\cdot \mathbbm{1}_{\caseii} \ge [(\qsig(\tnu))^2+(\usig(\tnu))^2]\Delta \tnu \cdot \mathbbm{1}_{\tame(\nu)}\cdot \mathbbm{1}_{\caseii}.
\end{equation}
Putting \eqref{eq: pcs 13} and \eqref{eq: pcs 14} into \eqref{eq: pcs A*}, we learn that
\begin{equation}\label{eq: pcs Asharp}\tag{A\#}
    \begin{split}        \E[\ctg^\sigma(\tnu)|\cF_\nu^\sigma] &+ \hat{C}Q^{2m_0} \prob[\disaster | \cF_\nu^\sigma]\\
        \ge &\E[S(\qsig(t_{\nu+1}),t_{\nu+1},\zosig(t_{\nu+1}),\ztsig(t_{\nu+1}))\cdot \mathbbm{1}_{\tame(\nu)} | \cF_\nu^\sigma]\\
        & + [(\qsig(\tnu))^2 + (\usig(\tnu))^2]\Delta \tnu \prob[\tame(\nu)|\cF_\nu^\sigma]\\
        &+\E\Big[\sum_{t_{\nu+1}\le t_\mu < \tau} (\discrep^\sigma(t_\mu))^2 \Delta t_\mu \Big| \cF_\nu^\sigma\Big]\\
        & - Q^{2m_0} \cdot (T-t_{\nu+1})\cdot (\Delta t_\mx)^{1/20}.
    \end{split}
\end{equation}
To obtain an analogous result from \eqref{eq: pcs B*}, we note that
\begin{align*}
S(\qsig(t_{\nu+1}),&t_{\nu+1},\zosig(t_{\nu+1}),\ztsig(t_{\nu+1}))\cdot\mathbbm{1}_{\casei} \\ \le &S(\qsig(t_{\nu+1}),t_{\nu+1},\zosig(t_{\nu+1}),\ztsig(t_{\nu+1}))\cdot\mathbbm{1}_{(\qsig(t_{\nu+1}),\zosig(t_{\nu+1}),\ztsig(t_{\nu+1}))\in\cU^+}\\
\le & S(\qsig(t_{\nu+1}),t_{\nu+1},\zosig(t_{\nu+1}), \ztsig(t_{\nu+1}))\cdot \mathbbm{1}_{\tame(\nu)} \\&+ C Q^{2m_0}\cdot \mathbbm{1}_{\nottame(\nu)},
\end{align*}
and therefore
\begin{equation}\label{eq: pcs 15}
\begin{split}    \E[S(\qsig(t_{\nu+1}&),t_{\nu+1},\zosig(t_{\nu+1}),\ztsig(t_{\nu+1}))\cdot \mathbbm{1}_{\casei} | \cF_\nu^\sigma] \\
\le& \E[S(\qsig(t_{\nu+1}),t_{\nu+1},\zosig(t_{\nu+1}),\ztsig(t_{\nu+1}))\cdot \mathbbm{1}_{\tame(\nu)} | \cF_\nu^\sigma]\\& + CQ^{2m_0}(\Delta \tnu)^{20},
\end{split}
\end{equation}
thanks to \eqref{eq: pcs 12}.

Putting \eqref{eq: pcs 15} into \eqref{eq: pcs B*}, we obtain for $\sigma = \tsig$ the inequality
\begin{equation}\label{eq: pcs Bsharp}\tag{B\#}
    \begin{split}
        \E[\ctg^\sigma(\tnu) | &\cF_\nu^\sigma]\\
        \le& \E[S(\qsig(t_{\nu+1}),t_{\nu+1},\zosig(t_{\nu+1}),\ztsig(t_{\nu+1}))\cdot\mathbbm{1}_{\tame(\nu)} | \cF_\nu^\sigma]\\
        & + CQ^{2m_0}(\Delta \tnu)^{20} + [(\qsig(\tnu))^2 + (\usig(\tnu))^2] \Delta \tnu \\
        &+ \hat{C}Q^{2m_0}\prob[\disaster | \cF_\nu^\sigma] \\ &+ Q^{2m_0}(T-t_{\nu+1})\cdot (\Delta t_\mx)^{1/20}.
    \end{split}
\end{equation}
Estimates \eqref{eq: pcs Asharp} (for general $\sigma$) and \eqref{eq: pcs Bsharp} (for $\sigma = \tsig$) hold whenever \eqref{eq: pcs 9} is satisfied.

We next study the quantity
\[
\E[S(\qsig(t_{\nu+1}),t_{\nu+1},\zosig(t_{\nu+1}),\ztsig(t_{\nu+1}))\cdot \mathbbm{1}_{\tame(\nu)} | \cF_\nu^\sigma],
\]
which appears in \eqref{eq: pcs Asharp} and \eqref{eq: pcs Bsharp}.

Thanks to \eqref{eq: pcs 9}, \eqref{eq: pcs 4}, and \eqref{eq: pcs 11}, the points $(\qsig(\tnu),\zosig(\tnu),\ztsig(\tnu))$ and $(\qsig(t_{\nu+1}),\zosig(t_{\nu+1}),\ztsig(t_{\nu+1}))$ both lie in $\cU^+$ provided $\tame(\nu)$ holds; see \eqref{eq: pcs 3}. Recall that the third derivatives of $S$ are assumed to be bounded a.e. by $CQ^{m_0}$ on $\cU^+$. Therefore, if $\tame(\nu)$ holds, then
\begin{equation}\label{eq: pcs 16}
    \begin{split}
\Big|S(\qsig&(t_{\nu+1}),t_{\nu+1},\zosig(t_{\nu+1}),\ztsig(t_{\nu+1})) \\
&- \sum_{p_1 + p_2 + p_3 + p_4 \le 2}\Big\{ \frac{(\partial_t^{p_1}\partial_q^{p_2}\partial_{\zo}^{p_3} \partial_{\zt}^{p_4}S)}{p_1!p_2!p_3!p_4!} (\Delta \tnu)^{p_1} (\Delta \qsig(\tnu))^{p_2}\\
&\qquad \qquad\qquad \qquad \cdot (\Delta \zosig(\tnu))^{p_3}(\Delta \ztsig(\tnu))^{p_4}\Big\} \Big|\\
&\le C Q^{m_0}\max\{(\Delta \tnu), |\Delta \qsig(\tnu)|, |\Delta \zosig(\tnu)|, |\Delta \ztsig(\tnu)|\}^3\\
&\le C' Q^{m_0}  [(\Delta \tnu)^{2/5}]^3,
    \end{split}
\end{equation}
where the partials of $S$ are evaluated at $(\qsig(\tnu), \tnu, \zosig(\tnu),\ztsig(\tnu))$, and we have made use of \eqref{eq: pcs 11}.

Moreover, the summands in \eqref{eq: pcs 16} satisfy
\begin{multline*}
| (\partial_t^{p_1} \cdots \partial_{\zt}^{p_4} S) (\Delta \tnu)^{p_1} (\Delta \qsig(\tnu))^{p_2}(\Delta \zosig(\tnu))^{p_3}(\Delta \ztsig(\tnu))^{p_4} |\\ \le C Q^{m_0} (\Delta \tnu)^{p_1 + \frac{2}{5}[p_2 + p_3 + p_4]}
\end{multline*}
whenever \eqref{eq: pcs 9} and $\tame(\nu)$ hold.

Consequently, \eqref{eq: pcs 16} implies that
\begin{equation}\label{eq: pcs 17}
    \begin{split}
        \Big|S(\qsig&(t_{\nu+1}),t_{\nu+1},\zosig(t_{\nu+1}),\ztsig(t_{\nu+1}))\cdot\mathbbm{1}_{\tame(\nu)} \\
&- \sum_{\substack{p_1 + p_2 + p_3 + p_4 \le 2 \\ p_1 + \frac{2}{5}(p_2+p_3+p_4) \le 1}}\Big\{ \frac{(\partial_t^{p_1}\partial_q^{p_2}\partial_{\zo}^{p_3} \partial_{\zt}^{p_4}S)}{p_1!p_2!p_3!p_4!} (\Delta \tnu)^{p_1} (\Delta \qsig(\tnu))^{p_2}\\
&\qquad \qquad\qquad \qquad\quad \cdot (\Delta \zosig(\tnu))^{p_3}(\Delta \ztsig(\tnu))^{p_4}\cdot \mathbbm{1}_{\tame(\nu)}\Big\} \Big|\\
& \le C Q^{m_0}(\Delta \tnu)^{6/5}.
    \end{split}
\end{equation}
The summands entering into \eqref{eq: pcs 17} arise from $S$, $\partial_t S$, $\partial_q S$, $\partial_{\zo} S$, $\partial_{\zt} S$, $\partial_q^2 S$, $\partial_{q\zo} S$, $\partial_{q\zt}S$, $\partial_{\zo}^2 S$, $\partial_{\zo \zt} S$, $\partial_{\zt}^2 S$. We conclude that
\begin{equation}\label{eq: pcs 18}
\begin{split}
     |S(\qsig(t_{\nu+1}),t_{\nu+1}, \zosig(t_{\nu+1}),\ztsig(t_{\nu+1}))\cdot \mathbbm{1}_{\tame(\nu)} - &\sum_{i =0}^{10}\term\;i |\\ &\le C Q^{m_0} (\Delta \tnu)^{6/5}
\end{split}
\end{equation}
whenever \eqref{eq: pcs 9} holds, where
\begin{flalign}
    &\term\;0= S \cdot \mathbbm{1}_{\tame(\nu)},\label{eq: pcs 19}&\\
    &\term\;1=(\partial_t S)\cdot (\Delta \tnu)\cdot\mathbbm{1}_{\tame(\nu)}, \label{eq: pcs 20}\\
    &\term \; 2 = (\partial_qS) \cdot (\Delta \qsig(\tnu))\cdot \mathbbm{1}_{\tame(\nu)}, \label{eq: pcs 21}\\
    &\term \; 3 = (\partial_{\zo} S)\cdot (\Delta \zosig(\tnu))\cdot\mathbbm{1}_{\tame(\nu)}, \label{eq: pcs 22}\\
    &\term \; 4 = (\partial_{\zt}S)\cdot (\Delta\ztsig(\tnu))\cdot\mathbbm{1}_{\tame(\nu)}, \label{eq: pcs 23}\\
    &\term \; 5 = \frac{1}{2}(\partial_q^2 S)\cdot (\Delta \qsig(\tnu))^2\cdot\mathbbm{1}_{\tame(\nu)}, \label{eq: pcs 24}\\
    &\term \; 6 = (\partial_{q\zo} S)\cdot(\Delta \qsig(\tnu))(\Delta \zosig(\tnu))\cdot\mathbbm{1}_{\tame(\nu)}, \label{eq: pcs 25}\\
    &\term \; 7 = (\partial_{q\zt}S) \cdot(\Delta \qsig(\tnu))(\Delta \ztsig(\tnu))\cdot\mathbbm{1}_{\tame(\nu)}, \label{eq: pcs 26}\\
    &\term \; 8 = \frac{1}{2}(\partial_{\zo}^2 S)\cdot(\Delta \zosig(\tnu))^2\cdot \mathbbm{1}_{\tame(\nu)}, \label{eq: pcs 27}\\
    &\term \; 9 = (\partial_{\zo\zt} S)\cdot (\Delta \zosig(\tnu))(\Delta \ztsig(\tnu))\cdot \mathbbm{1}_{\tame(\nu)}, \label{eq: pcs 28}\\
    &\term \; 10 = \frac{1}{2} (\partial_{\zt}^2 S)\cdot (\Delta \ztsig(\tnu))^2 \cdot \mathbbm{1}_{\tame(\nu)}. \label{eq: pcs 29}
\end{flalign}
Here, again, $S$ and its partial derivatives are evaluated at \[
(\qsig(\tnu), \tnu, \zosig(\tnu),\ztsig(\tnu));
\] hence these $(\partial^\alpha S)$ are deterministic once we condition on $\cF_\nu^\sigma$. Consequently,
\begin{flalign*}
    &\E[\term\;0 | \cF_\nu^\sigma]= S \cdot \prob[\tame(\nu) |\cF_\nu^\sigma],&\\
    &\E[\term\;1 | \cF_\nu^\sigma]=(\partial_t S)\cdot (\Delta \tnu)\cdot\prob[\tame(\nu) |\cF_\nu^\sigma], \\
    &\E[\term\;2| \cF_\nu^\sigma] = (\partial_qS) \cdot \E[\Delta \qsig(\tnu)\cdot\mathbbm{1}_{\tame(\nu)} |\cF_\nu^\sigma], \\
    &\E[\term\;3 | \cF_\nu^\sigma] = (\partial_{\zo} S)\cdot\E[\Delta \zosig(\tnu)\cdot \mathbbm{1}_{\tame(\nu)} |\cF_\nu^\sigma], \\
    &\E[\term\;4 | \cF_\nu^\sigma] = (\partial_{\zt}S)\cdot \E[\Delta\ztsig(\tnu)\cdot\mathbbm{1}_{\tame(\nu)} |\cF_\nu^\sigma], \\
    &\E[\term\;5 | \cF_\nu^\sigma] = \frac{1}{2}(\partial_q^2 S)\cdot \E[(\Delta \qsig(\tnu))^2\cdot\mathbbm{1}_{\tame(\nu)} |\cF_\nu^\sigma], \\
    &\E[\term\;6 | \cF_\nu^\sigma]= (\partial_{q\zo} S)\cdot\E[(\Delta \qsig(\tnu))(\Delta \zosig(\tnu))\cdot\mathbbm{1}_{\tame(\nu)} |\cF_\nu^\sigma], \\
    &\E[\term\;7 | \cF_\nu^\sigma] = (\partial_{q\zt}S) \cdot\E[(\Delta \qsig(\tnu))(\Delta \ztsig(\tnu))\cdot\mathbbm{1}_{\tame(\nu)} |\cF_\nu^\sigma], \\
    &\E[\term\;8 | \cF_\nu^\sigma] = \frac{1}{2}(\partial_{\zo}^2 S)\cdot \E[(\Delta \zosig(\tnu))^2\cdot\mathbbm{1}_{\tame(\nu)} |\cF_\nu^\sigma], \\
    &\E[\term\;9 | \cF_\nu^\sigma]= (\partial_{\zo\zt} S)\cdot  \E[(\Delta \zosig(\tnu))(\Delta \ztsig(\tnu))\cdot\mathbbm{1}_{\tame(\nu)} |\cF_\nu^\sigma], \\
    &\E[\term\;10 | \cF_\nu^\sigma] = \frac{1}{2} (\partial_{\zt}^2 S)\cdot \E[(\Delta \ztsig(\tnu))^2 \cdot \mathbbm{1}_{\tame(\nu)} |\cF_\nu^\sigma]. 
\end{flalign*}
The expectations on the right-hand sides have been computed modulo a small error in Lemma \ref{lem: post exp}. Applying that lemma, recalling that $|\partial^\alpha S| \le C Q^{m_0}$ for $|\alpha | \le 2$, and substituting the results into \eqref{eq: pcs 18}, we find that
\begin{equation}\label{eq: pcs 19.5}
    \begin{split}
        \E[S(\qsig&(t_{\nu+1}),t_{\nu+1},\zosig(t_{\nu+1}),\ztsig(t_{\nu+1}))\cdot\mathbbm{1}_{\tame(\nu)} | \cF_\nu^\sigma] \\
        &= S + (\Delta \tnu) \{\partial_t S + [\bar{a} \qsig(\tnu) + \usig(\tnu)]\partial_q S + \bar{a}\cdot(\qsig(\tnu))^2\partial_{\zo}S\\
        &\qquad \qquad \qquad + (\qsig(\tnu))^2 \partial_{\zt} S + \frac{1}{2} \partial_q^2 S + \qsig(\tnu)\partial_{q\zo} S \\
        &\qquad \qquad \qquad+ \frac{1}{2}(\qsig(\tnu))^2\partial_{\zo}^2 S\} + \error,
    \end{split}
\end{equation}
where $S$ and its derivatives are evaluated at $(\qsig(\tnu),\tnu,\zosig(\tnu),\ztsig(\tnu)),$ $\bar{a}$ denotes $\bar{a}(\zosig(\tnu),\ztsig(\tnu)),$ and 
\begin{equation}\label{eq: pcs 30}
\begin{split}
    |\error| &\le C Q^{m_0} (\Delta \tnu)^{6/5} + C Q^{m_0+4} (\Delta t_\mx)^{1/4} \Delta \tnu\\
    & \le CQ^{m_0} (\Delta t_\mx)^{1/5} \Delta \tnu.
\end{split}
\end{equation}
(We have used our assumption \eqref{eq: pcs 1}.) Since $S$ satisfies our PDE \eqref{eq: pde 1}, equation \eqref{eq: pcs 19.5} may be rewritten in the equivalent form
\begin{equation}\label{eq: pcs 20.5}
    \begin{split}
        \E[S(\qsig&(t_{\nu+1}),t_{\nu+1},\zosig(t_{\nu+1}),\ztsig(t_{\nu+1}))\cdot \mathbbm{1}_{\tame(\nu)} | \cF_\nu^\sigma] \\
        =& S(\qsig(\tnu),\tnu,\zosig(\tnu),\ztsig(\tnu)) \\
        &+ (\Delta \tnu) \cdot \{ (\partial_q S) \cdot [\usig(\tnu) - u_\op(\qsig(\tnu),\tnu,\zosig(\tnu),\ztsig(\tnu))] \\
        &\qquad \qquad\quad- [(\qsig(\tnu))^2 + (u_\op(\qsig(\tnu),\tnu,\zosig(\tnu),\ztsig(\tnu)))^2]\} \\
        &+ \error,
    \end{split}
\end{equation}
with
\begin{equation}\label{eq: pcs 31}
|\error| \le C Q^{m_0} (\Delta t_\mx)^{1/5}\Delta \tnu.
\end{equation}
Recalling that $\partial_q S = - 2u_\op$ (see \eqref{eq: pde 2}), we see that the expression in curly brackets in \eqref{eq: pcs 20.5} is equal to
\[
\{-2u_\op \cdot [\usig - u_\op] - (\qsig)^2 - u_\op^2\} = \{(\usig - u_\op)^2 - (\qsig)^2 - (\usig)^2\},
\]
where we have written $\usig$ for $\usig(\tnu)$, and $u_\op$ for $u_\op(\qsig(\tnu),\tnu,\zosig(\tnu),\ztsig(\tnu))$. Moreover,
\begin{align*}
\usig - u_\op &= \usig(\tnu) - u_\op(\qsig(\tnu),\tnu,\zosig(\tnu),\ztsig(\tnu))\\
& \equiv \discrep^\sigma(\tnu)
\end{align*}
(see \eqref{eq: pcs 8}). Therefore, \eqref{eq: pcs 20.5}, \eqref{eq: pcs 31} are equivalent to:
\begin{equation}\label{eq: pcs 32}
    \begin{split}
        \E[S(\qsig&(t_{\nu+1}),t_{\nu+1},\zosig(t_{\nu+1}),\ztsig(t_{\nu+1}))\cdot \mathbbm{1}_{\tame(\nu)}|\cF_\nu^\sigma] \\
        =& S(\qsig(\tnu),\tnu,\zosig(\tnu),\ztsig(\tnu)) \\
        &+ (\Delta \tnu)\{(\discrep^\sigma(\tnu))^2 - (\qsig(\tnu))^2 - (\usig(\tnu))^2\} \\
        &+ \error,
    \end{split}
\end{equation}
with 
\begin{equation}\label{eq: pcs 33}
    |\error| \le C Q^{m_0} (\Delta t_\mx)^{1/5} \Delta \tnu.
\end{equation}
Here, \eqref{eq: pcs 32} and \eqref{eq: pcs 33} are valid wherever \eqref{eq: pcs 9} holds.

We now substitute \eqref{eq: pcs 32}, \eqref{eq: pcs 33} into \eqref{eq: pcs Asharp} and \eqref{eq: pcs Bsharp}, to obtain the following results, valid whenever \eqref{eq: pcs 9} holds.
\begin{equation}\label{eq: pcs Asharpsharp}\tag{A\#\#}
\begin{split}
    \E[\ctg^\sigma&(\tnu) | \cF_\nu^\sigma] +\hat{C}Q^{2m_0}\prob[\disaster | \cF_\nu^\sigma]\\
     \ge & \{ S(\qsig(\tnu),\tnu,\zosig(\tnu),\ztsig(\tnu)) + (\Delta \tnu) (\discrep^\sigma(\tnu))^2\\
    & - [(\qsig(\tnu))^2 + (\usig(\tnu))^2]\Delta \tnu - C Q^{m_0}(\Delta t_\mx)^{1/5} \Delta \tnu\} \\
    & + [(\qsig(\tnu))^2 + (\usig(\tnu))^2]\Delta \tnu \cdot \prob[\tame(\nu) | \cF_\nu^\sigma]\\
    & + \E\Big[ \sum_{t_{\nu+1} \le t_\mu < \tau} (\discrep^\sigma(t_\mu))^2 \Delta t_\mu | \cF_\nu^\sigma\Big]\\
    &-Q^{2m_0}(T-t_{\nu+1})\cdot(\Delta t_\mx)^{1/20};
\end{split}
\end{equation}
and if $\sigma = \tsig$, then (since $\discrep^\sigma(\tnu) = 0)$ we have
\begin{equation}\label{eq: pcs Bsharpsharp}\tag{B\#\#}
    \begin{split}
        \E[\ctg^\sigma(\tnu) | \cF_\nu^\sigma] \le& \{S(\qsig(\tnu),\tnu,\zosig(\tnu),\ztsig(\tnu)) \\
        &+ C Q^{m_0} (\Delta t_\mx)^{1/5} \Delta \tnu\} + CQ^{2m_0}(\Delta \tnu)^{20} \\
        &+ \hat{C}Q^{2m_0}\prob[\disaster | \cF_\nu^\sigma] \\
        &+ Q^{2m_0}(T-t_{\nu+1})\cdot (\Delta t_\mx)^{1/20}.
    \end{split}
\end{equation}
In \eqref{eq: pcs Asharpsharp} we note that
\[
[(\qsig(\tnu))^2 + (\usig(\tnu))^2]\prob[\nottame | \cF_\nu^\sigma] \le C Q^2 (\Delta \tnu)^{20};
\]
see hypothesis \eqref{eq: pcs star} of this lemma. Therefore, in \eqref{eq: pcs Asharpsharp}, the terms
\[
[(\qsig(\tnu))^2 + (\usig(\tnu))^2]\Delta \tnu\cdot \prob[\tame(\nu)]
\]
and
\[
-[(\qsig(\tnu))^2 + (\usig(\tnu))^2]\Delta \tnu
\]
nearly cancel---they produce a term dominated by $CQ^2(\Delta \tnu)^{20}$. Moreover, thanks to \eqref{eq: pcs 10}, we have
\begin{multline*}
    (\Delta \tnu) (\discrep^\sigma(\tnu))^2 + \E\Big[ \sum_{t_{\nu+1} \le t_\mu < \tau} (\discrep^\sigma(t_\mu))^2 \Delta t_\mu \Big| \cF_\nu^\sigma\Big]\\
    = \E\Big[ \sum_{t_\nu \le t_\mu < \tau} (\discrep^\sigma(t_\mu))^2 \Delta t_\mu \Big| \cF_\nu^\sigma\Big].
\end{multline*}
(We have also used the fact that $\discrep^\sigma(\tnu)$ is deterministic when conditioned on $\cF_\nu^\sigma$.)

Finally, our assumptions that $\Delta t_\mx \le Q^{-2000m_0}$ and $Q>C$ ($C$ large enough) imply that
\begin{align*}
    CQ^2 (\Delta \tnu)^{20}& + C Q^{m_0} (\Delta t_\mx)^{1/5} \Delta \tnu + Q^{2m_0} (T-t_{\nu+1})\cdot (\Delta t_\mx)^{1/20} \\ &\le Q^{2m_0} (\Delta t_\mx)^{1/20} \cdot (t_{\nu+1}-\tnu) + Q^{2m_0}(T-t_{\nu+1})\cdot (\Delta t_\mx)^{1/20} \\ &= Q^{2m_0}(T-\tnu)\cdot (\Delta t_\mx)^{1/20}
\end{align*}
and also that the sum of terms
\begin{equation*}
     CQ^{m_0}(\Delta t_\mx)^{1/5}(\Delta \tnu) + C Q^{2m_0} (\Delta \tnu)^{20} + Q^{2m_0}(T-t_{\nu+1})(\Delta t_\mx)^{1/20}
\end{equation*}
is at most
\begin{multline*}
    Q^{2m_0} (\Delta t_\mx)^{1/20}\Delta \tnu + 
    Q^{2m_0}(T-t_{\nu+1})\cdot (\Delta t_\mx)^{1/20}\\
     = Q^{2m_0}(T-\tnu)\cdot (\Delta t_\mx)^{1/20}.
\end{multline*}
In view of the above remarks, \eqref{eq: pcs Asharpsharp} and \eqref{eq: pcs Bsharpsharp} imply the following results, valid whenever \eqref{eq: pcs 9} holds.
\begin{align*}
    \E[\ctg^\sigma&(\tnu) | \cF_\nu^\sigma] + \hat{C} Q^{2m_0}\prob[\disaster | \cF_\nu^\sigma]
    \\
    \ge & S(\qsig(\tnu),\tnu,\zosig(\tnu),\ztsig(\tnu)) + \E\Big[ \sum_{\tnu \le t_\mu < \tau} (\discrep^\sigma(t_\mu))^2 \Delta t_\mu \Big| \cF_\nu^\sigma \Big]\\
    &- Q^{2m_0}\cdot (T-\tnu)\cdot (\Delta t_\mx)^{1/20};
\end{align*}
and if $\sigma = \tsig$, then
\begin{align*}
    \E[\ctg^\sigma(\tnu) | \cF_\nu^\sigma] \le &S(\qsig(\tnu),\tnu,\zosig(\tnu),\ztsig(\tnu)) + \hat{C}Q^{2m_0}\prob[\disaster | \cF_\nu^\sigma]\\& + Q^{2m_0}\cdot (T-\tnu)\cdot (\Delta t_\mx)^{1/20}.
\end{align*}
These are precisely our desired conclusions \eqref{eq: pcs A} and \eqref{eq: pcs B}. Our downward induction on $\nu$ is complete, thus proving Lemma \ref{lem: comp strategies}.
\end{proof}

We now draw conclusions from Lemma \ref{lem: comp strategies}. Setting $\nu = 0$, we obtain the following results, comparing the allegedly optimal strategy $\tsig$ to the competing strategy $\sigma$.
\begin{equation}\label{eq: pcs 34}
    \begin{split}
        \E\Big[ \sum_{0 \le t_\mu < \tau}& \{(\qsig(t_\mu))^2 + (\usig(t_\mu))^2 \}\Delta t_\mu \Big] + \hat{C}Q^{2m_0} \prob[\disaster]\\
        \ge & S(q_0,0,0,0) + \E \Big[ \sum_{0 \le t_\mu < \tau} (\discrep^\sigma(t_\mu))^2 \Delta t_\mu \Big] \\
        &- Q^{2m_0}(\Delta t_\mx)^{1/20}T;
    \end{split}
\end{equation}
and if $\sigma = \tsig$, then
\begin{equation}\label{eq: pcs 35}
    \begin{split}
        \E\Big[ \sum_{0 \le t_\mu < \tau} \{ (\qsig(t_\mu))^2 + (\usig&(t_\mu))^2\} \Delta t_\mu\Big]\\ \le &  S(q_0,0,0,0) + \hat{C} Q^{2m_0} \prob[\disaster] \\
        &+Q^{2m_0}(\Delta t_\mx)^{1/20}T.
    \end{split}
\end{equation}
From \eqref{eq: pcs 2}, \eqref{eq: pcs 4}, \eqref{eq: pcs 5} and Lemma \ref{lem: bayesian rare events}, we have
\begin{equation}\label{eq: pcs 35.5}
    \prob[\disaster] \le C \exp(-cQ).
\end{equation}

Let us investigate what happens if $\disaster$ occurs.

Since 
\[
|\usig(t_\mu)| \le C [|\qsig(t_\mu)| + 1],
\]
we have
\[
\sum_{0 \le \mu < N} \{ (\qsig(t_\mu))^2 + (\usig(t_\mu))^2\} \Delta t_\mu \le C ( \max_\mu |\qsig(t_\mu)| )^2 + C.
\]
Also,
\begin{align*}
    |\discrep^\sigma(t_\mu)| &= |\usig(t_\mu) - u_\op(\qsig(t_\mu),t_\mu,\zosig(t_\mu),\ztsig(t_\mu))|\\
    &\le |\usig(t_\mu)| + C[|\qsig(t_\mu)| + 1]
\end{align*}
by our PDE Assumption (see \eqref{eq: pde 6}). Hence,
\[
\sum_{0 \le \mu < N} (\discrep^\sigma(t_\mu))^2 \Delta t_\mu \le C (\max_\mu |\qsig(t_\mu)|)^2 + C.
\]
The above remarks and Lemma \ref{lem: bayesian rare events} yield the estimates:
\[
\E\Big[ \Big( \sum_{0\le \mu < N}\{ (\qsig(t_\mu))^2 + (\usig(t_\mu))^2\} \Delta t_\mu \Big)^2\Big] \le C
\]
and
\[
\E\Big[ \Big( \sum_{0 \le \mu < N}(\discrep^\sigma(t_\mu))^2 \Delta t_\mu\Big)^2\Big]\le C.
\]
Consequently, Cauchy-Schwarz and \eqref{eq: pcs 35.5} imply that
\begin{equation}\label{eq: pcs 36}
\begin{split}
    \E\Big[ \Big( \sum_{0\le \mu < N} \{ (\qsig(t_\mu))^2 + (\usig(t_\mu))^2\} \Delta t_\mu \Big) \cdot &\mathbbm{1}_{\disaster} \Big]  \\ &\le C \cdot ( \prob[\disaster])^{1/2}\\
    & \le C' \exp(-c' Q)
\end{split}
\end{equation}
and
\begin{equation}\label{eq: pcs 37}
    \begin{split}
        \E\Big[ \Big( \sum_{0 \le \mu < N} (\discrep^\sigma(t_\mu))^2 \Delta t_\mu\Big) \cdot &\mathbbm{1}_{\disaster}\Big]\\
        & \le C \cdot (\prob[\disaster])^{1/2} \\ 
        & \le C' \exp(-cQ).
    \end{split}
\end{equation}
So we have controlled the consequences of $\disaster$. 

On the other hand, if $\disaster$ does not occur, then $\tau= T$; see \eqref{eq: pcs 5} and \eqref{eq: pcs 6}.

Therefore,
\begin{equation}\label{eq: pcs 38}
    \begin{split}
        \E\Big[ \Big( \sum_{0\le \mu < N} \{(\qsig(t_\mu))^2 + &(\usig(t_\mu))^2 \} \Delta t_\mu \Big) \cdot \mathbbm{1}_{\ndisaster}\Big]\\
        &\le \E\Big[ \sum_{0 \le t_\mu < \tau} \{ (\qsig(t_\mu))^2 + (\usig(t_\mu))^2 \} \Delta t_\mu\Big]
    \end{split}
\end{equation}
and
\begin{equation}\label{eq: pcs 39}
    \begin{split}
        \E\Big[ \Big( \sum_{0 \le \mu < N}(\discrep^\sigma(t_\mu))^2 \Delta t_\mu \Big)& \cdot \mathbbm{1}_{\ndisaster}\Big]\\
        &\le \E\Big[ \sum_{0 \le t_\mu < \tau} (\discrep^\sigma(t_\mu))^2 \Delta t_\mu\Big].
    \end{split}
\end{equation}
Also, obviously,
\begin{equation}\label{eq: pcs 40}
    \begin{split}
        \E\Big[ \sum_{0 \le \mu < N} \{ (\qsig(t_\mu))^2 + (\usig&(t_\mu))^2\} \Delta t_\mu \Big]\\ &\ge \E\Big[ \sum_{0 \le t_\mu < \tau} \{ (\qsig(t_\mu))^2 + (\usig(t_\mu))^2 \} \Delta t_\mu \Big].
    \end{split}
\end{equation}
Substituting \eqref{eq: pcs 40}, \eqref{eq: pcs 39}, \eqref{eq: pcs 35.5} into \eqref{eq: pcs 34}, we find that
\begin{align*}
    \E\Big[ \sum_{0 \le \mu < N} \{ (\qsig&(t_\mu))^2 + (\usig(t_\mu))^2\} \Delta t_\mu\Big] + C Q^{2m_0} \exp(-cQ)\\
    \ge & S(q_0,0,0,0) + \E\Big[\Big( \sum_{0 \le \mu < N} (\discrep^\sigma(t_\mu))^2\Big)\cdot \mathbbm{1}_{\ndisaster}\Big]\\
    & - Q^{2m_0} (\Delta t_\mx)^{1/20}T.
\end{align*}
Together with \eqref{eq: pcs 37}, this in turn yields:
\begin{equation}\label{eq: pcs 41}
    \begin{split}
        \E\Big[ \sum_{0 \le \mu < N} \{ (\qsig(t_\mu))^2 + (\usig&(t_\mu))^2\} \Delta t_\mu\Big]\\
        \ge & S(q_0,0,0,0) + \E\Big[ \sum_{0 \le \mu < N} (\discrep^\sigma(t_\mu))^2 \Big) \Big]\\& - CQ^{2m_0}\exp(-cQ) - Q^{2m_0}(\Delta t_\mx)^{1/20} T.
    \end{split}
\end{equation}
Similarly, suppose $\sigma = \tsig$. Then, substituting \eqref{eq: pcs 38} and \eqref{eq: pcs 35.5} into \eqref{eq: pcs 35}, we find that
\begin{multline*}
    \E\Big[\Big( \sum_{0 \le \mu < N} \{ (\qsig(t_\mu))^2 + (\usig(t_\mu))^2\} \Delta t_\mu\Big) \cdot \mathbbm{1}_{\ndisaster}\Big]\\
    \le S(q_0,0,0,0) + C Q^{2m_0} \exp(-cQ) + Q^{2m_0}(\Delta t_\mx)^{1/20}T.
\end{multline*}
Together with \eqref{eq: pcs 36}, this implies that
\begin{equation}\label{eq: pcs 42}
\begin{split}
    \E\Big[ \sum_{0 \le \mu < N} &\{ (\qsig(t_\mu))^2 + (\usig(t_\mu))^2\}\Delta t_\mu\Big] \\
    &\le S(q_0,0,0,0) + CQ^{2m_0}\exp(-cQ) + Q^{2m_0}(\Delta t_\mx)^{1/20}T.
\end{split}
\end{equation}
We have proven \eqref{eq: pcs 41} for $Q\ge C$ and $\Delta t_\mx \le Q^{-2000m_0}$; see \eqref{eq: pcs 1}. Similarly, \eqref{eq: pcs 42} holds for $Q \ge C$, $\Delta t_\mx \le Q^{-2000m_0}$, $\sigma = \tsig$.

Now suppose $\varepsilon > 0$ is given. We pick $Q \ge C $ so large that 
\[
C Q^{2m_0} \exp(-cQ) < \varepsilon / 2
\]
in \eqref{eq: pcs 41}, \eqref{eq: pcs 42}, and also so large that 
\[
Q^{2m_0}(Q^{-2000m_0})^{1/20}T<\varepsilon/2.
\]
Then, if $\Delta t_\mx < Q^{-2000m_0}$, we obtain from \eqref{eq: pcs 41} and \eqref{eq: pcs 42} the estimates
\begin{multline*}
    \E\Big[ \sum_{0 \le \mu < N} \{ (\qsig(t_\mu))^2 + (\usig(t_\mu))^2\}\Delta t_\mu\Big]\\
    \ge S(q_0,0,0,0) + \E\Big[ \sum_{0\le\mu < N} (\discrep^\sigma(t_\mu))^2 \Delta t_\mu \Big] - \varepsilon,
\end{multline*}
and if $\sigma = \tsig$ then
\[
    \E\Big[ \sum_{0 \le \mu < N} \{ (\qsig(t_\mu))^2 + (\usig(t_\mu))^2\} \Delta t_\mu\Big]\le S(q_0,0,0,0) + \varepsilon.
\]
Thus, we have proven the following result, \emph{modulo} our PDE Assumption.

\begin{lem}[First Bayesian Main Lemma]\label{lem: first bayesian main}
    Let $\tsig$ be the \textsc{allegedly optimal strategy} for the partition $0 = t_0 < t_1 < \dots < t_N = T$, and let $\varepsilon > 0$ be given.

    Let $\sigma$ be another tame deterministic strategy for the same partition. Assume that $\sigma$ satisfies
    \[
    |\usig(\tnu)| \le C_\etame^\sigma [|\qsig(\tnu)| + 1].
    \]

    Suppose that
    \[
    \Delta t_\emx = \max_\nu (t_{\nu+1} - \tnu)
    \]
    is less than a small enough $\delta > 0$, determined by $\varepsilon$, $C_\etame^\sigma$, and the \textsc{boilerplate constants}. Then
    \begin{multline*}
        \eE\Big[ \sum_{0 \le \nu < N} \{ (\qsig(\tnu))^2 + (\usig(\tnu))^2\}\Delta \tnu \Big] \\ \ge S(q_0,0,0,0) + \eE\Big[ \sum_{0 \le \nu < N} (\discrep^\sigma(\tnu))^2 \Delta \tnu\Big] - \varepsilon
    \end{multline*}
    and
    \[
    \eE\Big[ \sum_{0 \le \nu < N} \{ (q^{\tsig}(\tnu))^2 + (u^{\tsig}(\tnu))^2\}\Delta \tnu\Big] \le S(q_0,0,0,0) + \varepsilon,
    \]
    where
    \[
    \discrep^\sigma(\tnu) = \usig(\tnu) - u_\eop(\qsig(\tnu),\tnu,\zosig(\tnu),\ztsig(\tnu)).
    \]
\end{lem}

\begin{cor}\label{cor: first bayesian main}
    Let $\sigma$, $\tsig$ be as above. If
    \begin{multline*}
        \eE\Big[ \sum_{0 \le \nu < N} \{ (\qsig(\tnu))^2 + (\usig(\tnu))^2\} \Delta \tnu\Big]\\ \le \eE\Big[ \sum_{0 \le \nu < N} \{ (q^{\tsig}(\tnu))^2 + (u^{\tsig}(\tnu))^2\} \Delta \tnu \Big] + \varepsilon,
    \end{multline*}
    then
    \[
    \eE\Big[ \sum_{0 \le \nu < N} (\discrep^\sigma(\tnu))^2 \Delta \tnu\Big] \le 3\varepsilon.
    \]
\end{cor}

\section{Stability of the Allegedly Optimal Strategy}\label{sec: stability}
We begin by setting up the notation for this section.
\begin{itemize}
    \item $\tsig$ denotes the \textsc{allegedly optimal strategy}.
    \item $\sigma$ denotes some other tame deterministic strategy based on the same partition $0 = t_0 < t_1 < \dots < t_N = T$ as $\tsig$.
    \item $\qnusig$ denotes $\qsig(\tnu)$, and $\Delta \qnusig$ denotes $\qsig(t_{\nu+1}) - \qsig(\tnu)$.
    \item $\qnutsig$ denotes $q^{\tsig}(\tnu)$, and $\Delta \qnutsig$ denotes $q^{\tsig}(t_{\nu+1}) - q^{\tsig}(\tnu)$.
    \item $\zonusig$ denotes $\zosig(\tnu)$, and $\Delta \zonusig$ denotes $\zosig(t_{\nu+1}) - \zosig(\tnu)$.
    \item $\zonutsig$ denotes $\zonutsig(\tnu)$, and $\Delta \zonutsig$ denotes $\zo^{\tsig}(t_{\nu+1}) - \zo^{\tsig}(\tnu)$.
    \item $\ztnusig$ denotes $\ztsig(\tnu)$, and $\Delta \ztnusig$ denotes $\ztsig(t_{\nu+1}) - \ztsig(\tnu)$.
    \item $\ztnutsig$ denotes $\ztnutsig(\tnu)$, and $\Delta \ztnutsig$ denotes $\zt^{\tsig}(t_{\nu+1}) - \zt^{\tsig}(\tnu)$.
    \item $\unusig$ denotes $\usig(\tnu)$, and $\unutsig$ denotes $u^{\tsig}(\tnu)$.
\end{itemize}
We recall that
\[
|\usig(\tnu)| \le C_\tame^\sigma [|\qsig(\tnu)| + 1], \qquad |u^{\tsig}(\tnu)| \le C_\tame^\op [|q^{\tsig}(\tnu)| + 1].
\]
In this section, $c$, $C$, $C'$, etc.\ denote constants determined by $C_\tame^\sigma$ and the \textsc{boilerplate constants} (one of which is $C_\tame^\op$). As usual, these symbols may denote different constants in different occurrences.

We fix $a \in [-a_\mx, + a_\mx]$, and condition on $a_\tru = a$. We write $\prob_a[\cdots]$ and $\E_a[\cdots]$ to denote the corresponding probability and expectation, respectively. 

We write $\cF_\nu$ to denote the sigma algebra of events determined by the Brownian motion $(W(t))_{t \in [0, \tnu]}$ (and by $a_\tru = a$).

Recall that
\begin{flalign}
    & \Delta \qnusig = (a\qnusig + \unusig)(\Delta \tnu^*) + \Delta W_\nu &\label{eq: saos 1}\\
    & \Delta \qnutsig = (a\qnutsig + \unutsig)(\Delta \tnu^*) + \Delta W_\nu,\label{eq: saos 2}\\
    & \Delta \zonusig = \qnusig(\Delta \qnusig - \unusig \Delta \tnu)\label{eq: saos 3}\\
    & \Delta \zonutsig = \qnutsig (\Delta \qnutsig - \unutsig \Delta \tnu)\label{eq: saos 4}\\
    & \Delta \ztnusig = (\qnusig)^2 \Delta \tnu,\label{eq: saos 5}\\
    & \Delta \ztnutsig = (\qnutsig)^2 \Delta \tnu,\label{eq: saos 6}
\end{flalign}
and that
\begin{equation}\label{eq: saos 7}
    \qsig_0 = q^{\tsig}_0 = q_0,\qquad \zeta_{1,0}^\sigma = \zeta_{1,0}^{\tsig} = \zeta_{2,0}^\sigma = \zeta_{2,0}^{\tsig} = 0.
\end{equation}
In \eqref{eq: saos 1} and \eqref{eq: saos 2}, $\Delta \tnu^* = \Delta \tnu + O((\Delta \tnu)^2),$ and $\Delta W_\nu$ is a normal random variable with mean 0 and variance $O(\Delta \tnu)$, independent of $\cF_\nu$.

The quantities $\qnusig$, $\qnutsig$, $\unusig$, $\unutsig$, $\zonusig$, $\zonutsig$, $\ztnusig$, $\ztnutsig$ are deterministic once we condition on $\cF_\nu^\sigma$.

As in Section \ref{sec: pcs}, we define
\begin{equation}\label{eq: saos 8}
    \discrep^\sigma(\tnu) = \usig(\tnu) - u_\op(\qsig(\tnu),\tnu, \zosig(\tnu), \ztsig(\tnu)).
\end{equation}
Our goal is to show that if
\begin{equation}\label{eq: saos 9}
    \E_a \Big[ \sum_{0 \le \nu < N} (\discrep^\sigma(\tnu))^2 \Delta \tnu\Big] \;\text{is small},
\end{equation}
then also
\begin{equation}\label{eq: saos 10}
    \E_a\Big[ \sum_{0 \le \nu < N} \{ |\qsig(\tnu) - q^{\tsig}(\tnu)|^2 + |\usig(\tnu) - u^{\tsig}(\tnu)|^2\} \Delta \tnu\Big]\;\text{is small}.
\end{equation}
Here, \eqref{eq: saos 9} asserts that $\sigma$ does something close to what $\tsig$ would do in the circumstances encountered by $\sigma$. On the other hand, \eqref{eq: saos 10} asserts that $\sigma$ and $\tsig$ produce nearly equal outcomes. To show that \eqref{eq: saos 9} implies \eqref{eq: saos 10}, we introduce the vector
\begin{equation}\label{eq: saos 11}
    \mathcal{X}_\nu = 
\begin{pmatrix} \qnusig - \qnutsig \\ \zonusig - \zonutsig \\ \ztnusig - \ztnutsig \end{pmatrix}
\equiv \begin{pmatrix} \cX_{\nu,1}\\ \cX_{\nu,2}\\ \cX_{\nu,3} \end{pmatrix} \in \R^3.
\end{equation}
Thanks to \eqref{eq: saos 7}, we have
\begin{equation}\label{eq: saos 12}
    \cX_0 = 0.
\end{equation}

For a large enough $C$, we introduce a positive number $Q$ satisfying
\begin{equation}\label{eq: saos 13}
    Q \ge C.
\end{equation}
Under the assumption
\begin{equation}\label{eq: saos 14}
    |\qnusig|, |\qnutsig|, |\zonusig|, |\zonutsig|, |\ztnusig|, |\ztnutsig| \le Q,
\end{equation}
we will estimate
\[
\Delta \cX_\nu \equiv \begin{pmatrix}
    \Delta \cX_{\nu,1}\\ \Delta \cX_{\nu,2} \\ \Delta \cX_{\nu,3}
\end{pmatrix}
\equiv \cX_{\nu+1} - \cX_\nu.
\]
Until further notice we fix $\nu$ and assume \eqref{eq: saos 14}.

We write $G_1, G_2, \dots$ to denote random variables satisfying
\begin{flalign}
    & G_i\;\text{is deterministic once we condition on }\cF_\nu \text{ and}\label{eq: saos 15}&\\
    & |G_i| \le C Q^{m_0}, \;\text{with } m_0\;\text{as in our PDE Assumption (see \eqref{eq: pde 5}.}\label{eq: saos 16}
\end{flalign}
To estimate $\Delta \cX_\nu$, we first apply our PDE Assumption (see \eqref{eq: pde 5}) to show that
\begin{align*}
u_\op(\qsig(\tnu),\tnu,&\zosig(\tnu),\ztsig(\tnu)) - u^{\tsig}(\tnu)\\
    = & u_\op(\qsig(\tnu),\tnu, \zosig(\tnu),\ztsig(\tnu)) - u_\op(q^{\tsig}(\tnu),\tnu,\zeta_1^{\tsig}(\tnu),\zeta_2^{\tsig}(\tnu))\\
    = & G_1 [\qnusig - \qnutsig] + G_2 [\zonusig - \zonutsig] + G_3[\ztnusig - \ztnutsig]
\end{align*}
with $G_1$, $G_2$, $G_3$ as in \eqref{eq: saos 15}, \eqref{eq: saos 16}.

Consequently, \eqref{eq: saos 8} implies that
\begin{equation}\label{eq: saos 17}
\usig(\tnu) - u^{\tsig}(\tnu) = \discrep^\sigma(\tnu) + G_1 \cX_{\nu,1} + G_2 \cX_{\nu,2} + G_3 \cX_{\nu,3}.
\end{equation}
We subtract \eqref{eq: saos 2} from \eqref{eq: saos 1} and apply \eqref{eq: saos 17}. Thus,
\begin{multline*}
    \Delta \cX_{\nu,1} = a \cX_{\nu,1}(\Delta \tnu^*) + \discrep^\sigma (\tnu)(\Delta \tnu^*)\\ + G_1 \cX_{\nu,1} (\Delta \tnu^*) + G_2 \cX_{\nu,2}(\Delta \tnu^*) + G_3 \cX_{\nu,3} (\Delta \tnu^*),
\end{multline*}
which implies that
\begin{equation}\label{eq: saos 18}
    \begin{split}
        \Delta \cX_{\nu,1} = G_4 \cX_{\nu,1}(\Delta \tnu) + G_5 &\cX_{\nu,2}(\Delta \tnu)\\& + G_6 \cX_{\nu,3}(\Delta \tnu) + \discrep^\sigma(\tnu)\cdot (\Delta \tnu^*)
    \end{split}
\end{equation}
with $G_4$, $G_5$, $G_6$ satisfying \eqref{eq: saos 15}, \eqref{eq: saos 16}. Next, we deduce from \eqref{eq: saos 1}, \eqref{eq: saos 3} that
\begin{align*}
    \Delta \zonusig &= \qnusig \cdot \big([(a \qnusig + \unusig) (\Delta \tnu^*) + \Delta W_\nu] - \unusig \Delta \tnu\big)\\
    & = a (\qnusig)^2 (\Delta \tnu^*) + \qnusig \Delta W_\nu + \qnusig \cdot \unusig\cdot (\Delta \tnu^* - \Delta \tnu),
\end{align*}
and similarly
\[
\Delta \zonutsig = a (\qnutsig)^2 (\Delta \tnu^*) + \qnutsig \Delta W_\nu + \qnutsig \cdot \unutsig \cdot (\Delta \tnu^* - \Delta \tnu).
\]
Subtracting, and recalling our assumptions \eqref{eq: saos 14}, we find that
\begin{equation}\label{eq: saos 19}
    \Delta \cX_{\nu,2} = G_7 \cX_{\nu,1} (\Delta \tnu) + \cX_{\nu,1} \Delta W_\nu + G_8 (\Delta \tnu)^2,
\end{equation}
with $G_7$, $G_8$ as in \eqref{eq: saos 15}, \eqref{eq: saos 16}. (Here, we use also the estimate 
\[
|\usig(\tnu)| \le C[|\qsig(\tnu)| + 1]
\]
as well as the corresponding estimate for $u^{\tsig}$, $q^{\tsig}$.)

Again recalling \eqref{eq: saos 14}, we see from \eqref{eq: saos 5}, \eqref{eq: saos 6} that
\begin{equation}\label{eq: saos 20}
    \Delta \cX_{\nu,3} = G_9 \cX_{\nu,1} (\Delta \tnu),
\end{equation}
with $G_9$ as in \eqref{eq: saos 15}, \eqref{eq: saos 16}.

Equations \eqref{eq: saos 18}, \eqref{eq: saos 19}, \eqref{eq: saos 20} tell us that
\begin{equation}\label{eq: saos 21}
    \Delta \cX_\nu = G \cX_\nu (\Delta \tnu) + H \cX_\nu (\Delta W_\nu) + F_\nu,
\end{equation}
where:
\begin{itemize}
    \item The entries of the matrix $G$ satisfy \eqref{eq: saos 15}, \eqref{eq: saos 16}.
    \item $H$ is the constant matrix
    \[
    \begin{pmatrix} 0 & 0 & 0 \\ 1 & 0 & 0 \\ 0 & 0 & 0 \end{pmatrix}.
    \]
    \item The vector $F_\nu$ satisfies
    \[
    |F_\nu| \le C | \discrep^\sigma(\tnu)|\cdot (\Delta \tnu) + C Q^{m_0} (\Delta \tnu)^2.
    \]
\end{itemize}
We now estimate
\begin{equation}\label{eq: saos 22}
    \begin{split}
        \E_a[ |\cX_{\nu+1}|^2 - |\cX_\nu|^2 | \cF_\nu ] =& 2 \E_a[ (\Delta \cX_\nu) \cdot \cX_\nu | \cF_\nu] + \E_a[ |\Delta \cX_\nu|^2 | \cF_\nu]\\
        \equiv & 2 \cdot \term 1 + \term 2.
    \end{split}
\end{equation}
Recall that $G$, $\cX_\nu$ are deterministic once we condition on $\cF_\nu$, while $\Delta W_\nu$ is independent of $\cF_\nu$, with mean 0 and variance $\le C (\Delta \tnu)$. Hence, from \eqref{eq: saos 21} we have the following estimates.
\begin{align*}
    \term 1 = & \E_a[\cX_\nu \cdot (\Delta \cX_\nu) | \cF_\nu]\\
    \le & C Q^{m_0} (\Delta \tnu) |\cX_\nu|^2 + C |\discrep^\sigma(\tnu)| \cdot |\cX_\nu| (\Delta \tnu) \\
    &+ C |\cX_\nu | Q^{m_0} (\Delta \tnu)^2.
\end{align*}
\begin{align*}
    \term 2 = & \E_a[ | \Delta \cX_\nu |^2 | \cF_\nu] \\
    \le & C |G \cX_\nu (\Delta \tnu)|^2 + C | H \cX_\nu |^2 \E_a[ | \Delta W_\nu |^2]\\& + C (\discrep^\sigma(\tnu))^2 (\Delta \tnu)^2 + C Q^{2m_0}(\Delta \tnu)^4.
\end{align*}
Our assumption \eqref{eq: saos 14} implies that $|\cX_\nu| \le CQ$, hence the above estimates imply that
\begin{equation}\label{eq: saos 23}
    \begin{split}
        \term 1 \le& C Q^{m_0} (\Delta \tnu) | \cX_\nu|^2 + C (\discrep^\sigma(\tnu))^2 (\Delta \tnu)\\& + C Q^{m_0+1}(\Delta \tnu)^2
    \end{split}
\end{equation}
and
\begin{equation}\label{eq: saos 24}
    \begin{split}
        \term 2 \le & C Q^{2m_0+2} (\Delta \tnu)^2 + C |\cX_\nu |^2 (\Delta \tnu)\\& + C (\discrep^\sigma(\tnu))^2 (\Delta \tnu)^2 + C Q^{2m_0}(\Delta \tnu)^4.
    \end{split}
\end{equation}

Putting \eqref{eq: saos 23} and \eqref{eq: saos 24} into \eqref{eq: saos 22}, we learn that
\begin{equation}\label{eq: saos 25}
    \begin{split}
        \E_a[ |\cX_{\nu+1}|^2 | \cF_\nu] \le & (1 + C Q^{m_0}(\Delta \tnu)) | \cX_\nu|^2 + C (\discrep^\sigma(\tnu))^2 \Delta \tnu \\&+ C Q^{2m_0+2} (\Delta \tnu)^2.
    \end{split}
\end{equation}

We have proven \eqref{eq: saos 25} under the assumption \eqref{eq: saos 14}. We now drop assumption \eqref{eq: saos 14}, and let $\cE_\nu$ denote the event
\[
\left[
\begin{array}{l}
    |\qsig_\mu|, |q^{\tsig}_\mu|, |\zeta_{1,\mu}^\sigma|, |\zeta_{1,\mu}^{\tsig}|, |\zeta_{2,\mu}^\sigma|, |\zeta_{2,\mu}^{\tsig}| \le Q\\
    \text{for all} \; \mu \le \nu.
\end{array}
\right]
\]
Note that $\mathbbm{1}_{\cE_{\nu+1}} \le \mathbbm{1}_{\cE_\nu}$, and that \eqref{eq: saos 14} holds whenever $\cE_\nu$ occurs. Moreover, $\cE_\nu$ is deterministic once we condition on $\cF_\nu$. Therefore, from \eqref{eq: saos 25} we deduce that
\begin{equation}\label{eq: saos 26}
\begin{split}
    \E_a[|\cX_{\nu+1}|^2 \mathbbm{1}_{\cE_{\nu+1}} | \cF_\nu] \le & \E_a[ | \cX_{\nu+1}|^2 \cdot \mathbbm{1}_{\cE_\nu} | \cF_\nu]\\
    \le & (1 + CQ^{m_0}(\Delta \tnu)) | \cX_\nu|^2 \mathbbm{1}_{\cE_\nu} \\&+ C (\discrep^\sigma(\tnu))^2 \Delta \tnu + C Q^{2m_0+2} (\Delta \tnu)^2.
\end{split}
\end{equation}
We now cease conditioning on $\cF_\nu$, and condition merely on $a_\tru = a$. From \eqref{eq: saos 26} we learn that
\begin{equation}\label{eq: saos 27}
    \begin{split}
\E_a[|\cX_{\nu+1}|^2 \mathbbm{1}_{\cE_{\nu+1}}] \le & (1+C Q^{m_0} (\Delta \tnu)) \E_a [ |\cX_\nu|^2 \mathbbm{1}_{\cE_\nu}] \\&+ C \E_a[(\discrep^\sigma(\tnu))^2] (\Delta \tnu) + C Q^{2m_0+2}(\Delta \tnu)^2.
    \end{split}
\end{equation}
Recall that $\cX_0 = 0$.

We impose the smallness assumption
\begin{equation}\label{eq: saos 28}
    (\Delta t_\mx)^{1/2} \le Q^{-(2m_0+2)}.
\end{equation}
Then \eqref{eq: saos 27} implies that
\begin{align*}
    \E_a[ |\cX_\nu|^2 \cdot \mathbbm{1}_{\cE_\nu}] \le C \exp(C Q^{m_0}\tnu)\cdot \Big\{ \E_a & \Big[ \sum_{0 \le \mu < \nu}(\discrep^\sigma(t_\mu))^2 \Delta t_\mu\Big]\\ &+ C (\Delta t_\mx)^{1/2}\Big\}
\end{align*}
for $0 \le \nu \le N$. Since $\mathbbm{1}_{\cE_N} \le \mathbbm{1}_{\cE_\nu}$, it follows that
\begin{equation}\label{eq: saos 29}
\begin{split}
    \E_a [ |\cX_\nu|^2 \mathbbm{1}_{\cE_N}] \le C\exp(CQ^{m_0}) \Big\{ \E_a & \Big[ \sum_{0 \le \mu < N} (\discrep^\sigma(t_\mu))^2 \Delta t_\mu\Big] \\&+ (\Delta t_\mx)^{1/2}\Big\}.
\end{split}
\end{equation}
In particular, since $\qnusig - \qnutsig$ is the first component of $\cX_\nu$, we have
\begin{equation}\label{eq: saos 30}
    \begin{split}
        \E_a[ | \qnusig - \qnutsig|^2 \cdot &\mathbbm{1}_{\cE_N}]\\ \le & C\exp(CQ^{m_0}) \Big\{ \E_a \Big[ \sum_{0 \le \mu < N} (\discrep^\sigma(t_\mu))^2 \Delta t_\mu\Big]\\ &\qquad \qquad \qquad \qquad+ (\Delta t_\mx)^{1/2}\Big\}.
    \end{split}
\end{equation}
Moreover, \eqref{eq: saos 17} and \eqref{eq: saos 29} together yield:
\begin{equation}\label{eq: saos 31}
    \begin{split}
        \E_a[ |\unusig - &\unutsig|^2 \cdot \mathbbm{1}_{\cE_N}]\\ \le &   C Q^{2m_0} \exp(CQ^{m_0}) \Big\{ \E_a \Big[ \sum_{0 \le \mu < N} (\discrep^\sigma(t_\mu))^2 \Delta t_\mu\Big] \\ &\qquad\qquad\qquad\qquad\qquad + (\Delta t_\mx)^{1/2}\Big\} \\&+ C \E_a[(\discrep^\sigma(\tnu))^2].
    \end{split}
\end{equation}
Summing \eqref{eq: saos 30} and \eqref{eq: saos 31} against $\Delta \tnu$, we find that
\begin{equation}\label{eq: saos 32}
    \begin{split}
        \E_a\Big[ \sum_{0 \le \nu < N} \big\{ | \qnusig & - \qnutsig|^2 + |\unusig - \unutsig|^2\big\} \Delta \tnu \cdot \mathbbm{1}_{\cE_N}\Big]\\
        \le& C Q^{2m_0} \exp(CQ^{m_0}) \Big\{\E_a\Big[ \sum_{0\le \mu < N} (\discrep^\sigma(t_\mu))^2 \Delta t_\mu\Big]\\ &\qquad\qquad\qquad\qquad\qquad+ (\Delta t_\mx)^{1/2}\Big\}.
    \end{split}
\end{equation}
We now turn to the case in which $\cE_N$ doesn't occur. Recall that $\cE_N$ fails precisely when, for some $\nu$, we have
\[
\max\{|\qnusig|, |\qnutsig|, |\zonusig|, |\zonutsig|, |\ztnusig|, |\ztnutsig|\} > Q.
\]
Thanks to Lemma \ref{lem: rare events}, applied to the strategies $\sigma$ and $\tsig$, we have
\begin{equation}\label{eq: saos 33}
    \prob_a[\cE_N\;\text{fails}] \le C \exp(-cQ)
\end{equation}
and also that
\[
\E_a [ (\max_\nu \{ |\qnusig| + |\qnutsig| + |\unusig| + |\unutsig|\} )^4] \le C,
\]
hence
\begin{equation}\label{eq: saos 34}
    \E_a \Big[ \Big( \sum_{0\le \nu < N} \big\{ |\qnusig - \qnutsig|^2 + |\unusig - \unutsig|^2\big\} \Delta \tnu\Big)^2\Big] \le C.
\end{equation}
From \eqref{eq: saos 33}, \eqref{eq: saos 34}, and Cauchy-Schwarz, we obtain the estimate
\begin{equation}\label{eq: saos 35}
    \E_a \Big[ \sum_{0 \le \nu < N} \big\{ |\qnusig - \qnutsig|^2 + |\unusig - \unutsig|^2\big\}\Delta \tnu \cdot \mathbbm{1}_{\cE_N \;\text{fails}}\Big]\le C \exp(-cQ).
\end{equation}

Finally, combining \eqref{eq: saos 32} and \eqref{eq: saos 35}, we find that
\begin{equation}\label{eq: saos 36}
\begin{split}
    \E_a \Big[ \sum_{0 \le \nu < N} \big\{ |\qnusig & - \qnutsig|^2 + |\unusig + \unutsig|^2\big\} \Delta \tnu \Big] \\ \le &C Q^{2m_0}\exp(CQ^{m_0})\Big\{ \E_a \Big[ \sum_{0 \le \mu < N} (\discrep^\sigma(t_\mu))^2 \Delta t_\mu\Big] \\&\qquad\qquad\qquad\qquad\quad+ (\Delta t_\mx)^{1/2}\Big\} + C \exp(-cQ).
\end{split}
\end{equation}
We have proven \eqref{eq: saos 36} under the assumption \eqref{eq: saos 28}.

Now let $\varepsilon>0$ be given.

We take $Q$ in \eqref{eq: saos 36} large enough so that $C\exp(-cQ) < \varepsilon/3$.

Having picked $Q$, we strengthen our smallness assumption \eqref{eq: saos 28} by demanding that
\[
CQ^{2m_0}\exp(CQ^{m_0})\cdot (\Delta t_\mx)^{1/2} \le \frac{\varepsilon}{3}.
\]
If also
\[
\E_a\Big[ \sum_{0 \le \mu < N}( \discrep^\sigma(t_\mu))^2 \Delta t_\mu\Big] \le \frac{\varepsilon}{3}[CQ^{2m_0}\exp(CQ^{m_0})]^{-1},
\]
then \eqref{eq: saos 36} implies that
\[
\E_a \Big[ \sum_{0 \le \nu < N} \big\{ |\qnusig - \qnutsig|^2 + |\unusig - \unutsig|^2\big\} \Delta \tnu\Big] < \varepsilon.
\]
Thus, we have proven the following result.
\begin{lem}\label{lem: stability}[Stability Lemma]
    Let $\varepsilon>0$, let $\tsig$ be the \textsc{allegedly optimal strategy} for the partition $0 = t_0 < t_1 < \dots < t_N = T$, and let $\sigma$ be another tame deterministic strategy for that same partition. Assume that $\sigma$ satisfies
    \[
    |\usig(\tnu)| \le C_\etame^\sigma [|\qsig(\tnu)|+1].
    \]
    Fix $a \in [-a_\emx, +a_\emx]$. Suppose that
    \[
    \Delta t_\emx = \max_\nu (t_{\nu+1}-\tnu) < \delta
    \]
    and
    \[
    \eE_a\Big[ \sum_{0 \le \nu < N}(\usig(\tnu) - u_\eop(\qsig(\tnu),\tnu,\zosig(\tnu),\ztsig(\tnu))^2 \Delta \tnu\Big] < \delta
    \]
    for a small enough $\delta >0$ determined by $\varepsilon$, $C_\etame^\sigma$, and the \textsc{boilerplate constants}. Then
    \[
    \eE_a \Big[\sum_{0 \le \nu < N} \big\{ |\usig(\tnu) - u^{\tsig}(\tnu)|^2 + |\qsig(\tnu) - q^{\tsig}(\tnu)|^2\big\} \Delta \tnu\Big] < \varepsilon.
    \]
\end{lem}

\section{The Second Bayesian Main Lemma}\label{sec: second bayesian}

\begin{lem}[The Second Bayesian Main Lemma]\label{lem: second bayesian main}
    Let $\varepsilon>0$, let $\tsig$ be the \textsc{allegedly optimal strategy} for a partition of $[0,T]$, and let $\sigma$ be another deterministic tame strategy for that same partition. Assume that $\sigma$ satisfies
    \[
    |\usig(\tnu)| \le C_\etame^\sigma [|\qsig(\tnu)|+1].
    \]
    Suppose that
    \[
    \Delta t_\emx < \delta
    \]
    and that
    \begin{equation}\label{eq: 0sharp}
        \begin{split}
            \eE\Big[ \sum_{0 \le \nu < N} \{ (\qsig(\tnu))^2  + (&\usig(\tnu))^2\} \Delta \tnu \Big]\\& \le \eE\Big[ \sum_{0 \le \nu < N}\{ (q^{\tsig}(\tnu))^2 + (u^{\tsig}(\tnu))^2\}\Delta \tnu\Big] + \delta,
        \end{split}
    \end{equation}
    for a small enough $\delta > 0$ determined by $\varepsilon$, $C_\etame^\sigma$, and the \textsc{boilerplate constants}.

    Then, for every $a \in [-a_\emx, + a_\emx]$, we have
    \[
    \eE_a \Big[ \sum_{0 \le \nu < N} \{ | \qsig(\tnu) - q^{\tsig}(\tnu)|^2 + |\usig(\tnu) - u^{\tsig}(\tnu)|^2\} \Delta \tnu \Big] < \varepsilon.
    \]
\end{lem}

\begin{proof}
    We first note that we can replace the \textsc{boilerplate constant} $C_\tame^\op$ by $\max\{C_\tame^\sigma, C_\tame^\op\}$. 
    
    Observe that
    \[
    \discrep^\sigma(\tnu) := \usig(\tnu) - u_\op(\qsig(\tnu), \tnu, \zosig(\tnu), \ztsig(\tnu))
    \]
    satisfies
    \[
    |\discrep^\sigma(\tnu)| \le |\usig(\tnu)| + C [|\qsig(\tnu)| + 1],
    \]
    hence
    \[
    \sum_{0 \le \nu < N} (\discrep^\sigma(\tnu))^2 \Delta \tnu \le C' \max_{0\le\nu < N} \{ |\qsig(\tnu)|^2 + |\usig(\tnu)|^2 + 1\}.
    \]
    Lemma \ref{lem: rare events} therefore yields the estimate
    \begin{equation}\label{eq: 1sharp}
        \E_a\Big[ \Big\{ \sum_{0 \le \nu < N} (\discrep^\sigma(\tnu))^2 \Delta \tnu\Big\}^p\Big] \le C_p
    \end{equation}
    for any $p \ge 1$ and any $a \in [-a_\mx, a_\mx]$.

    Now let $\varepsilon>0$ be given. We pick a small enough $\varepsilon_1 > 0$ depending on $\varepsilon$, then we pick a large enough $Q>1$ depending on $\varepsilon_1$, next we pick a small enough $\varepsilon_2>0$ depending on $Q$, and finally we pick a small enough $\delta >0$ depending on $\varepsilon_2$. We then argue as follows.

    Suppose that $\Delta t_\mx < \delta$, and suppose \eqref{eq: 0sharp} holds. Since $\delta > 0$ has been picked small enough, depending on $\varepsilon_2$, Corollary \ref{cor: first bayesian main} tells us that
    \begin{equation}\label{eq: 2sharp}
        \E\Big[ \sum_{0 \le \nu < N} (\discrep^\sigma(\tnu))^2\Delta \tnu\Big] < 3 \varepsilon_2.
    \end{equation}
    For any random variable $X$, the expected value $\E[X]$ is an average of $\E_a[X]$ over $a \in [-a_\mx,+a_\mx]$ with respect to our given prior probability distribution for $a_\tru$. Therefore, \eqref{eq: 2sharp} implies that for some $a_1 \in [-a_\mx, a_\mx]$ we have
    \[
    \E_{a_1} \Big[ \sum_{0 \le \nu < N} (\discrep^\sigma(\tnu))^2 \Delta \tnu\Big] < 4\varepsilon_2.
    \]
    Consequently, for any $ a\in [-a_\mx, +a_\mx]$, Lemma \ref{lem: change of assumption} and estimate \eqref{eq: 1sharp} with $p=2$ give
    \begin{equation}\label{eq: 3sharp}
        \E_a\Big[ \sum_{0\le\nu < N} (\discrep^\sigma(\tnu))^2 \Delta \tnu\Big] \le \exp(CQ^2)\cdot 4\varepsilon_2 + C \exp(-cQ^2).
    \end{equation}
    Since $\varepsilon_2$ has been picked small enough depending on $Q$, while $Q$ has been picked large enough depending on $\varepsilon_1$, the right-hand side of \eqref{eq: 3sharp} is less than $C'\exp(-cQ^2) < \varepsilon_1$. Thus, for any $a \in [-a_\mx, a_\mx]$, we have
    \begin{equation}\label{eq: 4sharp}
        \E_a \Big[ \sum_{ 0 \le \nu < N} (\discrep^\sigma(\tnu))^2 \Delta \tnu\Big] < \varepsilon_1.
    \end{equation}
    Since $\varepsilon_1$ has been picked small enough depending on $\varepsilon$, estimate \eqref{eq: 4sharp} and Lemma \ref{lem: stability} imply that
    \[
    \E_a \Big[ \sum_{0 \le \nu < N} \{ |\usig(\tnu) - u^{\tsig}(\tnu)|^2 + |\qsig(\tnu) - q^{\tsig}(\tnu)|^2 \} \Delta \tnu\Big] < \varepsilon 
    \]
    for all $a \in [-a_\mx, a_\mx]$, completing the proof of the lemma.
\end{proof}

\begin{cor}\label{cor: second bayesian main}
    Under the assumptions of Lemma \ref{lem: second bayesian main} we have
    \begin{multline*}
        \bigg|\eE_a\Big[ \sum_{0 \le \nu < N} \{ (\qsig(\tnu))^2 + (\usig(\tnu))^2\} \Delta \tnu\Big]\\ - \eE_a \Big[ \sum_{0 \le \nu < N} \{ (q^{\tsig}(\tnu))^2 + (u^{\tsig}(\tnu))^2\}\Delta \tnu\Big]\bigg| \le C \varepsilon^{1/2}
    \end{multline*}
    for each $a \in [-a_\emx, a_\emx]$.
\end{cor}
\begin{proof}
    The Corollary follows from Lemma \ref{lem: second bayesian main}, together with Minkowski's inequality and the estimate
    \[
    \E_a \Big[ \sum_{0 \le \nu <N} \{(q^{\tsig}(\tnu))^2 +(u^{\tsig}(\tnu))^2\} \Delta \tnu\Big] \le C,
    \]
    which in turn follows from Lemma \ref{lem: rare events}.
\end{proof}

\section{Allowing for Dependence on Coin Flips}
Let $\tsig$ be the \textsc{allegedly optimal strategy} associated to a partition of $[0,T]$, and let $\sigma$ be a tame strategy associated to the same partition. In this section we allow $\sigma$ to depend on the coin flips $\vxi$. For fixed $\veta \in \{0,1\}^\N$, we write $\sigveta$ for the strategy prescribed by $\sigma$ in case $\vxi = \veta$. We write $\prob_B[\dots]$ and $\E_B[\dots]$ (``B'' for ``Bernoulli'') to denote probability and expectation with respect to the natural (product) probability measure on $\{0,1\}^\N$, in which each $\xi_\nu$ is equal to 0 with probability $1/2.$

Our goal here is to extend Lemma \ref{lem: second bayesian main} to the case of the $\vxi$-dependent strategy $\sigma$.

To so we denote
\[
\cost_D(\sigma) = \sum_{0\le\nu < N} \{(\qsig(\tnu))^2 + (\usig(\tnu))^2\} \Delta \tnu,
\]
and similarly for $\cost_D(\tsig)$ and $\cost_D(\sigveta)$.

Because $\sigma$ is tame, we have
\begin{equation}\label{eq: adc 1}
    |\usig(\tnu)| \le C_\tame^\sigma \cdot [|\qsig(\tnu)| + 1]
\end{equation}
for a constant $C_\tame^\sigma$.

Let $\varepsilon>0$ be given, and let $\delta >0$ be less than a small enough positive number determined by $\varepsilon$, $C_\tame^\sigma$, and the \textsc{boilerplate constants}. 

Suppose that $\Delta t_\mx < \delta$, and that
\begin{equation}\label{eq: adc 2}
    \E[\cost_D(\sigma)] \le \E[\cost_D(\tsig)] + \delta.
\end{equation}

We want to show that
\[
\E_a \Big[ \sum_{0 \le \nu <N} \{ | \qsig(\tnu) - q^{\tsig}(\tnu)|^2 + |\usig(\tnu) - u^{\tsig}(\tnu)|^2 \} \Delta \tnu \Big] < \varepsilon
\]
for all $a \in [-a_\mx, a_\mx]$. To see this, we first pick $\hat{\delta}$ small enough, determined by $\varepsilon$, $C_\tame^\sigma$, and the \textsc{boilerplate constants}; then we pick $\delta$ small enough, determined by $\hat{\delta}$, $C_\tame^\sigma$, and the \textsc{boilerplate constants}.

Lemma \ref{lem: first bayesian main}, with $\hat{\delta}^2$ in place of $\varepsilon$, shows that
\begin{equation}\label{eq: adc 3}
    \E[\cost_D(\sigveta)] \ge \E[\cost_D(\tsig)] - \hat{\delta}^2\;\text{for all}\; \veta \in \{0,1\}^\N.
\end{equation}
On the other hand, \eqref{eq: adc 2} shows that
\begin{equation}\label{eq: adc 4}
    \E_B [ \{ \E[\cost_D(\sigveta)] - \E[\cost_D(\tsig)] + \hat{\delta}^2\}]\le \delta + \hat{\delta}^2 < 2\hat{\delta}^2,
\end{equation}
since $\delta$ is less than a small enough constant depending on $\hat{\delta}$. The quantity in curly brackets in \eqref{eq: adc 4} is nonnegative, thanks to \eqref{eq: adc 3}. Therefore, if we set
\begin{equation}\label{eq: adc 5}
    \goodflips = \{ \veta \in \{0,1\}^\N : \E[\cost_D(\sigveta)] \le \E[\cost_D(\tsig)] + \hat{\delta}\}
\end{equation}
and
\begin{equation}\label{eq: adc 6}
    \badflips = \{ \veta \in \{0,1\}^\N : \E[\cost_D(\sigveta)] > \E[\cost_D(\tsig)] + \hat{\delta}\},
\end{equation}
then
\begin{equation}\label{eq: adc 7}
    \prob_B[\badflips] < 10 \hat{\delta}.
\end{equation}
Moreover, since $\sigveta$ and $\tsig$ are tame for each $\veta$, Lemma \ref{lem: rare events} implies that for each $a \in [-a_\mx, a_\mx]$ we have
\[
\E_a\Big[ \Big\{ \sum_{0 \le \nu < N}[ |q^{\sigveta}(\tnu) - q^{\tsig}(\tnu)|^2 + |u^{\sigveta}(\tnu) - u^{\tsig}(\tnu)|^2]\Delta \tnu \Big\}^2\Big]\le C,
\]
where we allow constants $C$ to depend on $C_\tame^\sigma$ in \eqref{eq: adc 1}. Consequently,
\[
\E_a\Big[ \Big\{ \sum_{0 \le \nu < N} [ |\qsig(\tnu) - q^{\tsig}(\tnu)|^2 + |\usig(\tnu) - u^{\tsig}(\tnu)|^2] \Delta \tnu \Big\}^2\Big] \le C.
\]
Together with \eqref{eq: adc 7} and Cauchy-Schwarz, this implies that
\begin{equation}\label{eq: adc 8}
\begin{split}
    \E_a\Big[ \Big\{ \sum_{0 \le \nu < N} [ |\qsig(\tnu) - q^{\tsig}(\tnu)|^2 + |\usig(\tnu) - u^{\tsig}(\tnu)|^2]\Delta \tnu\Big\}\cdot \mathbbm{1}_{\badflips}(\vxi)\Big]\\
    \le C \hat{\delta}^{1/2}.
    \end{split}
\end{equation}
On the other hand, if $\veta \in \goodflips$, then Lemma \ref{lem: second bayesian main} tells us that
\[
\E_a\Big[ \sum_{0\le \nu < N} \{ | q^{\sigveta}(\tnu) - q^{\tsig}(\tnu)|^2 + |u^{\sigveta}(\tnu) - u^{\tsig}(\tnu)|^2\}\Delta \tnu\Big] < \frac{\varepsilon}{2}.
\]
Consequently,
\begin{equation}\label{eq: adc 9}
\begin{split}
    \E_a\Big[ \Big\{ \sum_{0\le \nu <N} [ |\qsig(\tnu) - q^{\tsig}(\tnu)|^2 + |\usig(\tnu) - u^{\tsig}(\tnu&)|^2] \Delta \tnu\Big\}\\ &\cdot \mathbbm{1}_{\goodflips}(\vxi)\Big] \le \frac{\varepsilon}{2}.
\end{split}
\end{equation}
We learn from \eqref{eq: adc 8} and \eqref{eq: adc 9} that
\begin{equation}\label{eq: adc 10}
    \E_a\Big[ \sum_{0\le \nu < N}\{|\qsig(\tnu) - q^{\tsig}(\tnu)|^2 + |\usig(\tnu) - u^{\tsig}(\tnu)|^2\}\Delta \tnu\Big] < \varepsilon.
\end{equation}
Thus, we have shown that \eqref{eq: adc 1}, \eqref{eq: adc 2} and $\Delta t_\mx < \delta$ imply \eqref{eq: adc 10}.

That is, Lemma \ref{lem: second bayesian main} holds without the assumption that $\sigma$ is deterministic. From now on, when we apply that lemma, we needn't check that $\sigma$ is deterministic.

\section{Reformulating the Main Bayesian Results}
Let $\sigma$ be a tame strategy.

Recall, from Section \ref{sec: continuous vs discrete}, that $\qsig_C(t) = \qsig(t)$ and $\qsig_D(t) = \qsig(\tnu)$ for $ t\in [\tnu, t_{\nu+1})$, $0 \le \nu < N$. Also, $\usig(t) = \usig(\tnu)$ for $ t\in [\tnu, t_{\nu+1})$, $0 \le \nu < N$. Therefore,
\begin{multline*}
    \sum_{0\le \nu < N} \{ |\usig(\tnu) - u^{\tsig}(\tnu)|^2 + |\qsig(\tnu) - q^{\tsig}(\tnu)|^2\} \Delta \tnu \\= \int_0^T\{ |\usig(t) - u^{\tsig}(t)|^2 + |\qsig_D(t) - q^{\tsig}_D(t)|^2\} \ dt
\end{multline*}
where $\tsig$ is the \textsc{allegedly optimal strategy} for some Bayesian prior. Similarly,
\[
    \sum_{0 \le \nu < N} \{  |\usig(\tnu)|^2 + |\qsig(\tnu)|^2\} \Delta \tnu = \int_0^T \{ |\usig(t)|^2 + |\qsig_D(t)|^2\} \ dt
\]
and the analogous formula holds for $\tsig$.

From Section \ref{sec: continuous vs discrete}, we have
\[
\E_a\big[ \max_{t \in [0,T]} | \qsig(t) - \qsig_D(t)|^m\big] \le C_m (\Delta t_\mx)^{m/4}\;\text{for any}\; m \ge 1
\]
and any $a \in [-a_\mx, a_\mx]$; the analogous estimate holds for $\tsig$ in place of $\sigma$.

In view of the above remarks, we can reformulate our main previous results, replacing $\qsig_D(t)$ by $\qsig(t)$.

We define
\[
\cost(\sigma) = \int_0^T \{(\usig(t))^2 +  (\qsig(t))^2\}\ dt.
\]

We combine the above discussion with Lemmas \ref{lem: ac}, \ref{lem: undercontrol}, \ref{lem: first bayesian main}, and \ref{lem: second bayesian main} to deduce the following.

\begin{thm}\label{thm: bayesian strategies}[Main Theorem on Bayesian Strategies]
    Let $\sigma$ be a tame strategy, satisfying
    \[
    |\usig(\tnu)| < C_\etame^\sigma [|\qsig(\tnu)| + 1].
    \]
    Fix a Bayesian prior $d\eprob(a)$ on $[-a_\emx,+a_\emx]$, and suppose our PDE Assumption holds for the PDE arising from that prior. Let $\tsig$ denote the \textsc{allegedly optimal strategy} for the same partition of $[0,T]$ used to define $\sigma$.

    Then given $\varepsilon>0$, there exists $\delta >0$ determined by $\varepsilon$, together with the \textsc{boilerplate constants} and the constant $C_\etame^\sigma$, such that the following holds. Suppose $\Delta t_\emx < \delta $. Then:
    \begin{enumerate}[label={\emph{(\arabic*)}}, ref={(\arabic*)}]
        \item\[
        |\eE[\cost(\tsig)] - S(q_0,0,0,0)| < \varepsilon,
        \]
        where $S$ is our PDE solution.\label{eq: rmb 1}
        \item If $\eE[\cost(\sigma)] < \eE[\cost(\tsig)] + \delta$, then for any $a \in [-a_\emx,+a_\emx]$ we have
        \[
        \eE_a \Big[ \int_0^T \{ |\usig(t) - u^{\tsig}(t)|^2 + |\qsig(t) - q^{\tsig}(t)|^2\}\ dt\Big] < \varepsilon.
        \]\label{eq: rmb 2}
        \item There exists an analytic function $I(a)$ defined on $\cR \equiv (-a_\mx, + a_\mx) \times (-\hat{c},\hat{c})$ for some $\hat{c}$, such that $|I(a)| \le C$ on $\cR$, and
        \[
        |I(a) - \eE_a[\cost(\sigma)]| < \varepsilon\;\text{for}\; a \in (-a_\emx,+a_\emx).
        \]\label{eq: rmb 3}
        \item If $a$ exceeds a large enough constant $C$, then
        \[
        \eE_a[\cost(\sigma)] > cT^2\exp(caT).
        \]\label{eq: rmb 4}
    \end{enumerate}
\end{thm}
Note that strictly speaking, we proved \ref{eq: rmb 3} and \ref{eq: rmb 4} for deterministic strategies, so they hold for $\sigma$ once we condition on $\vxi = \veta$ for a fixed $\veta \in \{0,1\}^\N$. Integrating over $\veta$, we obtain \ref{eq: rmb 3} and \ref{eq: rmb 4} as stated.

\section{Comparing the Allegedly Optimal Strategies for Two Partitions}
Let $\pi$ be a partition $0 = t_0 < t_1 < \dots < t_N = T$, and let $\pi'$ be a refinement of $\pi$, given by $0 = t_0' < t_1' < \dots < t_N' = T$. Let $\tsig, \tsig'$ be the corresponding \textsc{allegedly optimal strategies}. We set $\Delta t_\mx = \max_\nu (t_{\nu+1} - \tnu)$ and $\Delta t_\mx' = \max_\nu (t_{\nu+1}' - \tnu') \le \Delta t_\mx.$ We will prove the following result.

\begin{lem}\label{lem: comparing}
    Given $\varepsilon>0$, there exists $\delta > 0$ such that if $\Delta t_\emx < \delta$, then
    \[
    \eE_a\Big[ \int_0^T \{ |q^{\tsig}(t) - q^{\tsig'}(t)|^2 + |u^{\tsig}(t) - u^{\tsig'}(t)|^2\}\ dt \Big] < \varepsilon
    \]
    for all $ a \in [-a_\emx, + a_\emx]$.
\end{lem}

\begin{proof}
    Given $\varepsilon>0$ we pick $\hat{\delta}>0$ small enough, then pick $\delta >0$ small enough, depending on $\hat{\delta}$.

    Suppose $\Delta t_\mx < \delta$; then also $\Delta t_\mx' < \delta$. Theorem \ref{thm: bayesian strategies} gives
    \begin{equation}\label{eq: ctp 1}
\Big| \E\Big[ \int_0^T\{ |q^{\tsig}(t)|^2 + |u^{\tsig}(t)|^2\}\ dt \Big] - S(q_0,0,0,0)\Big| < \hat{\delta}
    \end{equation}
    and
    \begin{equation}\label{eq: ctp 2}
        \Big| \E\Big[ \int_0^T \{ |q^{\tsig'}(t)|^2 + |u^{\tsig'}(t)|^2\}\ dt \Big] - S(q_0,0,0,0)\Big| < \hat{\delta}.
    \end{equation}
    In particular,
    \begin{equation}\label{eq: ctp 3}
        \E\Big[ \int_0^T \{ |q^{\tsig}(t)|^2 + |u^{\tsig}(t)|^2\}\ dt\Big] < C
    \end{equation}
    and
    \begin{equation}\label{eq: ctp 4}
        \E\Big[ \int_0^T \{ |q^{\tsig'}(t)|^2 + |u^{\tsig'}(t)|^2\} \ dt \Big] < C.
    \end{equation}
    Lemma \ref{lem: refinement} gives a tame strategy $\hatsigma$ associated to the partition $\pi'$ for which
    \begin{equation}\label{eq: ctp 5}
        \E_a\Big[ \int_0^T \{ |q^{\tsig}(t) - q^{\hatsigma}(t)|^2 + |u^{\tsig}(t) - u^{\hatsigma}(t)|^2\}\ dt \Big] \le \hat{\delta}^2
    \end{equation}
    for every $ a\in [-a_\mx, +a_\mx]$, hence
    \[
    \E\Big[ \int_0^T \{ |q^{\tsig}(t) - q^{\hatsigma}(t)|^2 +|u^{\tsig}(t) - u^{\hatsigma}(t)|^2\}\ dt\Big] \le \hat{\delta}^2.
    \]
    Together with \eqref{eq: ctp 3}, this implies that
    \[
    \E\Big[ \int_0^T\{ |q^{\hatsigma}(t)|^2 + |u^{\hatsigma}(t)|^2\} \ dt \Big] \le \E\Big[ \int_0^T \{ |q^{\tsig}(t)|^2 + |u^{\tsig}(t)|^2\}\ dt\Big] + C \hat{\delta}.
    \]
    Thanks to \eqref{eq: ctp 1} and \eqref{eq: ctp 2}, this in turn implies that 
    \begin{equation}\label{eq: ctp 6}
        \E\Big[ \int_0^T \{ |q^{\hatsigma}(t)|^2 + |u^{\hatsigma}(t)|^2\}\ dt\Big] \le \E\Big[ \int_0^T \{ |q^{\tsig'}(t)|^2 + |u^{\tsig'}(t)|^2\} \ dt\Big] + C \hat{\delta}.
    \end{equation}
    Recall that $\tsig'$ is the \textsc{allegedly optimal strategy} associated to the partition $\pi'$, while $\hatsigma$ is another tame strategy associated to $\pi'$.

    Consequently, by virtue of Theorem \ref{thm: bayesian strategies}, \eqref{eq: ctp 6} implies that
    \begin{equation}\label{eq: ctp 7}
        \E_a \Big[ \int_0^T \{ |q^{\hatsigma}(t) - q^{\hatsigma'}(t)|^2 + |u^{\hatsigma}(t) - u^{\hatsigma'}(t)|^2\}\ dt\Big] < \frac{\varepsilon}{2}
    \end{equation}
    for every $a \in [-a_\mx, +a_\mx]$.

    From \eqref{eq: ctp 5} and \eqref{eq: ctp 7} we have
    \[
    \E_a \Big[ \int_0^T \{ |q^{\tsig}(t) - q^{\tsig'}(t)|^2 + |u^{\tsig}(t) - u^{\tsig'}(t)|^2\} \ dt\Big] < \varepsilon,
    \]
    completing the proof of the lemma.
\end{proof}

\begin{cor}\label{cor: comparing}
    Given $\varepsilon>0$ there exists $\delta >0$ for which the following holds.

    Let
    \[
    \pi:\; 0 = t_0 = t_0 < t_1 < \dots < t_N = T
    \]
    and
    \[
    \pi':\; 0 = t_0' < t_1' < \dots < t_{N'}' = T
    \]
    be partitions, let
    \[
    \Delta t_\emx = \max_\nu (t_{\nu+1}-\tnu)\quad \text{and}\quad \Delta t_\emx' = \max_\nu (t_{\nu+1}' - \tnu'),
    \]
    and let $\tsig$, $\tsig'$ be the \textsc{allegedly optimal strategies} associated to $\pi$, $\pi'$, respectively.

    If $\Delta t_\emx,\Delta t_\emx'<\delta$, then
    \[
    \eE_a \Big[ \int_0^T \{ |q^{\tsig}(t) - q^{\tsig'}(t)|^2 + |u^{\tsig}(t) - u^{\tsig'}(t)|^2\} \ dt\Big] < \varepsilon
    \]
    for every $a \in [-a_\emx, +a_\emx]$.
\end{cor}

\begin{proof}
    Compare both $\tsig$ and $\tsig'$ to the \textsc{allegedly optimal strategy} arising from a common refinement of $\pi$ and $\pi'$.
\end{proof}

\chapter{Decisions in Continuous Time}\label{chap: 4}

\section{Tame Strategies with Decisions in Continuous Time}\label{sec: ctns time}

Suppose that for each $n=1,2,3,\dots$ we are given a tame strategy $\sigma^n$ associated to a partition $\pi^n$ of the time interval $[0,T]$. Say $\pi^n$ is given by
\begin{equation}\label{eq: tsdct 0}
    0 = t_0^n < t_1^n < \dots < t_{N(n)}^n = T,
\end{equation}
and $\sigma^n$ is given by the collection of tame rules
\[
\sigma^n = (\sigma^n_{\tnu^n})_{\nu=0,1,\dots,N(n)-1},
\]
where each $\sigma_{\tnu^n}^n$ is a function of $\nu$ real variables $\barq_1, \dots, \barq_\nu$ and the coin flips $\vxi$.

We assume that
\begin{equation}\label{eq: tsdct 1}
    \Delta t_\mx^n := \max_\nu (t_{\nu+1}^n - \tnu^n)\rightarrow 0\;\text{as}\; n\rightarrow \infty,
\end{equation}
and that
\begin{equation}\label{eq: tsdct 2}
    |\sigma_{\tnu^n}^n (\barq_1,\dots,\barq_\nu,\vxi)| \le C_\tame^{\vsigma} [|\barq_\nu| + 1]\;\text{for all}\; \barq_1,\dots,\barq_\nu
\end{equation}
for all $n,\nu$, with $C_\tame^{\vsigma}$ independent of $n,\nu,\barq_1,\dots, \barq_\nu$.

Each $\sigma^n$ gives rise to control trajectories and particle trajectories $u^{\sigma^n,a}(t)$ and $q^{\sigma^n,a}(t)$, respectively, for $t \in [0,T]$, once we specify that $a_\tru = a$. Here, $a$ is an arbitrary real number, not necessarily belonging to $[-a_\mx, + a_\mx]$. We define the \emph{expected cost} of $\sigma^n$ by
\[
\excost(\sigma^n, a) = \E_a \Big[ \int_0^T \{ (u^{\sigma^n,a}(t))^2 + (q^{\sigma^n,a}(t))^2\}\ dt \Big].
\]
We say that $(\sigma^n)_{n\ge 1}$ is a \emph{Cauchy sequence of uniformly tame strategies} if \eqref{eq: tsdct 1} and \eqref{eq: tsdct 2} hold, and
\begin{equation}\label{eq: tsdct 3}
    \lim_{n,m\rightarrow \infty} \E_a \Big[ \int_0^T \{ |u^{\sigma^n,a}(t) - u^{\sigma^m,a}(t)|^2 + |q^{\sigma^n,a}(t) - q^{\sigma^m,a}(t)|^2\} \ dt\Big] = 0,
\end{equation}
uniformly for $a$ in any bounded subset of $\R$.

If $(\sigma^n)_{n\ge 1}$ and $(\hatsigma^n)_{n\ge 1}$ are two Cauchy sequences of uniformly tame strategies, then we call those sequences \emph{equivalent} if we have
\begin{equation}\label{eq: tsdct 4}
\lim_{n\rightarrow \infty} \E_a \Big[ \int_0^T \{ |u^{\sigma^n,a}(t) - u^{\hatsigma^n,a}(t)|^2 + |q^{\sigma^n,a}(t) - q^{\hatsigma^n,a}(t)|^2\}\ dt\Big]=0
\end{equation}
for each $a\in \R$.

If $(\sigma^n)_{n\ge 1}$ is a Cauchy sequence of uniformly tame strategies, then for each $a \in \R$ there exist random functions $u^a(t)$, $q^a(t)$ s.t.
\begin{equation}\label{eq: :)}
    \lim_{n\rightarrow \infty} \E_a \Big[ \int_0^T \{ |u^{\sigma^n,a}(t) - u^a(t)|^2 + |q^{\sigma^n,a}(t) - q^a(t)|^2\}\ dt \Big] = 0,
\end{equation}
uniformly for $a$ in any bounded subset of $\R$.

Moreover, if two Cauchy sequences $(\sigma^n)_{n \ge 1}$, $(\hatsigma^n)_{n \ge 1}$ are equivalent, then for each $a\in\R$, the $u^a$, $q^a$ defined by those sequences are equal, for a.e.\ $t \in [0,T]$, almost surely with respect to $\prob_a$.

We define a \emph{tame strategy} (for decisions in continuous time) to be an equivalence class of Cauchy sequences of uniformly tame strategies, with respect to the equivalence relation \eqref{eq: tsdct 4}. We denote a tame strategy by $\vsigma = [[(\sigma^n)_{n\ge 1}]]$, and we say that $C_\tame^{\vsigma}$ in \eqref{eq: tsdct 2} is a \emph{tame constant} for $\vsigma$. If $\vsigma = [[(\sigma^n)_{n \ge 1}]]$ is a tame strategy, then we write $q^{\vsigma, a}(t)$ and $u^{\vsigma,a}(t)$ to denote the functions $q^a(t)$, $u^a(t)$ in \eqref{eq: :)}.

If $\vsigma=[[(\sigma^n)_{n \ge 1}]]$ is a tame strategy, then we define
\begin{equation*}
    \excost(\vsigma, a) = \lim_{n \rightarrow \infty} \excost(\sigma^n,a).
\end{equation*}
This quantity is well-defined, since the limit exists and two equivalent Cauchy sequences produce the same expected cost. Immediately from our results on tame strategies associated to partitions of $[0,T]$, we have the following results, for any tame strategy $\vsigma = [[(\sigma^n)_{n \ge 1}]]$.

Note that
\[
\excost(\vsigma, a) = \E_a\Big[ \int_0^T \{ (q^{\vsigma,a}(t))^2 +(u^{\vsigma,a}(t))^2\} \ dt\Big]
\]
for any tame strategy $\vsigma$ and any $a \in \R$. For large enough $a>0$, we have
\begin{equation}\label{eq: tsdct 6}
    \excost(\vsigma,a) \ge cT^2 \exp(caT).
\end{equation}
Moreover, we will see that the function 
\begin{equation}\label{eq: tsdct 7}
\begin{split}
    \parbox{20em}{$[-a_\mx, + a_\mx] \ni a \mapsto \excost(\vsigma, a)$\\
    continues to a bounded analytic function on\\
    $(-a_\mx, +a_\mx) + i (-\hat{c}, \hat{c})$
    for some $\hat{c}>0$ determined by the \textsc{boilerplate constants} and the constant $C_\tame^{\vsigma}$.}
\end{split}
\end{equation}

Moreover, we may replace $a_\mx$ by any $\hat{a}_\mx > a_\mx$, and the assumptions of the preceding sections are still valid; the constants determined by the \textsc{boilerplate constants} will now depend on $\hat{a}_\mx$. In particular, \eqref{eq: tsdct 7} immediately implies that the function
\[
[-\hat{a}_\mx, +\hat{a}_\mx] \ni a\mapsto \excost(\vsigma,a)
\]
continues to a bounded analytic function on 
\[
(-\hat{a}_\mx, +\hat{a}_\mx) + i (-\hat{c}(\hat{a}_\mx), + \hat{c}(\hat{a}_\mx));
\]
the bound for that analytic function depends on $\hat{a}_\mx$. This implies that the function
\[
\R \ni a \mapsto \excost(\vsigma,a)
\]
continues to an analytic function on a neighborhood of the real axis in $\C$. In other words,
\begin{equation}\label{eq: tsdct 8}
    \excost(\vsigma,a)\;\text{is a real-analytic function of} \; a \in \R.
\end{equation}

Let us check our assertion \eqref{eq: tsdct 7}. Recall our previous result on analytic continuation:

Given $\varepsilon >0$, there exists $\delta >0$ s.t. whenever $\Delta t_\mx^n < \delta $ we have
\begin{equation}\label{eq: tsdct 9}
    \excost(\sigma^n,a) = I^n(a) + \error^n(a)
\end{equation}
on $(-a_\mx, + a_\mx)$, where
\begin{equation}\label{eq: tsdct 10}
    |\error^n(a)| < \varepsilon\;\text{for}\; a \in (-a_\mx,a_\mx),
\end{equation}
and
\begin{equation}\label{eq: tsdct 11}
    I^n(a)\;\text{is analytic on }\cR\equiv (-a_\mx, +a_\mx)+i(-\hat{c},\hat{c}),
\end{equation}
with $|I^n(a)| \le C$ everywhere on that rectangle. In particular, \eqref{eq: tsdct 9}, \eqref{eq: tsdct 10}, \eqref{eq: tsdct 11} hold for all large enough $n$, since $\Delta t_\mx^n \rightarrow 0$ as $n \rightarrow \infty$.

Since the $I^n(a)$ are uniformly bounded analytic functions on the rectangle $\cR$, we may pick out a subsequence $I^{n_j}(a)$ that converges to a bounded analytic function $I^\infty(a)$ uniformly on compact subsets of $\cR$. Applying \eqref{eq: tsdct 9} and \eqref{eq: tsdct 10} to $\sigma^{n_j}$, and passing to the limit as $j\rightarrow \infty$, we find that $\excost(\vsigma,a) = I^\infty(a)$ for all $a \in (-a_\mx, a_\mx)$, completing the proof of \eqref{eq: tsdct 7}.

Next, we construct an \textsc{allegedly optimal strategy} with \emph{decisions in continuous time}.

Fix a prior probability distribution $d\prob(a)$ on $[-a_\mx, + a_\mx]$. Given $\vsigma = [[(\sigma_n)_{n \ge 1}]]$ we define
\begin{align*}
&\excost(\sigma^n) = \int_{-a_\mx}^{a_\mx} \excost(\sigma^n,a)\ d\prob(a),\\
&\excost(\vsigma) = \int_{-a_\mx}^{a_\mx} \excost(\vsigma, a)\ d\prob(a).
\end{align*}
Let $\pi^n$ be a sequence of partitions of $[0,T]$ (as in \eqref{eq: tsdct 0}), with $\Delta t_\mx^n\rightarrow 0$ as $n \rightarrow \infty$.

For each $n$, let $\tsig^n$ denote the allegedly optimal strategy associated to the partition $\pi^n$. Corollary \ref{cor: comparing} tells us that $(\tsig^n)_{n \ge 1}$ satisfies condition \eqref{eq: tsdct 3}. Moreover, we have assumed condition \eqref{eq: tsdct 1}, and our PDE Assumption (see \eqref{eq: pde 6}) tells us that \eqref{eq: tsdct 2} holds. Thus, the $(\tsig^n)_{n \ge 1}$ form a Cauchy sequence of uniformly tame strategies.

We write $\vsigma_\op = [[(\tsig^n)_{n \ge 1}]]$ to denote the resulting continuous tame strategy. Note that $\vsigma_\op$ is independent of the sequence of partitions used to define it. We will show that it is optimal for Bayesian control.

Let $\vsigma = [[(\sigma^n)_{n\ge 1}]]$ be a tame strategy with tame constant $C_\tame^{\vsigma}$. Then 
\begin{align*}
    \excost(\vsigma) =& \int_{-a_\mx}^{a_\mx} \excost(\vsigma,a)\ d\prob(a)\\
    =& \lim_{n \rightarrow \infty} \int_{-a_\mx}^{a_\mx} \excost(\sigma^n,a)\ d\prob(a)\\
    = & \lim_{n \rightarrow\infty} \excost(\sigma^n).
\end{align*}
Here, the interchange of limit and integral is justified by the uniform convergence in $a$ that we assumed in our definition of Cauchy sequences.

Let $\varepsilon>0$ be given. For $n$ large enough, condition \eqref{eq: tsdct 1} for the $\sigma^n$ allows us to apply Theorem \ref{thm: bayesian strategies}; we conclude that
\[
\excost(\sigma^n) \ge \excost(\tsig^n) - \varepsilon
\]
for $n$ large enough. Here, $\tsig^n$ denotes the \textsc{allegedly optimal strategy} associated to the partition relevant to $\sigma^n$.

Passing to the limit as $n \rightarrow \infty$, we see that
\[
\excost([[(\sigma^n)_{n \ge 1}]]) \ge \excost(\vsigma_\op).
\]
Thus, indeed, $\vsigma_\op$ is the optimal tame Bayesian strategy; any competing tame Bayesian strategy has an expected cost at least that of $\vsigma_\op$.

Next, we compare the cost of $\vsigma_\op$ with that of another tame strategy $\vsigma = [[(\sigma_n)_{n \ge 1}]]$ conditioned on $a_\tru = a$ for a given $a\in [-a_\mx, a_\mx]$.

We will prove the following assertion:
\begin{equation}\label{eq: tsdct 12}
\begin{split}
    &\parbox{20em}{Given $\varepsilon>0$ there exists $\delta$, depending on the \textsc{boilerplate constants} and the constant $C_\tame^{\vsigma}$, such that if}\\
    &\qquad\qquad\excost(\vsigma) \le \excost(\vsigma_\op)+\delta,\\
    &\text{then, for every} \; a\in [-a_\mx, +a_\mx], \text{we have}\\
    & \qquad\qquad |\excost(\vsigma,a) - \excost(\vsigma_\op,a)| < \varepsilon.
\end{split}
\end{equation}
To prove \eqref{eq: tsdct 12}, we recall that $\vsigma_\op = [[(\tsig_n)_{n\ge 1}]]$ with $\tsig_n$ the allegedly optimal strategy associated to the partition associated to $\sigma_n$. Then
\begin{flalign}
    & \excost(\vsigma) = \lim_{n\rightarrow \infty} \excost(\sigma_n)&\label{eq: tsdct 13}\\
    & \excost(\vsigma_\op) = \lim_{n \rightarrow\infty} \excost(\tsig_n)\label{eq: tsdct 14}\\
    &\excost(\vsigma,a) = \lim_{n\rightarrow \infty} \excost(\sigma_n,a)\label{eq: tsdct 15}\\
    &\excost(\vsigma_\op,a)=\lim_{n\rightarrow\infty} \excost(\tsig_n,a).\label{eq: tsdct 16}
\end{flalign}
If $\excost(\vsigma) \le \excost(\vsigma_\op) +  \delta$, then by \eqref{eq: tsdct 13}, \eqref{eq: tsdct 14} we have 
\[
\excost(\sigma_n) \le \excost(\tsig_n) + 2\delta\;\text{for large enough}\; n.
\]
Also, for large enough $n$, the partition of $[0,T]$ associated to $\sigma_n$, $\tsig_n$ has mesh less than $2\delta$. It therefore follows from Theorem \ref{thm: bayesian strategies} that
\[
|\excost(\sigma_n,a) - \excost(\tsig_n,a)| \le \frac{\varepsilon}{2}
\]
for all $a \in [-a_\mx, + a_\mx]$ and all large enough $n$.

From \eqref{eq: tsdct 15}, \eqref{eq: tsdct 16} we now see that
\[
|\excost(\vsigma,a) - \excost(\vsigma_\op,a)| \le \varepsilon
\]
for all $ a \in [-a_\mx, + a_\mx]$, completing the proof of \eqref{eq: tsdct 12}.

\section{Not-Necessarily-Tame Strategies}

In this section we drop the restriction to tame strategies.

Let $\vsigma_1, \vsigma_2, \vsigma_3, \dots$ be tame strategies in the sense of Section \ref{sec: ctns time}. We do \emph{not} assume the $\vsigma_n$ have a tame constant independent of $n$.

We say that the sequence $(\vsigma_n)_{n\ge 1}$ is \emph{Cauchy} if
\[
\lim_{m,n\rightarrow \infty} \E_a \Big[ \int_0^T \{ |q^{\vsigma_n,a}(t) - q^{\vsigma_m,a}(t)|^2 + |u^{\vsigma_n,a}(t) - u^{\vsigma_m,a}(t)|^2\}\ dt\Big] =0,
\]
uniformly for $a$ in any bounded subset of $\R$. Two Cauchy sequences $(\vsigma_n)_{n\ge 1}$ and $(\vsigma_n^\#)_{n \ge 1}$ will be called \emph{equivalent} if
\[
\lim_{n \rightarrow \infty} \E_a \Big[ \int_0^T \{ |q^{\vsigma_n,a}(t) - q^{\vsigma^\#_n,a}(t)|^2 +|u^{\vsigma_n,a}(t) - u^{\vsigma^\#_n,a}(t)|^2\}\ dt\Big]=0
\]
for each $a \in \R$.

To a Cauchy sequence $(\vsigma_n)_{n \ge 1}$ as above, we associate the trajectories $q^a(t)$, $u^a(t)$ for which we have
\[
\lim_{n\rightarrow \infty} \E_a \Big[ \int_0^T \{ | q^{\vsigma_n,a}(t) - q^a(t)|^2 + |u^{\vsigma_n,a}(t) - u^a(t)|^2\} \ dt \Big] = 0.
\]
Two equivalent Cauchy sequences yield the same $q^a$ and $u^a$. We define a \emph{strategy} to be an equivalence class of Cauchy sequences under the above equivalence relation. We denote strategies by $\vvsigma = [[(\vsigma_n)_{n \ge 1}]]$, and we write $q^{\vvsigma, a}(t)$, $u^{\vvsigma, a}(t)$, respectively, to denote the above functions $q^a(t)$, $u^a(t)$. We define
\begin{equation*}
\begin{split}
\excost(\vvsigma,a) &= \lim_{n\rightarrow \infty}\excost(\sigma_n,a)\\
& = \E_a \Big[ \int_0^T\{(q^{\vvsigma,a}(t))^2 +  (u^{\vvsigma,a}(t))^2\}\ dt\Big].
\end{split}
\end{equation*}
These limits converge uniformly for $a$ in any bounded subset of $\R$. Note that every tame strategy $\vsigma$ may also be regarded as a strategy as defined just above, namely the equivalence class associated with the constant sequence $\vsigma, \vsigma, \vsigma, \dots$.

Suppose we are given a prior probability distribution $d\prior$ on the interval $[-a_\mx, a_\mx]$. We define
\[
\excost(\vvsigma, d\prior) = \int_{-a_\mx}^{a_\mx} \excost(\vvsigma,a) \ d\prior(a).
\]
If $\vvsigma = [[(\vsigma_n)_{n \ge 1}]]$, then
\begin{align*}
\excost(\vvsigma, d\prior) &= \int_{-a_\mx}^{a_\mx} \lim_{n \rightarrow \infty} \excost(\vsigma_n,a)\ d\prior(a)\\
& = \lim_{n\rightarrow \infty} \int_{-a_\mx}^{a_\mx} \excost(\vsigma_n,a) \ d\prior(a)\\
& = \lim_{n \rightarrow\infty}\excost(\vsigma_n, d\prior);
\end{align*}
the interchange of limit and integral is justified by the uniform convergence noted above.

Now let $\vsigma_\bayes(d\prior)$ be the optimal Bayesian strategy for $d\prior$, given in Section \ref{sec: ctns time}. For any tame strategy $\vsigma$, we have seen that
\[
\excost(\vsigma, d\prior) \ge \excost(\vsigma_\bayes(d\prior),d\prior).
\]
In particular, if $\vvsigma = [[(\vsigma_n)_{n \ge 1}]]$, then
\[
\excost(\vsigma_n,d\prior) \ge \excost(\vsigma_\bayes(d\prior),d\prior),
\]
hence
\begin{align*}
\excost(\vvsigma, d\prior) &= \lim_{n\rightarrow\infty}\excost(\vsigma_n,d\prior)\\
& \ge \excost(\vsigma_\bayes(d\prior),d\prior).
\end{align*}
So we see that the strategy $\vsigma_\bayes(d\prior)$ has expected cost less than or equal to that of any competing strategy $\vvsigma$.

From now on, we drop the arrows from our notation. When we mention a strategy, we will make clear whether it is a general strategy, a tame strategy, or a tame strategy associated to a partition of $[0,T]$.

Combining the results of this Chapter with Theorem \ref{thm: bayesian strategies}, we deduce Theorems \ref{thm: nintro 1} and \ref{thm: nintro 2} from the introduction.

\chapter{Agnostic Control}\label{chap: 5}

Throughout this chapter the random variable $a_\tru \in [-a_\mx, a_\mx]$ is unknown and we \emph{do not} assume that we have a prior belief about $a_\tru$.

\section{Mixed Strategies}\label{sec: mixed}

Let
\[
\phi: (\xi_1, \xi_2, \xi_3, \dots) \mapsto (\xi_1, \xi_3, \xi_5, \dots)
\]
be the map from $\{0,1\}^\N \rightarrow \{0,1\}^\N$ that erases every other bit.

Fix a partition
\begin{equation}\label{eq: con 1}
0 = t_0 < t_1< \dots < t_N = T,
\end{equation}
and let $\sigma = (\sigma_{\tnu})_{0 \le \nu < N}$ be a tame strategy associated to the partition \eqref{eq: con 1}. Since $\sigma$ is tame, we have
\begin{equation}\label{eq: con 2}
    |\usig(\tnu)| \le C_\tame^\sigma \cdot [|\qsig(\tnu)| + 1] \;\text{for each}\; \nu.
\end{equation}
We can pass from $\sigma$ to the morally equivalent strategy 
\[
\sigma^\# = (\sigma^\#_{\tnu})_{0\le\nu < N}
\]
by setting
\[
\sigma_{\tnu}^\#(q_1, \dots, q_\nu, \vxi) = \sigma_{\tnu}(q_1, \dots, q_\nu, \phi(\vxi)).
\]
Thus, $\sigma^\#$ does precisely what $\sigma$ does, except that whereas $\sigma$ makes use of the bits $\xi_1, \xi_2, \xi_3,\dots$, $\sigma^\#$ makes use only of the bits $\xi_1, \xi_3, \xi_5,\dots$.

Now let $\sigma^0 = (\sigma_{\tnu}^0)_{0 \le \nu < N}$ and $\sigma^1 = (\sigma_{\tnu}^1)_{0 \le \nu < N}$ be two tame strategies, both associated to the partition \eqref{eq: con 1}, and let $\theta \in [0,1]$ be given. We define a \emph{mixed strategy} $\sigma^\theta$, as follows. First, we pass from $\sigma^0$ and $\sigma^1$ to the strategies $\sigma^{0\#}$, $\sigma^{1\#}$ as above. These strategies make use of the bits $\xi_1, \xi_3, \xi_5, \dots$ but ignore the bits $\xi_2, \xi_4, \xi_6,\dots$. We regard $\xi_2, \xi_4, \xi_6,\dots$ as the binary digits of a random variable $Y$ taking values in $[0,1]$. If $Y \le \theta$ then we play the strategy $\sigma^{1\#}$ at all times $\tnu$. If instead $Y>\theta$, then we play the strategy $\sigma^{0\#}$ at all times $\tnu$.

Evidently,
\[
\excost (\sigma^{0\#},a) = \excost(\sigma^0,a)
\]
and
\[
\excost(\sigma^{1\#},a) = \excost(\sigma^1,a)
\]
for all $a \in [-a_\mx, a_\mx]$; and
\begin{equation}\label{eq: con 3}
    \excost(\sigma^\theta,a) = \theta \excost(\sigma^1,a) + (1-\theta)\excost(\sigma^0,a)
\end{equation}
for all $a \in [-a_\mx, a_\mx]$, since $Y \le \theta$ with probability $\theta$. Note that $\sigma^\theta$ is a tame strategy, with
\begin{equation}\label{eq: con 4}
    C_\tame^{\sigma^\theta} \le \max\{ C_\tame^{\sigma^0}, C_\tame^{\sigma^1}\}.
\end{equation}
We have defined the intermediate strategy $\sigma^\theta$ when $\sigma^0$ and $\sigma^1$ are tame strategies associated to the same partition \eqref{eq: con 1} of $[0,T]$.

We next extend our definition to tame strategies with decisions in continuous time. Fix a sequence $\pi_1, \pi_2, \dots$ of partitions of $[0,T]$, with $\text{mesh}(\pi_i) \rightarrow 0$ as $i \rightarrow \infty$. Let $\theta \in [0,1]$ be given. Let $\sigma_1^0, \sigma_2^0,\sigma_3^0, \dots$ and $\sigma_1^1, \sigma_2^1, \sigma_3^1,\dots$ be tame strategies, where, for each $i$, the strategies $\sigma_i^0$ and $\sigma_i^1$ are associated to the partition $\pi_i$ of $[0,T].$

Suppose that $(\sigma_i^0)_{i=1,2,\dots}$ and $(\sigma_i^1)_{i=1,2,\dots}$ are Cauchy sequences, in the sense of Section \ref{sec: ctns time}. Thus, $\vsigma^0 = [[(\sigma_i^0)_{i\ge 1}]]$ and $\vsigma^1 = [[(\sigma_i^1)_{i\ge 1}]]$ are tame strategies in the sense of that section. For each $i$, we pass from $\sigma_i^0$, $\sigma_i^1$ to the mixed strategy $\sigma_i^\theta$ associated to the partition $\pi_i$ of $[0,T]$. Then $\sigma_1^\theta$, $\sigma_2^\theta, \dots$ is again a Cauchy sequence in the sense of Section \ref{sec: ctns time}. We write $\vsigma^\theta$ to denote the tame strategy $[[(\sigma_i^\theta)_{i\ge 1}]]$.
Then we have
\[
\excost(\vsigma^\theta,a) = \theta \excost(\vsigma^1, a) + (1-\theta)\excost(\vsigma^0,a)
\]
for all $ a \in [-a_\mx, a_\mx]$, and
\[
C_\tame^{\sigma_i^\theta} \le \max\{ C_\tame^{\sigma_i^0} , C_\tame^{\sigma_i^1}\},
\]
as follows easily from \eqref{eq: con 3} and \eqref{eq: con 4}.

Note that we have restricted attention to tame strategies $[[(\sigma_i^0)_{i\ge 1}]]$ and $[[(\sigma_i^1)_{i\ge 1}]]$ in which, for each $i$, $\sigma_i^0$ and $\sigma_i^1$ are associated to the same partition of $[0,T]$. It would be natural to dispense with this restriction, but for our purposes that won't be necessary.

\section{Efficient Strategies are Bayesian}
In this section we deal with tame strategies in the sense of Section \ref{sec: ctns time} of Chapter \ref{chap: 4}.

Suppose we are given a class of strategies, which we call the \textsc{legal strategies}.

Assume that given two \textsc{legal strategies} $\sigma^0$ and $\sigma^1$, and given $\theta \in [0,1]$, there exists a \textsc{legal strategy} $\sigma^\theta$ for which we have
\begin{equation}\label{eq: esab 1}
    \excost(\sigma^\theta,a) = (1-\theta)\excost(\sigma^0,a) + \theta \excost(\sigma^1,a)
\end{equation}
for all $a \in [-a_\mx, a_\mx]$.

For example, suppose we fix a constant $\hat{C}$ and a sequence of partitions $(\pi_i)_{i\ge 1}$ of $[0,T]$, with $\text{mesh}(\pi_i)\rightarrow 0$ as $i \rightarrow \infty$.

Then the class of all tame strategies $[[(\sigma_i)_{i\ge 1}]]$ with $\sigma_i$ associated to $\pi_i$ and $C_\tame^{\sigma_i} \le \hat{C}$ satisfies \eqref{eq: esab 1}, thanks to our discussion of mixed strategies in Section \ref{sec: mixed}.

Fix a finite set $A \subset [-a_\mx, a_\mx]$, and let $\varepsilon \ge 0$ be given. (Note that we allow $\varepsilon=0$.)

A \textsc{legal strategy} $\sigma$ will be said to be \emph{efficient with tolerance $\varepsilon$} if there does not exist another \textsc{legal strategy} $\sigma'$ such that
\begin{equation}\label{eq: esab 2}
\excost(\sigma', a) < \excost(\sigma, a) - \varepsilon
\end{equation}
for all $ a \in A$. This notion depends on the set $A$ and the class of \textsc{legal strategies}.

In this section we use a simple convexity argument to prove the following result.

\begin{lem}[Efficient Strategies are Bayesian]\label{lem: esab}
    Fix $A, \varepsilon$ and a class of \textsc{legal strategies} as above, and let $\hat{\sigma}$ be a \textsc{legal strategy}. Suppose $\hat{\sigma}$ is efficient with tolerance $\varepsilon$. Then there exists a prior probability distribution $(p(a))_{a \in A}$ such that for all other \textsc{legal strategies} $\sigma'$ we have
    \begin{equation}\label{eq: esab 3}
        \sum_{a \in A} p(a) \excost(\hat{\sigma},a) \le \sum_{a \in A} p(a) \excost(\sigma',a) + \varepsilon.
    \end{equation}
\end{lem}
\begin{proof}
    For any strategy $\sigma$, define the \emph{cost vector} $\vexcost(\sigma)$ to be the vector $(\excost(\sigma,a))_{a \in A} \in \R^A$. Thanks to \eqref{eq: esab 1}, the set $\cK$ of all cost vectors of legal strategies is convex. Define another convex set $\cK_- \subset \R^A$ to consist of all vectors $(v_a)_{a \in A}$ such that $v_a < \excost(\hat{\sigma}, a) - \varepsilon$ for all $a \in A$. Because $\hat{\sigma}$ is efficient with tolerance $\varepsilon$, the convex sets $\cK$ and $\cK_-$ are disjoint. Hence there exists a nonzero linear functional $\lambda: \R^A \rightarrow \R$ such that $\lambda(v) \le \lambda(v^*)$ whenever $v \in \cK_-$ and $v^* \in \cK$. The functional $\lambda$ has the form
    \[
    \lambda((v_a)_{a \in A}) = \sum_{a \in A} p(a) v_a,
    \]
    with at least one nonzero coefficient $p(a_0)$. By definition of $\cK_-$, $\cK$, and $\lambda$, the following holds.
\begin{equation}\label{eq: esab 4}
\begin{split}
    &\parbox{20em}{Let $\sigma'$ be a \textsc{legal strategy}, and let $(v_a)_{a \in A}$ satisfy}\\
    &v_a < \excost(\hat{\sigma},a)-\varepsilon \text{ for all } a \in A. \\
    &\text{Then}\\
    &\sum_{a \in A} p(a) v_a \le \sum_{a \in A} p(a) \excost(\sigma',a).
\end{split}
\end{equation}
We claim that the $p(a)$ are all nonnegative. Indeed, suppose $p(\hat{a}) < 0$ for some $\hat{a}\in A$. We take $\sigma' = \hat{\sigma}$, $v_a = \excost(\hat{\sigma}, a) - \varepsilon - 1$ for $ a\in A \backslash \{\hat{a}\}$, and $v_{\hat{a}} = - \cV$ for some large positive $\cV$. If $\cV$ is large enough, then the above $\sigma'$, $(v_a)_{a \in A}$ violate \eqref{eq: esab 4}. So, as claimed, the $p(a)$ are all nonnegative.

Since also the $p(a)$ are not all zero, we may multiply the $p(a)$ by a positive normalizing constant to preserve \eqref{eq: esab 4} and achieve also 
\begin{equation}\label{eq: esab 5}
    \sum_{a \in A} p(a) = 1.
\end{equation}
Thus, $(p(a))_{a \in A}$ is a probability distribution.

Now let $\delta>0$, and let $v_a = \excost(\hat{\sigma},a) - \varepsilon - \delta$ for $a \in A$. Thanks to \eqref{eq: esab 4}, \eqref{eq: esab 5} we have
\[
\sum_{a \in A} p(a) \excost(\hat{\sigma}, a) - \varepsilon - \delta \le \sum_{a \in A} p(a) \excost(\sigma', a)
\]
for every legal strategy $\sigma'$. 

Since $\delta >0$ may be taken arbitrarily small, inequality \eqref{eq: esab 3} follows, completing the proof of the Lemma.
\end{proof}

\section{Regret}\label{sec: regret}

We fix continuous functions $\rho_0, \rho_1: \R \rightarrow \R$. We suppose that
\[
|\rho_0(a) | \le \tilde{C}\;\text{and}\; 0 < \tilde{c} < \rho_1(a) < \tilde{C}\;\text{for}\; a \in [-a_\mx, a_\mx].
\]
For any strategy $\sigma$ and any $a \in\R$, we define
\[
\regret(\sigma, a) = \rho_0(a) + \rho_1(a) \cdot \excost(\sigma, a).
\]
We add the above $\tilde{c}$ and $\tilde{C}$ to our list of \textsc{boilerplate constants}. As usual, $c$, $C$, $C'$, etc.\ denote constants depending only on the \textsc{boilerplate constants}. These symbols may denote different constants in different occurrences. The above notion of regret includes as special cases our earlier notions of additive, multiplicative, and hybrid regret.

\section{The Main Lemma on Agnostic Control}
Suppose our PDE Assumption holds (with the same constants $K$, $m_0$, $C^\op_\tame$) for every prior probability distribution on a given interval $[-a_\mx, a_\mx]$. Under that assumption (see Section \ref{sec: pde}), we prove the following result.

\begin{lem}\label{lem: agnostic control}
    Let $\varepsilon>0$, and let $A \subset[-a_\emx, a_\emx]$ be finite. Then there exist a subset $A_0 \subset A$, a probability measure $\mu$, and a strategy $\tilde{\sigma}$, with the following properties.
    \begin{enumerate}[label={\emph{(\arabic*)}}]
        \item The measure $\mu$ is concentrated on $A_0$.
        \item $\tilde{\sigma}$ is the optimal Bayesian strategy for the prior $\mu$.
        \item For $a \in A$ and $a_0 \in A_0$, we have
        \[
        \regret(\tilde{\sigma}, a) \le \regret(\tilde{\sigma}, a_0) + \varepsilon.
        \]
        In particular,
        \item
        $|\regret(\tilde{\sigma}, a_0) - \regret(\tilde{\sigma}, a_0')| \le \varepsilon$ for $a_0, a_0' \in A_0$.
    \end{enumerate}
\end{lem}

For the proof of the above lemma, we first fix a class of strategies that we call OK. For $i = 1,2,3,\dots,$ let $\sigma_i$ be a tame strategy arising from the partition $[0,T]\cap 2^{-i}T\Z$ of the time interval $[0,T]$. Suppose that the $\sigma_i$ form a Cauchy sequence in the sense of Section \ref{sec: ctns time}, and that the tame constants $C_\tame^{\sigma_i}$ are all less than or equal to the constant $C_\tame^\op$ in our PDE Assumption (see Section \ref{sec: pde}, inequality \eqref{eq: pde 6}). Then the strategy $\vsigma = [[(\sigma_i)_{i = 1,2,\dots}]]$ will be called \emph{OK}.

We make two crucial observations regarding OK strategies:
    \begin{enumerate}[label={(\Roman*)}]
        \item For any prior $\mu$ on $[-a_\mx, a_\mx]$, the optimal Bayesian strategy $\tilde{\sigma}$ is OK.\label{eq: ac 5}
        \item If $\sigma,\sigma'$ are OK strategies, then so is the mixed strategy that plays strategy $\sigma$ with probability $\theta$ and strategy $\sigma'$ with probability $(1-\theta)$ (for $0 \le \theta \le 1$).\label{eq: ac 6}
    \end{enumerate}
If $A$ is any finite subset of $[-a_\mx, a_\mx]$ and $\sigma$ is any strategy, we write $\mr(\sigma, A)$ to denote the quantity $\max\{ \regret(\sigma, a): a \in A\}$. For any strategy $\sigma$ and any prior $\mu$ on a finite set $A$ we write
\[
\excost(\sigma,\mu) = \sum_{a \in A} \excost(\sigma, a) \mu(a).
\]
We now begin the proof of Lemma \ref{lem: agnostic control}.

\begin{proof}
    We proceed by induction on $\#A$, the number of elements of $A$.

    \underline{In the base case}, $\#A = 1$, i.e., $A = \{a_0\}$ for some $a_0 \in [-a_\mx, a_\mx]$. We take $A_0 = A$, $\mu = \text{point mass at }a_0$, $\tilde{\sigma}= \text{optimal known-}a\text{ strategy for }a=a_0$. The conclusions of the lemma are obvious.
    
    \underline{For the induction step}, we fix $k \ge 2$ and assume the 

    \textsc{induction hypothesis}: Our lemma holds whenever $\#A < k$.

    We fix $A$ with $\#A = k$, and prove the Lemma for $A$.

    Let $\varepsilon>0$ be given. We pick $\varepsilon_0, \varepsilon_1, \dots, \varepsilon_7>0$, with $\varepsilon_0 = \varepsilon$; and with $\varepsilon_{i+1}$ small enough, depending on $\varepsilon_0,\dots, \varepsilon_i$ and the \textsc{boilerplate constants}.

    Let $\mr_* = \inf\{ \mr(\sigma,A): \sigma \; \text{any OK strategy}\}$ and let $\sigma_*$ be an OK strategy such that
    \[
    \mr(\sigma_*,A) \le \mr_* + \varepsilon_7.
    \]
    For any other OK strategy $\sigma'$, we have
    \begin{equation}\label{eq: ac 7}
        \mr(\sigma_*,A) \le \mr(\sigma' , A) + \varepsilon_7.
    \end{equation}
    If some OK strategy $\sigma'$ satisfied
    \[
    \excost(\sigma', a) \le \excost(\sigma_*,a) - C\varepsilon_7
    \]
    for all $ a \in A$ and a large enough constant $C$, then $\sigma'$ would violate \eqref{eq: ac 7}. Therefore, $\sigma_*$ is $C\varepsilon_7$-efficient on $A$ for the class of OK strategies. Thanks to observation \ref{eq: ac 6} and Lemma \ref{lem: esab}, there exists a probability measure $\mu$ on $A$, such that
    \begin{equation}\label{eq: ac 8}
    \excost(\sigma_*,\mu) \le \excost(\sigma', \mu) + \varepsilon_6
    \end{equation}
    for any OK strategy $\sigma'$. In particular, \eqref{eq: ac 8} holds for the optimal Bayesian strategy for $\mu$, denoted $\tilde{\sigma}$. (Here we use observation \ref{eq: ac 5}.) It therefore follows from Theorem \ref{thm: bayesian strategies} that
    \[
    |\excost(\sigma_*,a) - \excost(\tilde{\sigma},a)| \le \varepsilon_5\;\text{for all}\; a \in A.
    \]
    Together with \eqref{eq: ac 7}, this shows that
    \begin{equation}\label{eq: ac 9}
        \mr(\tilde{\sigma},A) \le \mr(\sigma',A) + C \varepsilon_5
    \end{equation}
    for all OK strategies $\sigma'$. It may happen that
    \begin{equation}\label{eq: ac 10}
        \regret(\tilde{\sigma},a) \ge \mr (\tilde{\sigma},A) - \varepsilon_3\;\text{for all}\; a \in A.
    \end{equation}
    In that case, the conclusions of our lemma hold for $\tilde{\sigma}$, $\mu$ and $A_0 = A$. Hence, we may assume that \eqref{eq: ac 10} is false. Let
    \begin{equation}\label{eq: ac 11}
        A_0 = \{a \in A : \mr(\tilde{\sigma}, A) - \varepsilon_3 \le \regret(\tilde{\sigma},a) \le \mr(\tilde{\sigma}, a)\}.
    \end{equation}
    Thus,
    \begin{equation}\label{eq: ac 12}
    \mr(\tilde{\sigma},A) - \varepsilon_3 \le \regret(\tilde{\sigma},a) \le \mr(\tilde{\sigma}, A)\;\text{for}\; a\in A_0
    \end{equation}
    and
    \begin{equation}\label{eq: ac 13}
\regret(\tilde{\sigma}, a) < \mr(\tilde{\sigma}, A) - \varepsilon_3\;\text{for}\; a \in A\backslash A_0.
    \end{equation}
Since \eqref{eq: ac 10} is false, we have $\#A_0 < \#A$, so our \textsc{inductive hypothesis} applies, i.e., our lemma holds for $A_0$.

Thus, there exist a subset $A_{00}\subset A_0$, a probability measure $\mu_0$, and a strategy $\tilde{\sigma}_0$, with the following properties.
\begin{flalign}
&\mu_0\;\text{is concentrated on } A_{00}.\label{eq: ac 14}&\\
& \tilde{\sigma}_0 \;\text{is the optimal Bayesian strategy for the prior } \mu_0.\label{eq: ac 15}\\
& \regret(\tilde{\sigma}_0,a) \le \regret(\tilde{\sigma}_0,a_0)+ \varepsilon_7\;\text{for } a \in A_0, a_0 \in A_{00}.\label{eq: ac 16}
\end{flalign}
In particular,
\begin{equation}\label{eq: ac 17}
|\regret(\tilde{\sigma}_0,a_0) - \regret(\tilde{\sigma}_0,a_0') | \le \varepsilon_7\;\text{for}\; a_0, a_0' \in A_{00}.
\end{equation}
From \eqref{eq: ac 16}, \eqref{eq: ac 17} we see that
\begin{equation}\label{eq: ac 18}
    \mr(\tilde{\sigma}_0,A_0) - \varepsilon_6 \le \regret(\tilde{\sigma}_0, a_0 ) \le \mr(\tilde{\sigma}_0,A_0)\;\text{for } a_0 \in A_{00}.
\end{equation}
Our plan is to prove that the conclusions of Lemma \ref{lem: agnostic control} for $A$ hold for the set $A_{00}$, the measure $\mu_0$, and the strategy $\tilde{\sigma}_0$; that will complete our induction on $\#A$ and prove Lemma \ref{lem: agnostic control}. To carry out our plan, we first prove that 
\begin{equation}\label{eq: ac 19}
    |\mr(\tilde{\sigma},A) - \mr(\tilde{\sigma}_0,A_0)| \le \varepsilon_2.
\end{equation}
To see \eqref{eq: ac 19}, we recall that
\[
\regret(\sigma, a) = \rho_0(a) + \rho_1(a) \cdot \excost(\sigma, a)
\]
with $c < \rho_1(a) < C$.

For $a_0 \in A_{00} \subset A_0$, estimates \eqref{eq: ac 12} and \eqref{eq: ac 18} therefore imply the inequalities
\begin{equation}\label{eq: ac 20}
\Big[ \frac{\mr(\tilde{\sigma},A) - \rho_0(a)}{\rho_1(a)} \Big] - C \varepsilon_3 \le \excost(\tilde{\sigma},a) \le \Big[ \frac{\mr(\tilde{\sigma},A) - \rho_0(a)}{\rho_1(a)}\Big]
\end{equation}
and
\begin{equation}\label{eq: ac 21}
\begin{split}
    \Big[ \frac{\mr(\tilde{\sigma}_0,A_0) - \rho_0(a)}{\rho_1(a)}\Big] - C \varepsilon_6 \le \excost(&\tilde{\sigma}_0,a)\\ &\le \Big[ \frac{\mr(\tilde{\sigma}_0,A_0)-\rho_0(a)}{\rho_1(a)}\Big].
\end{split}
\end{equation}
Let
\begin{equation}\label{eq: ac 22}
    H_1 = \sum_{a \in A_{00}} \frac{\mu_0(a)}{\rho_1(a)},\qquad H_0 = \sum_{a \in A_{00}} \frac{\mu_0(a) \rho_0(a)}{\rho_1(a)}.
\end{equation}
Since $\mu_0$ is a probability measure concentrated on $A_{00}$, and since $c < \rho_1(a) < C$, we have
\begin{equation}\label{eq: ac 23}
    c' < H_1 < C'.
\end{equation}

Multiplying \eqref{eq: ac 20}, \eqref{eq: ac 21} by $\mu_0(a)$, and summing over $a \in A_{00}$, we obtain the inequalities
\begin{equation}\label{eq: ac 24}
    H_1 \mr(\tilde{\sigma}, A) - H_0 - C \varepsilon_3 \le \excost(\tilde{\sigma},\mu_0) \le H_1 \mr(\tilde{\sigma},A) - H_0
\end{equation}
and
\begin{equation}\label{eq: ac 25}
    H_1 \mr(\tilde{\sigma}_0,A_0) - H_0 - C \varepsilon_6 \le \excost(\tilde{\sigma}_0,\mu_0) \le H_1 \mr(\tilde{\sigma}_0, A_0) - H_0.
\end{equation}
Moreover, since $\tilde{\sigma}_0$ is the optimal Bayesian strategy for the prior $\mu_0$, we have
\begin{equation}\label{eq: ac 26}
    \excost(\tilde{\sigma}_0,\mu_0) \le \excost(\tilde{\sigma},\mu_0).
\end{equation}
From \eqref{eq: ac 24}, \eqref{eq: ac 25}, \eqref{eq: ac 26} we see that
\begin{multline*}
    H_1 \mr(\tilde{\sigma}_0,A_0) - H_0 - C \varepsilon_6 \le \excost(\tilde{\sigma}_0,\mu_0) \le \excost(\tilde{\sigma}, \mu_0)\\ 
     \le H_1 \mr(\tilde{\sigma}, A) - H_0.
\end{multline*}
Thus,
\[
H_1 \mr(\tilde{\sigma}_0, A_0) \le H_1 \mr(\tilde{\sigma}, A) + C \varepsilon_6.
\]
Thanks to \eqref{eq: ac 23}, this tells us that 
\[
\mr(\tilde{\sigma}_0, A_0) \le \mr(\tilde{\sigma}, A) + C \varepsilon_6.
\]
So we've proven half of \eqref{eq: ac 19}. In particular, \eqref{eq: ac 19} holds unless we have
\begin{equation}\label{eq: ac 27}
    \mr(\tilde{\sigma}_0, A_0) \le \mr(\tilde{\sigma}, A) - \varepsilon_2.
\end{equation}
To complete the proof of \eqref{eq: ac 19}, we assume \eqref{eq: ac 27} and derive a contradiction as follows.

Observation \ref{eq: ac 5} tells us that both strategies $\tilde{\sigma}$ and $\tilde{\sigma}_0$ are OK. We form a mixed strategy $\sigma_\mix$ by playing the strategy $\tilde{\sigma}$ with probability $( 1 - \varepsilon_4)$ and the strategy $\tilde{\sigma}_0$ with probability $\varepsilon_4$. Observation \ref{eq: ac 6} tells us that $\sigma_\mix$ is an OK strategy. We will see that \eqref{eq: ac 27} implies that $\sigma_\mix$ outperforms $\tilde{\sigma}$, contradicting \eqref{eq: ac 9}. To see this, we first recall that since $\tilde{\sigma}, \tilde{\sigma}_0$ are tame strategies, we have $\cost(\tilde{\sigma}, a)$, $\cost(\tilde{\sigma}_0,a) \le C$ for all $a \in A$, hence
\begin{equation}\label{eq: ac 28}
    \regret(\tilde{\sigma}_0,a) \le \regret(\tilde{\sigma},a) + C \;\text{for any}\; a \in A.
\end{equation}
Now suppose $a \in A_0$. Then \eqref{eq: ac 27} yields the inequalities
\begin{align*}
    \regret(\sigma_\mix, a) =& (1-\varepsilon_4) \regret(\tilde{\sigma},a) + \varepsilon_4 \regret(\tilde{\sigma}_0,a)\\
    \le & (1-\varepsilon_4) \mr(\tilde{\sigma},A) + \varepsilon_4 \mr(\tilde{\sigma}_0,A_0)\\
    \le & ( 1- \varepsilon_4) \mr(\tilde{\sigma},A) + \varepsilon_4[\mr(\tilde{\sigma},A) - \varepsilon_2]\\
    = & \mr(\tilde{\sigma}, A) - \varepsilon_4 \varepsilon_2.
\end{align*}
On the other hand, for $ a \in A \backslash A_0$, inequalities \eqref{eq: ac 13} and \eqref{eq: ac 28} imply that 
\begin{align*}
    \regret(\sigma_\mix,a) =& (1-\varepsilon_4) \regret(\tilde{\sigma},a) + \varepsilon_4 \regret(\tilde{\sigma}_0,a)\\
    \le& (1-\varepsilon_4) [\mr(\tilde{\sigma}, A) -\varepsilon_3]+ \varepsilon_4 [\mr(\tilde{\sigma},A)+C]\\
    = & \mr(\tilde{\sigma}, A) + C\varepsilon_4 - (1-\varepsilon_4)\varepsilon_3 \le \mr(\tilde{\sigma},A) - \frac{1}{2}\varepsilon_3
\end{align*}
(since $\varepsilon_4 \ll \varepsilon_3 \ll 1)$. Thus, for all $a \in A$, we have
\begin{align*}
\regret(\sigma_\mix, a) &\le \mr(\tilde{\sigma}, A) - \min \{\frac{1}{2}\varepsilon_3, \varepsilon_4\varepsilon_2\} \\
& = \mr(\tilde{\sigma},A) - \varepsilon_2 \varepsilon_4.
\end{align*}
In other words,
\[
\mr(\sigma_\mix,A) \le \mr(\tilde{\sigma}, A) - \varepsilon_4\varepsilon_2.
\]
As promised, this contradicts \eqref{eq: ac 9}, completing the proof of \eqref{eq: ac 19}.

Returning to \eqref{eq: ac 24} and \eqref{eq: ac 25}, we now see that
\begin{equation}\label{eq: ac 29}
    |\excost(\tilde{\sigma},\mu_0) - \excost(\tilde{\sigma}_0, \mu_0)| \le C \varepsilon_3,
\end{equation}
thanks to \eqref{eq: ac 19} and \eqref{eq: ac 23}.

Since $\tilde{\sigma}_0$ is the optimal Bayesian strategy for $\mu_0$, and since $\tilde{\sigma}$ is tame with tame constant at most $C$, \eqref{eq: ac 29} and Theorem \ref{thm: bayesian strategies} together imply that
\begin{equation}\label{eq: ac 30}
    |\excost(\tilde{\sigma}, a) - \excost(\tilde{\sigma}_0,a)|\le \varepsilon_1 \;\text{for all}\; a \in A.
\end{equation}
We are ready to show that $A_{00}$, $\mu_0$, $\tilde{\sigma}_0$ satisfy the conclusions of Lemma \ref{lem: agnostic control} for $A$. Indeed, we know that $\mu_0$ is a probability measure concentrated on $A_{00}$, and that $\tilde{\sigma}_0$ is the optimal Bayesian strategy for the prior $\mu_0$. It remains only to show that
\begin{equation}\label{eq: ac 31}
    \regret(\tilde{\sigma}_0,a) \le \regret(\tilde{\sigma}_0,a_0) + \varepsilon\;\text{for any} \; a_0 \in A_{00}, a \in A.
\end{equation}
However, \eqref{eq: ac 30} yields
\begin{equation}\label{eq: ac 32}
\regret(\tilde{\sigma}_0,a) \le \regret(\tilde{\sigma},a) + C \varepsilon_1 \le \mr(\tilde{\sigma}, A) + C \varepsilon_1\;\text{for } a \in A,
\end{equation}
while \eqref{eq: ac 18} and \eqref{eq: ac 19} yield
\begin{equation}\label{eq: ac 33}
    \regret(\tilde{\sigma}_0,a_0) \ge \mr(\tilde{\sigma}_0, A_0) - \varepsilon_6 \ge \mr(\tilde{\sigma}, A) - C \varepsilon_2
\end{equation}
for $a_0 \in A_{00}$.

The desired estimate \eqref{eq: ac 31} is immediate from \eqref{eq: ac 32} and \eqref{eq: ac 33}. Thus, as promised, the conclusions of our lemma hold, with $A_{00}$, $\mu_0$, $\tilde{\sigma}_0$ in place of $A_0$, $\mu$, $\tilde{\sigma}$. Our induction on $\# A$ is complete, and Lemma \ref{lem: agnostic control} is proven.
\end{proof}

\section{An Interval of Allowed Parameters}

The previous section produced nearly optimal agnostic strategies when the parameter $a$ is known to belong to a finite set. In this section we pass to the case in which $a$ is known merely to belong to a given interval $[-a_\mx,+a_\mx]$. We continue to suppose our PDE Assumption (from Section \ref{sec: pde}) holds for every prior probability distribution on the interval $[-a_\mx, a_\mx]$. Under this assumption, we will prove the following result.

\begin{thm}\label{thm: 5}
    There exists a Bayesian prior probability measure $\mu_\infty$ supported on a subset $A_\infty \subset [-a_\emx, +a_\emx]$, for which the optimal Bayesian strategy $\sigma_\infty$ satisfies
    \begin{itemize}
        \item[\emph{(A)}] The function $a \mapsto \regret(\sigma_\infty,a)$ is constant on $A_\infty$, and
        \item[\emph{(B)}] The function $[-a_\emx, a_\emx] \ni a \mapsto \regret(\sigma_\infty,a)$ is maximized on $A_\infty$.
    \end{itemize}
\end{thm}
\begin{proof}
    Let $A_1, A_2, A_3, \dots$ be a sequence of sets of the form
    \begin{flalign}
        &A_N = [-a_\mx, a_\mx] \cap 2^{-m_N} \Z, \text{ with } m_N \rightarrow \infty\text{ as } N\rightarrow \infty; \text{ and let}\label{eq: iap 1}&\\
&\varepsilon_1,\varepsilon_2,\varepsilon_3,\dots \text{ be a sequence of positive numbers tending to zero.}\label{eq: iap 2}
    \end{flalign}
Applying Lemma \ref{lem: agnostic control} to each $A_N$, we obtain a probability measure $\mu_N$, concentrated on a subset $A_N^0 \subset A_N$, such that the optimal Bayesian strategy $\sigma_N$ for the prior $\mu_N$ satisfies
\begin{flalign}
    & \mr_N - \varepsilon_N \le \regret(\sigma_N, a^0)\le \mr_N\;\text{for all } a^0 \in A_N^0, \text{ where} \label{eq: iap 2.5}&\\
    & \mr_N = \max\{\regret(\sigma_N,a): a \in A_N\}.\label{eq: iap 3}
\end{flalign}
    Passing to a subsequence, we may assume that the $\mu_N$ converge weakly to a probability measure $\mu_\infty$ on $[-a_\mx, +a_\mx]$. Again passing to a subsequence, we may assume that the $\mr_N$ converge to a limit $\mr_\infty$ as $N\rightarrow \infty$. (Here we use the fact that the $\mr_N$ are bounded, thanks to Lemma \ref{lem: rare events}.)

    Let $\sigma_\infty$ be the optimal Bayesian strategy for the prior $\mu_\infty$. After again passing to a subsequence, we will show that
    \begin{equation}\label{eq: iap 4}
        \regret(\sigma_N,a) \rightarrow \regret(\sigma_\infty,a)\;\text{as}\; N \rightarrow \infty, 
    \end{equation}
    uniformly for $a \in [-a_\mx, + a_\mx]$. The proof of \eqref{eq: iap 4} is the main step in our argument. 

    Let us recall how $\regret(\sigma_N,a)$ and $\regret(\sigma_\infty,a)$ are defined. Starting from the prior $\mu_N$, we form the functions
    \[
    \bar{a}_N(\zo, \zt) = \frac{\int_{-a_\mx}^{a_\mx} a \exp\Big( - \frac{a^2}{2} \zt + a \zo\Big)\ d \mu_N(a)}{\int_{-a_\mx}^{a_\mx} \exp \Big( - \frac{a^2}{2} \zt + a\zo\Big)\ d \mu_N(a)},
    \]
    and similarly define $\bar{a}_\infty(\zo,\zt)$.

    Using $\bar{a}_N$ in place of $\bar{a}$ in equation \eqref{eq: pde 1}, we then obtain a PDE solution $S_N(q,t,\zo,\zt) \in C_{loc}^{2,1}(\R\times[0,T]\times \R\times [0,\infty)),$ satisfying the conditions given in Section \ref{sec: pde}. Thanks to the estimates on $\partial^\alpha S_N$ ($|\alpha| \le 3)$ in that section, we may again pass to a subsequence, and assume that
    \begin{equation}\label{eq: iap 5}
        \partial^\alpha S_N \rightarrow \partial^\alpha S_\infty\;\text{as}\; N \rightarrow \infty \;\text{for}\; |\alpha|\le 2,
    \end{equation}
    uniformly on compact subsets of $\R \times [0,T] \times \R \times [0,\infty)$. Here, $S_\infty \in C_{loc}^{2,1}$ satisfies all the estimates given in Section \ref{sec: pde}. Since the $\mu_N$ converge weakly to $\mu_\infty$, it follows that $\bar{a}_N (\zo,\zt)\rightarrow \bar{a}_\infty(\zo,\zt)$ as $N \rightarrow \infty$ for each $(\zo,\zt) \in \R \times [0,\infty)$. Together with \eqref{eq: iap 5} and the PDE satisfied by the $S_N$, this proves that $S_\infty$ satisfies the PDE \eqref{eq: pde 1} for the Bayesian prior $\mu_\infty$.

    Now define
\[
u_N(q,t,\zo,\zt) = - \frac{1}{2}\partial_q S_N(q,t,\zo, \zt)
\]
and
\[
u_\infty(q,t,\zo,\zt) = - \frac{1}{2}\partial_q S_\infty(q,t,\zo,\zt)
\]
for $(q,t,\zo,\zt) \in \R \times [0,T] \times \R \times [0, \infty)$. From \eqref{eq: iap 5} we have
\begin{equation}\label{eq: iap 6}
    u_N \rightarrow u_\infty\;\text{as}\; N \rightarrow \infty,
\end{equation}
uniformly on compact subsets of $\R \times [0,T] \times \R \times [0,\infty)$.

For each $k \ge 1$, we introduce the partition $\pi_k$ of $[0,T]$ given by
\begin{equation}\label{eq: iap 7}
0 = t_0^k < t_1^k < \dots < t_k^k = T, \;\text{with}\; t_\nu^k = \frac{\nu}{k}T.
\end{equation}
Let $\sigma(N,k)$ be the \textsc{allegedly optimal strategy} for the Bayesian prior $\mu_N$ and the partition $\pi_k$. Thus,
\begin{equation}\label{eq: iap 8}
    u^{\sigma(N,k)}(\tnu^k) = u_N(q^{\sigma(N,k)}(\tnu^k), \tnu^k, \zo^{\sigma(N,k)}(\tnu^k), \zt^{\sigma(N,k)}(\tnu^k))\;\text{for}\; 0 \le \nu < k.
\end{equation}
Similarly, let $\sigma(\infty,k)$ be the \textsc{allegedly optimal strategy} for the Bayesian prior $\mu_\infty$ and the partition $\pi_k$. Thus,
\begin{equation}\label{eq: iap 9}
\begin{split}
    u^{\sigma(\infty, k)}(&\tnu^k) \\ &= u_\infty(q^{\sigma(\infty,k)}(\tnu^k),\tnu^k, \zo^{\sigma(\infty,k)}(\tnu^k),\zt^{\sigma(\infty,k)}(\tnu^k))\;\text{for}\; 0 \le \nu < k.
\end{split}
\end{equation}
By definition,
\begin{equation}\label{eq: iap 10}
    \excost(\sigma_N,a) = \lim_{k\rightarrow \infty} \excost(\sigma(N,k),a)
\end{equation}
and
\begin{equation}\label{eq: iap 11}
    \excost(\sigma_\infty,a) = \lim_{k\rightarrow \infty} \excost(\sigma(\infty,k),a),
\end{equation}
for $a \in [-a_\mx, a_\mx]$. Finally,
\begin{equation}\label{eq: iap 12}
    \regret(\sigma_N,a) = \rho_0(a) + \rho_1(a) \excost(\sigma_N,a)
\end{equation}
and
\begin{equation}\label{eq: iap 13}
    \regret(\sigma_\infty, a) = \rho_0(a) + \rho_1(a) \excost(\sigma_\infty,a),
\end{equation}
\begin{sloppypar}
    \noindent for $a \in [-a_\mx, a_\mx]$. This concludes our review of the definition of $\regret(\sigma_N,a)$ and $\regret(\sigma_\infty,a)$.
\end{sloppypar}

For large enough $N$ and $k$, we will apply Lemma \ref{lem: stability} of Section \ref{sec: stability} to the Bayesian prior $\mu_\infty$, the \textsc{allegedly optimal strategy} $\sigma(\infty,k)$, and the alternative strategy $\sigma(N,k)$. Thus, $u_\infty$ will play the r\^{o}le of $u_\op$ in Section \ref{sec: stability}, while $u^{\sigma(N,k)}$, given by \eqref{eq: iap 8}, will play the r\^{o}le of $\usig$ in Section \ref{sec: stability}. The r\^{o}le of the quantity $\discrep^\sigma$ in Section \ref{sec: stability} will therefore be played by
\begin{equation}\label{eq: iap 14}
\begin{split}
    \discrep(N,k,\tnu^k) = u_N&(q^{\sigma(N,k)}(\tnu^k),\tnu^k, \zo^{\sigma(N,k)}(\tnu^k),\zt^{\sigma(N,k)}(\tnu^k)) \\&- u_\infty( q^{\sigma(N,k)}(\tnu^k),\tnu^k, \zo^{\sigma(N,k)}(\tnu^k), \zt^{\sigma(N,k)}(\tnu^k)).
\end{split}
\end{equation}
To apply Lemma \ref{lem: stability}, we must estimate
\begin{equation}\label{eq: iap 15}
    \E_a\Big[ \sum_{0\le\nu < k} |\discrep(N,k,\tnu^k)|^2 \Delta \tnu^k \Big]\;\text{for}\; a \in [-a_\mx, a_\mx],
\end{equation}
with $\Delta \tnu^k = t_{\nu+1}^k - \tnu^k = T/k$ (see \eqref{eq: iap 7}). To estimate the quantity in \eqref{eq: iap 15}, we recall that
\[
|u_N(q,t,\zo,\zt)|, |u_\infty(q,t,\zo,\zt)| \le C [|q| + 1],
\]
and, consequently,
\begin{equation}\label{eq: iap 16}
    |\discrep(N,k,\tnu^k)| \le C [|q^{\sigma(N,k)}(\tnu^*)| + 1].
\end{equation}
For $Q \ge C$, define the events
\[
\bad(N,k,Q) := \big\{ \max_{0 \le \nu < k} \{ |q^{\sigma(N,k)}(\tnu^k)| + |\zo^{\sigma(N,k)}(\tnu^k)| + |\zt^{\sigma(N,k)}(\tnu^k)|\} > Q\big\}
\]
and
\[
\good(N,k,Q) := \big\{ \max_{0 \le \nu < k} \{ |q^{\sigma(N,k)}(\tnu^k)| + |\zo^{\sigma(N,k)}(\tnu^k)| + |\zt^{\sigma(N,k)}(\tnu^k)|\} \le Q\big\}.
\]
We take $k \ge C$ for a large $C$, so that our partition of $[0,T]$ is fine enough to allow us to apply Lemma \ref{lem: rare events}. Lemma \ref{lem: rare events} then tells us that
\[
\E_a \big[ \max_{0 \le \nu < k} \{ |q^{\sigma(N,k)}(\tnu^k)| + 1\}^2 \cdot \mathbbm{1}_{\bad(N,k,Q)} \big] \le C Q^{-1}
\]
for $ a \in [-a_\mx, a_\mx]$. Hence, by \eqref{eq: iap 16},
\begin{equation}\label{eq: iap 17}
        \E_a \Big[ \sum_{0 \le \nu < k} (\discrep(N,k,\tnu^k))^2 \Delta \tnu^k \cdot \mathbbm{1}_{\bad(N,k,Q)}\Big] \le C Q^{-1}
\end{equation}
for $ a \in [-a_\mx, a_\mx]$. On the other hand, let $\delta >0$ be given, and suppose
\[
N \ge N_{\min}(\delta, Q)\;\text{for a large enough}\; N_{\min}(\delta, Q).
\]
Then, by comparing \eqref{eq: iap 6} with \eqref{eq: iap 14}, we see that
\[
| \discrep(N,k,\tnu^k)| \le \delta \;\text{for all}\; \nu, 
\]
provided $\good(N,k,Q)$ occurs. Therefore, 
\[
\E_a \Big[ \sum_{0 \le \nu < k} (\discrep(N,k,\tnu^k))^2 \Delta \tnu^k \cdot \mathbbm{1}_{\good(N,k,Q)}\Big] \le C \delta
\]
for all $ a \in [-a_\mx, a_\mx]$. Together with \eqref{eq: iap 17}, this implies that
\[
\E_a \Big[ \sum_{0 \le \nu < k} (\discrep(N,k,\tnu^k))^2 \Delta \tnu^k\Big] \le C \delta + C Q^{-1}
\]
for all $ a\in [-a_\mx, a_\mx]$ provided $N \ge N_{\min}(\delta, Q)$ and $k\ge C$. Taking $Q = \delta^{-1}$, we have
\begin{equation}\label{eq: iap 18}
    \E_a \Big[ \sum_{0 \le \nu < k} (\discrep(N,k,\tnu^k))^2 \Delta \tnu^k\Big] \le C \delta
\end{equation}
for $ a\in [-a_\mx, a_\mx]$, $N \ge N_{\min}'(\delta)$, and $k \ge C$. Also, recalling \eqref{eq: iap 7}, we see that
\begin{equation}\label{eq: iap 19}
    \Delta t_\mx^k \equiv \max_{0 \le \nu < k} (t_{\nu+1}^k - \tnu^k) \le \frac{C}{k} < \delta\;\text{provided}\; k > \frac{C}{\delta}.
\end{equation}
Now let $\varepsilon>0$ be given, and let $\delta$ be small enough, depending on $\varepsilon$. Our results \eqref{eq: iap 18} and \eqref{eq: iap 19} are the hypotheses of Lemma \ref{lem: stability}, with $\mu_\infty$ in place of $d \prior$, and with $\sigma(N,k)$ in place of $\sigma$. Applying that lemma, we learn that
\[
|\excost(\sigma(N,k),a) - \excost(\sigma(\infty,k),a)| \le \varepsilon,
\]
all $a \in [-a_\mx, a_\mx]$, for $k \ge k_{\min}(\varepsilon)$ and $N \ge N''_{\min}(\varepsilon)$. Passing to the limit as $k\rightarrow \infty$ for fixed $N$, and recalling \eqref{eq: iap 10} and \eqref{eq: iap 11}, we see that
\[
|\excost(\sigma_N,a) - \excost(\sigma_\infty,a)| \le \varepsilon
\]
for $N \ge N''(\varepsilon)$ and for all $ a\in[-a_\mx, a_\mx]$. Since $\varepsilon >0$ is arbitrary, we conclude that 
\[
\excost(\sigma_N,a) \rightarrow \excost(\sigma_\infty,a)\;\text{as}\; N \rightarrow \infty,
\]
uniformly for $ a \in [-a_\mx, a_\mx]$. Thanks to \eqref{eq: iap 12} and \eqref{eq: iap 13}, this in turn implies that
\[
\regret(\sigma_N,a) \rightarrow \regret(\sigma_\infty,a)\;\text{as}\; N \rightarrow \infty,
\]
uniformly for $a \in [-a_\mx, a_\mx]$. So, at last, we have proven \eqref{eq: iap 4}.

Note that the functions $\regret(\sigma_N,a)$, $\regret(\sigma_\infty,a)$ are continuous on $[-a_\mx, a_\mx]$. Thanks to \eqref{eq: iap 4}, they have a common modulus of continuity, i.e.,
\begin{equation}\label{eq: iap 20}
    | \regret(\sigma_N,a_1) - \regret(\sigma_N, a_2)| \le \omega(|a_1 - a_2|)
\end{equation}
for $a_1, a_2 \in [-a_\mx, a_\mx]$ and for all $N\ge 1$, for a function $\omega(t)$ satisfying
\begin{equation}\label{eq: iap 21}
    \omega(t) \rightarrow 0\;\text{as}\; t \rightarrow 0^+.
\end{equation}

We have defined the probability measure $\mu_\infty$ and its optimal Bayesian strategy $\sigma_\infty$. To complete the proof of our theorem, we must define a set $A_\infty
\subset [-a_\mx, a_\mx]$ and prove that
\begin{itemize}
    \item $\mu_\infty$ is supported on $A_\infty$,
    \item $\regret(\sigma_\infty,a)$ is constant on $A_\infty$, and
    \item $\regret(\sigma_\infty,a)$ is maximized on $A_\infty$ over all $a \in [-a_\mx, a_\mx]$.
\end{itemize}
We define $A_\infty$ to consist of all $a \in[-a_\mx, a_\mx]$ such that for all $\eta > 0$ and all $N_* \ge 1$ there exists $a^0 \in A_N^0 \cap (a-\eta, a+\eta)$ for some $N > N_*$. (Recall $A_N^0$ from the defining conditions for the $\mu_N$.)

Let us check that $\mu_\infty$ is supported in $A_\infty$. Thus, let $a \in [-a_\mx, a_\mx]\backslash A_\infty$. Then for some open interval $I \ni a$ and some $N_* \ge 1$, we have $A_N^0 \cap I = \emptyset$ for $N >N_*$. Since $\mu_N$ is supported in $A_N^0$, we have $\mu_N(I) = 0$. Since the probability measures $\mu_N$ converge weakly to $\mu_\infty$ as $N \rightarrow \infty$, it follows that $\mu_\infty(I) = 0$. So $ a \notin \text{support}(\mu_\infty)$, completing the proof that $\mu_\infty$ is supported in $A_\infty$. 

Next, suppose $a^0 \in A_\infty$. Then there exist sequences $N_\nu \rightarrow \infty$ and $a_\nu \rightarrow a^0$ as $\nu \rightarrow \infty$, with $a_\nu \in A_{N_\nu}^0$. From \eqref{eq: iap 2.5} we have
\begin{equation*}
    \mr_{N_\nu} - \varepsilon_{N_\nu} \le \regret(\sigma_N,a_\nu) \le \mr_{N_\nu},
\end{equation*}
hence, thanks to \eqref{eq: iap 20},
\begin{equation}\label{eq: iap 22}
\mr_{N_\nu} - \varepsilon_{N_\nu} - \omega(|a_\nu - a^0|) \le \regret(\sigma_N, a^0) \le \mr_{N_\nu} + \omega(|a_\nu - a^0|).
\end{equation}
As $\nu \rightarrow \infty$, we have $\mr_{N_\nu} \rightarrow \mr_\infty$, $\varepsilon_{N_\nu} \rightarrow 0$, and $\omega(|a_\nu - a^0|) \rightarrow 0$ thanks to \eqref{eq: iap 21}. Therefore, \eqref{eq: iap 22} implies that
\[
\lim_{N \rightarrow \infty} \regret(\sigma_N, a^0) = \mr_\infty.
\]
Recalling \eqref{eq: iap 4}, we see that
\begin{equation}\label{eq: iap 23}
\regret(\sigma_\infty, a^0) = \mr_\infty\;\text{for all}\; a^0 \in A_\infty.
\end{equation}
On the other hand, let $a \in [-a_\mx, a_\mx]$. From \eqref{eq: iap 1} we obtain a sequence $a_N \in A_N$ ($N \ge 1$) such that $a_N \rightarrow a$ as $N \rightarrow \infty$. Thanks to \eqref{eq: iap 3}, we have
\[
\regret(\sigma_N, a_N) \le \mr_N\;\text{for each}\; N,
\]
hence
\begin{equation}\label{eq: iap 24}
    \regret(\sigma_N,a) \le \mr_N + \omega(|a_N - a|),
\end{equation}
by \eqref{eq: iap 20}. As $N \rightarrow \infty$, we have $\mr_N \rightarrow \mr_\infty$ and $\omega(|a_N - a|) \rightarrow 0$ by \eqref{eq: iap 21}. Therefore, \eqref{eq: iap 24} and \eqref{eq: iap 4} yield the inequality
\begin{equation}\label{eq: iap 25}
    \regret(\sigma_\infty,a) \le \mr_\infty\;\text{for all}\; a \in [-a_\mx, a_\mx].
\end{equation}
From \eqref{eq: iap 23} and \eqref{eq: iap 25} we see that $\regret(\sigma_\infty,a)$ is constant on $A_\infty$, and that the maximum of $\regret(\sigma_\infty,a)$ over all $a \in [-a_\mx, a_\mx]$ is achieved on $A_\infty$. The proof of our theorem is complete.
\end{proof}

Under additional assumptions on the functions $\rho_0(a)$ and $\rho_1(a)$ in Section \ref{sec: regret}, we can easily deduce that the set $A_\infty$ in the above theorem is finite. Indeed, suppose $\rho_0$ and $\rho_1$ are real-analytic on $\R$, and suppose that for all $\varepsilon>0$ we have
\begin{equation}\label{eq: iap 26}
    \rho_0(t) \ge - \exp(\varepsilon t)\;\text{and}\; \rho_1(t) \ge \exp(-\varepsilon t)\;\text{for large positive}\; t.
\end{equation}
Recall that the function $a \mapsto \excost(\sigma_\infty, a)$ is real-analytic on $\R$ and grows exponentially as $a \rightarrow \infty$. (See Theorem \ref{thm: bayesian strategies}.)

Under our assumptions on $\rho_0, \rho_1$, it follows that the function 
\[
a \mapsto \regret(\sigma_\infty,a) = \rho_0(a) + \rho_1(a)\excost(\sigma_\infty,a)
\]
is again real-analytic on $\R$ and exponentially large as $a \rightarrow \infty$.

In particular, $[-a_\mx, a_\mx] \ni a \mapsto \regret(\sigma_\infty,a)$ is a nonconstant real-analytic function. Since $\regret(\sigma_\infty,a)$ is constant on $A_\infty$, it follows that $A_\infty$ is finite, as claimed.

Combining Theorem \ref{thm: 5} with the fact that $A_\infty$ is finite establishes Parts (I), (II), (III) of Theorem \ref{thm: nintro 3} in the introduction.

\bibliographystyle{plain}
\bibliography{ref}

\begin{thebibliography}{10}

\bibitem{abbasi2011regret}
Yasin Abbasi-Yadkori and Csaba Szepesv{\'a}ri.
\newblock Regret bounds for the adaptive control of linear quadratic systems.
\newblock In {\em Proceedings of the 24th Annual Conference on Learning
  Theory}, pages 1--26. JMLR Workshop and Conference Proceedings, 2011.

\bibitem{abeille2017thompson}
Marc Abeille and Alessandro Lazaric.
\newblock Thompson sampling for linear-quadratic control problems.
\newblock In {\em Artificial Intelligence and Statistics}, pages 1246--1254.
  PMLR, 2017.

\bibitem{astrom}
Karl Astr\"{o}m.
\newblock {\em Introduction to Stochastic Control Theory}.
\newblock Academic Press, 1970.

\bibitem{bertsekas2012dynamic}
Dimitri Bertsekas.
\newblock {\em Dynamic programming and optimal control: Volume I}, volume~1.
\newblock Athena scientific, 2012.

\bibitem{Brazy:2009}
D.P. Brazy.
\newblock Group chairman's factual report of investigation.
\newblock {\em National Transportation Safety Board Docket No. SA-532, Exhibit
  No. 12}, 2009.

\bibitem{almostoptimal2023}
Jacob Carruth, Maximilian~F. Eggl, Charles Fefferman, and Clarence~W. Rowley.
\newblock Controlling unknown linear dynamics with almost optimal regret.
\newblock {\em forthcoming}, 2023.

\bibitem{carruth2022controlling}
Jacob Carruth, Maximilian~F. Eggl, Charles Fefferman, Clarence~W. Rowley, and
  Melanie Weber.
\newblock Controlling unknown linear dynamics with bounded multiplicative
  regret.
\newblock {\em Revista Matem{\'a}tica Iberoamericana}, 38(7):2185--2216, 2022.

\bibitem{cesa2006prediction}
Nicolo Cesa-Bianchi and G{\'a}bor Lugosi.
\newblock {\em Prediction, learning, and games}.
\newblock Cambridge university press, 2006.

\bibitem{chen2021black}
Xinyi Chen and Elad Hazan.
\newblock Black-box control for linear dynamical systems.
\newblock In {\em Conference on Learning Theory}, pages 1114--1143. PMLR, 2021.

\bibitem{chen2023regret}
Xinyi Chen, Edgar Minasyan, Jason~D Lee, and Elad Hazan.
\newblock Regret guarantees for online deep control.
\newblock {\em Proceedings of Machine Learning Research vol XX}, 1:34, 2023.

\bibitem{Cohen:2019}
Alon Cohen, Tomer Koren, and Yishay Mansour.
\newblock Learning linear-quadratic regulators efficiently with only $\sqrt{T}$
  regret.
\newblock In Kamalika Chaudhuri and Ruslan Salakhutdinov, editors, {\em
  Proceedings of the 36th International Conference on Machine Learning},
  volume~97 of {\em Proceedings of Machine Learning Research}, pages
  1300--1309. PMLR, 09--15 Jun 2019.

\bibitem{dean2018regret}
Sarah Dean, Horia Mania, Nikolai Matni, Benjamin Recht, and Stephen Tu.
\newblock Regret bounds for robust adaptive control of the linear quadratic
  regulator.
\newblock {\em arXiv preprint arXiv:1805.09388}, 2018.

\bibitem{dean2019safely}
Sarah Dean, Stephen Tu, Nikolai Matni, and Benjamin Recht.
\newblock Safely learning to control the constrained linear quadratic
  regulator.
\newblock In {\em 2019 American Control Conference (ACC)}, pages 5582--5588.
  IEEE, 2019.

\bibitem{duchi2011adaptive}
John Duchi, Elad Hazan, and Yoram Singer.
\newblock Adaptive subgradient methods for online learning and stochastic
  optimization.
\newblock {\em Journal of Machine Learning Research}, 12(7), 2011.

\bibitem{faury2021regret}
Louis Faury, Yoan Russac, Marc Abeille, and Cl{\'e}ment Calauz{\`e}nes.
\newblock Regret bounds for generalized linear bandits under parameter drift.
\newblock {\em arXiv preprint arXiv:2103.05750}, 2021.

\bibitem{fefferman2021optimal}
Charles Fefferman, Bernat~Guill{\'e}n Pegueroles, Clarence~W Rowley, and
  Melanie Weber.
\newblock Optimal control with learning on the fly: a toy problem.
\newblock {\em Revista matem{\'a}tica iberoamericana}, 38(1):175--187, 2021.

\bibitem{feller}
Willliam Feller.
\newblock {\em An Introduction to Probability Theory and Its Applications,
  Volume 2}.
\newblock John Wiley \& Sons, Inc., 1971.

\bibitem{furieri2020learning}
Luca Furieri, Yang Zheng, and Maryam Kamgarpour.
\newblock Learning the globally optimal distributed {LQ} regulator.
\newblock In {\em Learning for Dynamics and Control}, pages 287--297. PMLR,
  2020.

\bibitem{goel2021competitive}
Gautam Goel and Babak Hassibi.
\newblock Competitive control.
\newblock {\em arXiv preprint arXiv:2107.13657}, 2021.

\bibitem{gurevich2022optimal}
Daniel Gurevich, Debdipta Goswami, Charles~L Fefferman, and Clarence~W Rowley.
\newblock Optimal control with learning on the fly: System with unknown drift.
\newblock In {\em Learning for Dynamics and Control Conference}, pages
  870--880. PMLR, 2022.

\bibitem{hazancontrol}
Elad Hazan and Karan Singh.
\newblock Online nonstochastic control.
\newblock {\em arXiv preprint arXiv:2211.09619}, 2022.

\bibitem{jedra2022minimal}
Yassir Jedra and Alexandre Proutiere.
\newblock Minimal expected regret in linear quadratic control.
\newblock In {\em International Conference on Artificial Intelligence and
  Statistics}, pages 10234--10321. PMLR, 2022.

\bibitem{kargin2022thompson}
Taylan Kargin, Sahin Lale, Kamyar Azizzadenesheli, Anima Anandkumar, and Babak
  Hassibi.
\newblock Thompson sampling achieves $\backslash$tilde o ($\backslash$sqrt
  $\{$T$\}$) regret in linear quadratic control.
\newblock {\em arXiv preprint arXiv:2206.08520}, 2022.

\bibitem{kumar2022online}
Raunak Kumar, Sarah Dean, and Robert~D Kleinberg.
\newblock Online convex optimization with unbounded memory.
\newblock {\em arXiv preprint arXiv:2210.09903}, 2022.

\bibitem{malik2019derivative}
Dhruv Malik, Ashwin Pananjady, Kush Bhatia, Koulik Khamaru, Peter Bartlett, and
  Martin Wainwright.
\newblock Derivative-free methods for policy optimization: Guarantees for
  linear quadratic systems.
\newblock In {\em The 22nd International Conference on Artificial Intelligence
  and Statistics}, pages 2916--2925. PMLR, 2019.

\bibitem{mania2019certainty}
Horia Mania, Stephen Tu, and Benjamin Recht.
\newblock Certainty equivalence is efficient for linear quadratic control.
\newblock {\em arXiv preprint arXiv:1902.07826}, 2019.

\bibitem{martin2022safe}
Andrea Martin, Luca Furieri, Florian D{\"o}rfler, John Lygeros, and Giancarlo
  Ferrari-Trecate.
\newblock Safe control with minimal regret.
\newblock In {\em Learning for Dynamics and Control Conference}, pages
  726--738. PMLR, 2022.

\bibitem{minasyan2021online}
Edgar Minasyan, Paula Gradu, Max Simchowitz, and Elad Hazan.
\newblock Online control of unknown time-varying dynamical systems.
\newblock {\em Advances in Neural Information Processing Systems},
  34:15934--15945, 2021.

\bibitem{powell2007approximate}
Warren~B Powell.
\newblock {\em Approximate Dynamic Programming: Solving the curses of
  dimensionality}, volume 703.
\newblock John Wiley \& Sons, 2007.

\bibitem{robbins1952some}
Herbert Robbins.
\newblock Some aspects of the sequential design of experiments.
\newblock {\em Bulletin of the American Mathematical Society}, 58(5):527--535,
  1952.

\bibitem{simchowitz2018learning}
Max Simchowitz, Horia Mania, Stephen Tu, Michael~I Jordan, and Benjamin Recht.
\newblock Learning without mixing: Towards a sharp analysis of linear system
  identification.
\newblock In {\em Conference On Learning Theory}, pages 439--473. PMLR, 2018.

\bibitem{vermorel2005multi}
Joannes Vermorel and Mehryar Mohri.
\newblock Multi-armed bandit algorithms and empirical evaluation.
\newblock In {\em European Conference on Machine Learning}, pages 437--448.
  Springer, 2005.

\bibitem{wagenmaker2020active}
Andrew Wagenmaker and Kevin Jamieson.
\newblock Active learning for identification of linear dynamical systems.
\newblock In {\em Conference on Learning Theory}, pages 3487--3582. PMLR, 2020.

\bibitem{wei2021non}
Chen-Yu Wei and Haipeng Luo.
\newblock Non-stationary reinforcement learning without prior knowledge: An
  optimal black-box approach.
\newblock In {\em Conference on Learning Theory}, pages 4300--4354. PMLR, 2021.

\end{thebibliography}
\end{document}